\documentclass[a4paper]{amsart}
\usepackage{graphics}
\usepackage{amssymb}
\usepackage{latexsym,bm}
\usepackage{amsmath,amsfonts,amssymb,mathrsfs}
\usepackage{amscd}

\newcommand{\Z}{\mathbb Z} \newcommand{\Q}{\mathbb Q}
\newcommand{\fS}{\mathfrak{S}} \newcommand{\C}{\mathbb C}
\newcommand{\fsp}{\mathfrak{sp}}
\newcommand{\U}{\mathbb U}
\newcommand{\HH}{\mathscr{H}}
\newcommand{\A}{\mathscr{A}}
\newcommand{\BS}{\mathfrak S}
\newcommand{\bb}{\mathfrak{B}}
\newcommand{\lsub}[2]{{}_{#2}{#1}}
\newcommand\bi{{\bf i}}
\newcommand\bh{{\bf h}}
\newcommand\bj{{\bf j}}
\newcommand\bk{{\bf b}}
\newcommand\kbk{{\bf k}}

\newcommand\bl{{\bf l}}
\newcommand\bc{{\bf c}}
\newcommand\bu{{\bf u}}

\newcommand\bS{{\bf S}}
\newcommand{\lam}{\lambda}
\newcommand\ft{{\mathfrak{t}}}
\newcommand\fs{{\mathfrak{s}}}

\DeclareMathOperator{\End}{End} 
\DeclareMathOperator{\Delt}{\Delta} \DeclareMathOperator{\Ker}{Ker}
\DeclareMathOperator{\SpM}{SpM}
\DeclareMathOperator{\sign}{sign}
\DeclareMathOperator{\id}{id} \DeclareMathOperator{\image}{Im}
\DeclareMathOperator{\wt}{wt} \DeclareMathOperator{\bwt}{WT}
\DeclareMathOperator{\pr}{pr} \DeclareMathOperator{\ann}{ann}
\DeclareMathOperator{\Std}{Std} \DeclareMathOperator{\diag}{diag}
\DeclareMathOperator{\Hom}{Hom}

\newtheorem{thm}[equation]{Theorem}
\newtheorem{cor}[equation]{Corollary}
\newtheorem{lem}[equation]{Lemma}
\newtheorem{dfn}[equation]{Definition}

\numberwithin{equation}{section}

\begin{document}
\title[Quantized Schur--Weyl duality in type $C$]{BMW algebra, quantized coordinate algebra
and type $C$ Schur--Weyl duality}
\author{Jun Hu}
\address{School of Mathematics and Statistics\\
\newline
\null\,\quad University of Sydney\\
\newline\null\,\quad NSW 2006, Australia\\
\bigskip\medskip}

\email{junhu303@yahoo.com.cn}

\date{}

\begin{abstract} We prove an integral version of the Schur--Weyl duality between the specialized
Birman--Murakami--Wenzl algebra $\bb_n(-q^{2m+1},q)$ and the quantum
algebra associated to the symplectic Lie algebra
$\mathfrak{sp}_{2m}$. In particular, we deduce that this Schur--Weyl
duality holds over arbitrary (commutative) ground rings, which
answers a question of Lehrer and Zhang (\cite{LZ}) in the symplectic
case. As a byproduct, we show that, as a
$\mathbb{Z}[q,q^{-1}]$-algebra, the quantized coordinate algebra
defined by Kashiwara in \cite{Ka2} (which was denoted by
$A_q^{\Z}(g)$ there) is isomorphic to the quantized coordinate
algebra arising from a generalized Faddeev--Reshetikhin--Takhtajan
construction (see \cite{FRT}, \cite{Ha}, \cite{Oe2}).
\end{abstract}

\maketitle

\tableofcontents

\section{Introduction}

Let $m, n\in\mathbb{N}$. Let ${\rm U}(\mathfrak{gl}_m)$ be the
universal enveloping algebra of the general linear Lie algebra
$\mathfrak{gl}_m(\C)$ over $\Q$. Let ${\rm U}_{\Z}(\mathfrak{gl}_m)$ be
the Kostant $\Z$-form (\cite{Ko}) in ${\rm U}(\mathfrak{gl}_m)$. For
any commutative $\Z$-algebra $K$, let ${\rm
U}_{K}(\mathfrak{gl}_m):={\rm
U}_{\Z}(\mathfrak{gl}_m)\otimes_{\Z}K$. The natural left action of
${\rm U}_{K}(\mathfrak{gl}_m)$ on $(K^{m})^{\otimes n}$ commutes
with the right place permutation action of the symmetric group
algebra $K\fS_n$. Let $\varphi_{A}, \psi_{A}$ be the natural
representations
$$ \varphi_{A}: (K\fS_n)^{\rm op}\rightarrow\End_{K}\bigl((K^{m})^{\otimes
n}\bigr),\,\,\,\,\psi_{A}:{\rm
U}_{K}(\mathfrak{gl}_m)\rightarrow\End_{K}\bigl((K^{m})^{\otimes
n}\bigr),
$$ respectively. The well-known type $A$ Schur--Weyl duality (see
\cite{CL}, \cite{dCP}, \cite{Do3}, \cite{Gr}, \cite{Sc}, \cite{W})
says that
\begin{enumerate}
\item[(a)] $\varphi_{A}\bigl(K\fS_n\bigr)=\End_{{\rm U}_{K}(\mathfrak{gl}_m)}\bigl((K^{m})^{\otimes n}\bigr)$,
\item[(b)] $\psi_{A}\bigl({\rm U}_{K}(\mathfrak{gl}_m)\bigr)=\End_{K\fS_n}\bigl((K^{m})^{\otimes n}\bigr)$;
\item[(c)] if $K$ is an infinite field, then  $$ \End_{{\rm
U}_{K}(\mathfrak{gl}_m)}\bigl((K^{m})^{\otimes
n}\bigr)=\End_{KGL_m(K)}\bigl((K^{m})^{\otimes n}\bigr),
$$
and the image of the group algebra $KGL_m(K)$ in $\End_{K}\bigl((K^{m})^{\otimes
n}\bigr)$ also coincides with $\End_{K\fS_n}\bigl((K^{m})^{\otimes
n}\bigr)$;
\item[(d)] if $K$ is a field of characteristic $0$, then there is an irreducible
${\rm U}_{K}(\mathfrak{gl}_m)$-$K\fS_n$-bimodule decomposition $$
(K^{m})^{\otimes n}=\bigoplus_{\substack{\lam=(\lam_1,\lam_2,\cdots)\vdash n\\
\ell(\lam)\leq m}}{\Delt}_{\lam}\otimes S^{\lam},$$ where
${\Delt}_{\lam}$ (resp., $S^{\lam}$) denotes the irreducible left
${\rm U}_{K}(\mathfrak{gl}_m)$-module (resp., irreducible right
$K\fS_n$-module) associated to $\lam$, $\lam\vdash n$ means $\lam$
is a partition of $n$, and $\ell(\lam)$ denotes the largest integer
$i$ such that $\lam_i\neq 0$.
\end{enumerate}\smallskip

There is a quantized version of the above type $A$ Schur--Weyl
duality. Let $q$ be an indeterminate over $\Z$. Let
$\A=\Z[q,q^{-1}]$ be the Laurent polynomial ring in $q$. Let
$\mathbb{U}_{\Q(q)}(\mathfrak{gl}_m)$ be the quantized enveloping
algebra of $\mathfrak{gl}_m$ over $\Q(q)$ (\cite{D1}, \cite{J1},
\cite{J2}), where $q$ is the quantum parameter. Let
$\mathbb{U}_{\A}(\mathfrak{gl}_m)$ be the Lusztig's $\A$-form
(\cite{Lu1}) in $\mathbb{U}_{\Q(q)}(\mathfrak{gl}_m)$. Let
$\HH_{\A}(\fS_n)$ be the Iwahori--Hecke algebra associated to the
symmetric group $\fS_n$, defined over $\A$ and with parameter $q$.
By definition, $\HH_{\A}(\fS_n)$ is generated by
$\widehat{T}_1,\cdots,\widehat{T}_{n-1}$ which satisfy the
well-known braid relations as well as the relation
$(\widehat{T}_i-q)(\widehat{T}_i+q^{-1})=0$, for $
i=1,2,\cdots,n-1$. For any commutative $\A$-algebra $K$, we use
$\zeta$ to denote the natural image of $q$ in $K$, and we define
$\mathbb{U}_{K}(\mathfrak{gl}_m):=\mathbb{U}_{\A}(\mathfrak{gl}_m)\otimes_{\A}K$,
$\HH_{K}(\fS_n):=\HH_{\A}(\fS_n)\otimes_{\A} K$. Then, there is a
left action of $\mathbb{U}_{\Q(q)}(\mathfrak{gl}_m)$ on $\Q(q)^m$
which quantizes the natural representation of $\mathfrak{gl}_m(\C)$.
Via the coproduct, we get an action of
$\mathbb{U}_{\Q(q)}(\mathfrak{gl}_m)$ on $(\Q(q)^m)^{\otimes n}$.
Furthermore, this action actually gives rise to an action of
$\mathbb{U}_{\A}(\mathfrak{gl}_m)$ on $(\A^m)^{\otimes n}$
(\cite{Du}). By base change, we get an action of
$\mathbb{U}_{K}(\mathfrak{gl}_m)$ on $(K^m)^{\otimes n}$ for any
commutative $\A$-algebra $K$. There is also a right action of
$\HH_K(\fS_n)$ on $(K^m)^{\otimes n}$. Let $\varphi_A, \psi_A$ be
the natural representations $$ \varphi_A: (\HH_{K}(\fS_n))^{\rm
op}\rightarrow\End_{K}\bigl((K^{m})^{\otimes n}\bigr),\quad
\psi_A:{\U}_{K}(\mathfrak{gl}_m)\rightarrow\End_{K}\bigl((K^{m})^{\otimes
n}\bigr),
$$ respectively. Then by \cite{BLM}, \cite{Du},
\cite{DPS} and \cite{J2},
\begin{enumerate}
\item[(a')] $\varphi_A\bigl(\HH_{K}(\fS_n)\bigr)=\End_{{\U}_{K}(\mathfrak{gl}_m)}\bigl((K^{m})^{\otimes n}\bigr)$;
\item[(b')] $\psi_A\bigl({\U}_{K}(\mathfrak{gl}_m)\bigr)=\End_{\HH_{K}(\fS_n)}\bigl((K^{m})^{\otimes n}\bigr)$;
\item[(c')] if $K$ is a field of characteristic $0$ and $\zeta$ is not a root of unity in $K$, then there is an irreducible
${\U}_{K}(\mathfrak{gl}_m)$-$\HH_{K}(\fS_n)$-bimodules decomposition
$$
(K^{m})^{\otimes n}=\bigoplus_{\substack{\lam=(\lam_1,\lam_2,\cdots)\vdash n\\
\ell(\lam)\leq m}}{\Delt}_{\lam}\otimes S^{\lam},$$ where
${\Delt}_{\lam}$ (resp., $S^{\lam}$) denotes the irreducible left
${\U}_{K}(\mathfrak{gl}_m)$-module (resp., irreducible right
$\HH_{K}(\fS_n)$-module) associated to $\lam$.
\end{enumerate}\smallskip

\noindent The algebra $\End_{\HH_{K}(\fS_n)}\bigl((K^{m})^{\otimes
n}\bigr)$ is called ``$q$-Schur algebra", which forms an important
class of quasi-hereditary algebra and has been extensively studied
by Dipper--James and many other people. It plays an important role
in the modular representation theory of finite groups of Lie type
(cf. \cite{DJ1}, \cite{DJ2}, \cite{GH}). The significance of the
above results lies in that it provide a bridge between the
representation theory of type $A$ quantum groups and of type $A$
Hecke algebras at an integral level. Note that in the semisimple
case, the above Schur--Weyl duality follows easily from the complete
reducibility. The difficult part lies in the non-semisimple case,
where the surjectivity of $\varphi_A$ was established in \cite{DPS}
by making use of Kazhdan--Lusztig bases of type $A$ Hecke algebra,
while the proof of the surjectivity of $\psi_A$ relies heavily on
the amazing work of \cite{BLM}, where the quantized enveloping
algebra of $\mathfrak{gl}_m$ is realized as certain ``limit" of
$q$-Schur algebras. To the best of our knowledge, there is no
alternative approach for this part.
\smallskip

A natural question arises: how about the Schur--Weyl dualities in
other types? The answer is: there do exist Schur--Weyl dualities in
types $B, C, D$ in semisimple case (for both classical and quantized
versions). However, it is an open question (see \cite[Page80,
Line1]{Ha}, \cite[Abstract]{LZ}) whether or not these
Schur--Weyl dualities hold in an integral or characteristic free
setting (like the type $A$ situation).\smallskip

The purpose of this paper is to give an affirmative answer to the
above open question in the quantized type $C$ case. That is, we
shall prove an integral version of quantized type $C$ Schur--Weyl
duality. Note that there is no counterpart in type $C$ of the work
\cite{BLM} in the literature. It turns out that our approach
provides a new and general framework to prove integral Schur--Weyl
dualities for all classical types. Before stating the main results
in this paper, we first recall the known results for the classical
type $C$ Schur--Weyl duality. Let $K$ be an infinite field. Let $V$
be a $2m$-dimensional $K$-linear space equipped with a skew bilinear
form $(,)$. Let $GSp(V)$ (resp., $Sp(V)$) be the symplectic
similitude group (resp., the symplectic group) on $V$ (\cite{Dt1},
\cite{Gri}). For any integer $i$ with $1\leq i\leq 2m$, set
$i':=2m+1-i$. We fix an ordered basis
$\bigl\{v_1,v_2,\cdots,v_{2m}\bigr\}$ of $V$ such that
$$ (v_i, v_{j})=0=(v_{i'}, v_{j'}),\,\,\,(v_i,
v_{j'})=\delta_{ij}=-(v_{j'}, v_{i}),\quad\forall\,\,1\leq i, j\leq
m. $$ Let $\bb_n(-2m)$ be the specialized Brauer algebra over $K$.
This algebra contains the group algebra $K\BS_n$ as a subalgebra.
There is a right action of $\bb_n(-2m)$ on $V^{\otimes n}$ which
extends the sign permutation action of $\BS_n$. We refer the reader
to \cite{DDH} for definitions of $\bb_n(-2m)$ and its action. Let
$\varphi_C, \psi_C$ be the natural representations
$$ \varphi_C: (\bb_n(-2m))^{\rm op}\rightarrow\End_{K}\bigl(V^{\otimes
n}\bigr),\,\,\,\,\,\, \psi_C:
KGSp(V)\rightarrow\End_{K}\bigl(V^{\otimes n}\bigr),
$$ respectively.

\begin{thm} {\rm (\cite{B}, \cite{B1}, \cite{B2})} \begin{enumerate}
\item[(1)]The natural left action of $GSp(V)$ on $V^{\otimes n}$ commutes
with the right action of $\bb_n(-2m)$. Moreover, if $K=\mathbb{C}$,
then
$$\begin{aligned}
\varphi_C\bigl(\bb_n(-2m)\bigr)&=\End_{\mathbb{C}
GSp(V)}\bigl(V^{\otimes n}\bigr)=\End_{\mathbb{C}
Sp(V)}\bigl(V^{\otimes n}\bigr),\\
\psi_C\bigl(\mathbb{C}GSp(V)\bigr)&=\psi_C\bigl(\mathbb{C}
Sp(V)\bigr)=\End_{\bb_n(-2m)}\bigl(V^{\otimes n}\bigr),\end{aligned}
$$
\item[(2)] if $K=\mathbb{C}$, then there is an irreducible
$\mathbb{C}GSp(V)$--$\bb_n(-2m)$--bimodule decomposition
$$
V^{\otimes n}=\bigoplus_{f=0}^{[n/2]}\bigoplus_{\substack{\lam\vdash n-2f\\
\ell(\lam)\leq m}}\Delta({\lam})\otimes D({\lam^t}),$$ where
$\Delta({\lam})$ (resp., $D({\lam^t})$) denotes the irreducible left
$\mathbb{C}GSp(V)$-module (resp., the irreducible right
$\bb_n(-2m)$-module) corresponding to $\lam$ (resp., corresponding
to $\lam^t$), $\lam^t$ denotes the transpose of $\lam$.
\end{enumerate}
\end{thm}

By the work of \cite{dCP}, \cite{DDH} and \cite{Oe1}, the complex
field $\mathbb{C}$ used in part (1) of the above theorem can be
replaced by arbitrary infinite field. That is, we have a
characteristic free version of type $C$ Schur--Weyl duality in group
case.

\begin{thm} {\rm (\cite{dCP}, \cite{DDH}, \cite{Oe1})} Let $K$ be an arbitrary infinite
field. Then
\begin{enumerate}
\item[(1)] $\psi_C\bigl(K
GSp(V)\bigr)=\End_{\bb_n(-2m)}\bigl(V^{\otimes n}\bigr)$;
\item[(2)] $\varphi_C\bigl(\bb_n(-2m)\bigr)=\End_{KGSp(V)}\bigl(V^{\otimes
n}\bigr) =\End_{KSp(V)}\bigl(V^{\otimes n}\bigr)$.
\end{enumerate}\end{thm}\smallskip

For the quantized type $C$ Schur--Weyl duality, we require $V$ to be
a $2m$ dimensional vector space over $\Q(q)$ equipped with a skew
bilinear form $(,)$. We fix an ordered basis $\{v_i\}_{i=1}^{2m}$ as
before. Let ${\U}_{\Q(q)}(\mathfrak{sp}_{2m})$ be the quantized
enveloping algebra of $\mathfrak{sp}_{2m}(\C)$ over $\Q(q)$, where $q$
is the quantum parameter. Let $\bb_n(-q^{2m+1},q)$ be the
specialized Birman--Murakami--Wenzl algebra (specialized BMW algebra for short)
over $\Q(q)$. There is a right action of $\bb_n(-q^{2m+1},q)$ on
$V^{\otimes n}$ which quantizes the right action of $\bb_n(-2m)$. We
refer the reader to Section 3 for precise definitions of
$\bb_n(-q^{2m+1},q)$ and its action. Let $\varphi_C, \psi_C$ be the
natural representations
$$\begin{aligned} \varphi_C:&\,\,\,
(\bb_n(-q^{2m+1},q))^{\rm op}\rightarrow\End_{\Q(q)}\bigl(V^{\otimes
n}\bigr),\\
\psi_C:&\,\,\,
{\U}_{\Q(q)}(\mathfrak{sp}_{2m})\rightarrow\End_{\Q(q)}\bigl(V^{\otimes
n}\bigr), \end{aligned}$$ respectively.

\begin{thm} {\rm (\cite[10.2]{CP}, \cite{LZ})} \label{ssqc}\begin{enumerate}
\item[(1)] The natural left action of ${\U}_{\Q(q)}(\mathfrak{sp}_{2m})$
 on $V^{\otimes n}$ commutes with the right action of $\bb_n(-q^{2m+1},q)$. Moreover,
 $$\begin{aligned}
\varphi_C\bigl(\bb_n(-q^{2m+1},q)\bigr)&=\End_{{\U}_{\Q(q)}(\mathfrak{sp}_{2m})}\bigl(V^{\otimes
n}\bigr),\\
\psi_C\bigl({\U}_{\Q(q)}(\mathfrak{sp}_{2m})\bigr)&=\End_{\bb_n(-q^{2m+1},q)}\bigl(V^{\otimes
n}\bigr);\end{aligned}
$$
\item[(2)] there is an irreducible
${\U}_{\Q(q)}(\mathfrak{sp}_{2m})$-$\bb_n(-q^{2m+1},q)$-bimodule
decomposition
$$
V^{\otimes n}=\bigoplus_{f=0}^{[n/2]}\bigoplus_{\substack{\lam\vdash n-2f\\
\ell(\lam)\leq m}}\Delta({\lam})\otimes D({\lam^t}),$$ where
$\Delta({\lam})$ (respectively, $D({\lam^{t}})$) denotes the irreducible
left ${\U}_{\Q(q)}(\mathfrak{sp}_{2m})$-module (respectively, the
irreducible right $\bb_n(-q^{2m+1},q)$-module) corresponding to
$\lam$ (resp., corresponding to $\lam^t$).\end{enumerate}
\end{thm}

Let $\mathbb{U}_{\A}(\mathfrak{sp}_{2m})$ be the Lusztig's $\A$-form in
$\mathbb{U}_{\Q(q)}(\mathfrak{sp}_{2m})$. Let $V_{\A}$ be the free
$\A$-module spanned by $\{v_i\}_{i=1}^{2m}$. Note that
$\bb_n(-q^{2m+1},q)$ has a natural $\A$-form
$\bb_n(-q^{2m+1},q)_{\A}$. For any commutative $\A$-algebra $K$, let
$\zeta$ be the natural image of $q$ in $K$, and we define
$\mathbb{U}_{K}(\mathfrak{sp}_{2m}):=\mathbb{U}_{\A}(\mathfrak{sp}_{2m})\otimes_{\A}K$,
$\bb_n(-\zeta^{2m+1},\zeta):=\bb_n(-q^{2m+1},q)_{\A}\otimes_{\A} K$.
The representation $\psi_C$ naturally gives rise to an action of
$\U_{\A}(\mathfrak{sp}_{2m})$ on $V_{\A}^{\otimes n}$ which commutes
with the right action of $\bb_n(-q^{2m+1},q)_{\A}$. By base change,
for any commutative $\A$-algebra $K$, we get an action of
$\U_{K}(\mathfrak{sp}_{2m})$ on $V_{K}^{\otimes n}$ which commutes
with the right action of $\bb_n(-\zeta^{2m+1},\zeta)$;

\medskip
The main results in this paper are the following two
theorems.\smallskip

\begin{thm} \label{mainthm1} For any commutative $\A$-algebra $K$, $$
\psi_C\Bigl(\U_{K}(\mathfrak{sp}_{2m})\Bigr)=\End_{\bb_n(-\zeta^{2m+1},\zeta)}\bigl(V_K^{\otimes
n}\bigr).
$$
\end{thm}\smallskip

\begin{thm} \label{mainthm2} For any commutative $\A$-algebra $K$, $$
\varphi_C\Bigl(\bb_n(-\zeta^{2m+1},\zeta)\Bigr)=\End_{\U_{K}(\mathfrak{sp}_{2m})}\bigl(V_K^{\otimes
n}\bigr).
$$
\end{thm}\smallskip

Note that if we specialize the parameter $q$ to $1_{K}\in K$, then
the BMW algebra $\bb_n(-q^{2m+1},q)$ becomes the specialized Brauer
algebra $\bb_n(-2m)$, and the action of $\bb_n(-q^{2m+1},q)$ on
$n$-tensor space becomes the action of $\bb_n(-2m)$ used in
\cite{DDH}. Applying the above two theorem, we get the following
corollary.

\begin{cor} For any commutative $\Z$-algebra $K$, \smallskip
\begin{enumerate}
\item[(1)] $\psi_C\bigl({\rm
U}_{K}(\mathfrak{sp}_{2m})\bigr)=\End_{\bb_n(-2m)_K}\bigl(V_{K}^{\otimes
n}\bigr)$;
\item[(2)] $\varphi_C\Bigl(\bb_n(-2m)_{K}\Bigr)=\End_{{\rm
U}_{K}(\mathfrak{sp}_{2m})}\bigl(V_K^{\otimes n}\bigr)$.
\end{enumerate}
\end{cor}
\noindent
Note that this corollary can also be deduced from the main result in \cite{DDH} by using the equivalence between the category of rational $Sp_{2m}(K)$-modules and the category of locally finite $U_K(\mathfrak{sp}_{2m})$-modules.
\smallskip

The algebra
$S_K^{sy}(2m,n):=\End_{\bb_n(-\zeta^{2m+1},\zeta)}\bigl(V_K^{\otimes
n}\bigr)$ is called ``symplectic $\zeta$-Schur algebra" by Oehms
(\cite{Oe2}). It is a cellular (in the sense of \cite{GL}) and
quasi-hereditary $K$-algebra. The strategy that we use to prove Theorem \ref{mainthm1}
is to inspect the induced natural
homomorphism $\widetilde{\psi}_C$ from Lusztig's modified quantum
algebra (see \cite{Lu3}) $\dot{\U}_K(\mathfrak{sp}_{2m})$ to the
symplectic $q$-Schur algebra $S_K^{sy}(2m,n)$, and (roughly
speaking) to interpret $\widetilde{\psi}_C$ as the dual of the
natural map from the $n$th homogeneous component of the quantized
coordinate algebra of $\SpM_{2m}(K)$ (symplectic monoid) to the
quantized coordinate algebra of ${\rm Sp}_{2m}(K)$ (symplectic
group). It turns out that the kernel of $\widetilde{\psi}_C$ is
spanned by the canonical basis elements it contains. As a
consequence, we deduce the following result, which is announced in
\cite{Dt2} without proof.

\begin{cor} \label{maincor} For any commutative $\A$-algebra $K$, $S_{K}^{sy}(2m,n)$ is
isomorphic to the generalized $q$-Schur algebra $\lsub{\bS}{K}(\pi)$
defined in \cite{Dt2}, where $\pi$ is the set of dominant weights
occurring in $V^{\otimes n}$. In particular, if specializing $q$ to
$1$, then we recover the symplectic Schur algebra studied in
\cite{Do2} and \cite{Dt1}.
\end{cor}

The strategy that we use to prove Theorem \ref{mainthm2} is
similar to that used in \cite{DDH}. We first prove the equality
under the assumption that $m\geq n$. Then we reduce the case $m<n$
to the case $m=n$ via a commutative diagram. Finally, we convert the
task of proving the equality concerning $\varphi_C$ to a purely type
$C$ quantum algebra representation theorietic problem which involves
no BMW algebras. However, the direct generalization from \cite{DDH}
does not work here. In our quantized case the proof is much more
difficult. We expect that our approach for both
equalities can be applied to prove integral versions of various
other Schur--Howe--Weyl dualities in Lie theory.\smallskip

The paper is organized as follows. In Section 2, we collect some
basic knowledge about the usual and the modified form of the
quantized enveloping algebra of $\mathfrak{sp}_{2m}(\C)$ as well as
their actions on the $n$-tensor space $V^{\otimes n}$. The new
result is Lemma \ref{keylem}, which enables us to reduce the proof
of the equality concerning $\psi_C$ to the proof of an equality
concerning $\widetilde{\psi}_C$. In Section 3, we show that each
finite truncation ${A}_{\A}^{sy}(2m,\leq\!n)$ of the
quantized coordinate algebra ${A}_{\A}^{sy}(2m)$ of $\SpM_{2m}(K)$
is a cellular coalgebra. The two-sided simple
comodule decomposition of the quantized coordinate algebra
$\widetilde{A}_{\Q(q)}^{sy}(2m)$ of ${\rm Sp}_{2m}(K)$ is obtained, which actually
coincides with Peter--Weyl's decomposition proved by Kashiwara
(\cite{Ka2}). In Section 4, after proving that the tensor product $\nabla_K^r(\lam)\otimes\nabla_K(\lam)$
of a cell left comodule and a cell right comodule of ${A}_{K}^{sy}(2m,\leq\!n)$ is actually a co-Weyl module of
the quantum algebra $\U_K(\mathfrak{g}\oplus\mathfrak{g})$ (Lemma \ref{coWeyl2}), we are able to identify the
type $C$ quantized coordinate algebra $A_q^{\Z}(\mathfrak{g})$
defined by Kashiwara in \cite{Ka2} with the quantized coordinate
algebra $\widetilde{A}_{\A}^{sy}(2m)$ arising from generalized FRT
construction. The proof relies on some nice properties of the upper global crystal basis (i.e., the dual canonical basis) of the quantized coordinate algebra $A_q^{\Z}(\mathfrak{g})$ introduced by Kashiwara. Then we give a proof of our first
main result---Theorem \ref{mainthm1}. In Section 5, we give a proof
of our second main result---Theorem \ref{mainthm2} in the case where
$m\geq n$, following a similar idea (but different and
more difficult arguments than) in \cite[Section 3]{DDH}. The case
where $m<n$ is dealt with in Section 6. We reduce the proof of
Theorem \ref{mainthm2} to the proof of the surjectivity of a map
between coinvariants of two tensor spaces and the commutativity of a
certain diagram of maps (Lemma \ref{lm61}). For the former, we use
Lusztig's theory on based modules (\cite{Lu3}) as in \cite[Section
4]{DDH}. The proof of the latter turns out to be quite delicate and
more interesting than in the classical case.
\smallskip

The main results of this paper were announced at the International
Conference ``GL07: Geometry and Lie Theory" (Sydney, July, 2007).
The author thanks the organizers for inviting him to attend and give
talk at this conference. He would like to thank Professor S. Doty, Professor G. Lehrer,
Professor C.W. Curtis, Professor S. Ariki, Professor F. Goodman and the anonymous referee
for their helpful remarks and comments.

\section{Tensor space and modified quantized enveloping algebra of type $C$}

In this section we shall recall some basic facts about the usual and
the modified form of the quantized enveloping algebra of type $C$ as
well as their actions on the $n$-tensor space $V^{\otimes n}$. We
shall show that, for any commutative $\A$-algebra $K$, the image of
the quantum algebra ${\U}_{K}(\mathfrak{sp}_{2m})$ in the
endomorphism algebra of the $n$-tensor space $V_K^{\otimes n}$
coincides with the image of the corresponding modified quantum
algebra $\dot{\U}_{K}(\mathfrak{sp}_{2m})$.\smallskip

Recall the Dynkin diagram of $\fsp_{2m}$ \begin{center}
\scalebox{0.25}[0.25]{\includegraphics{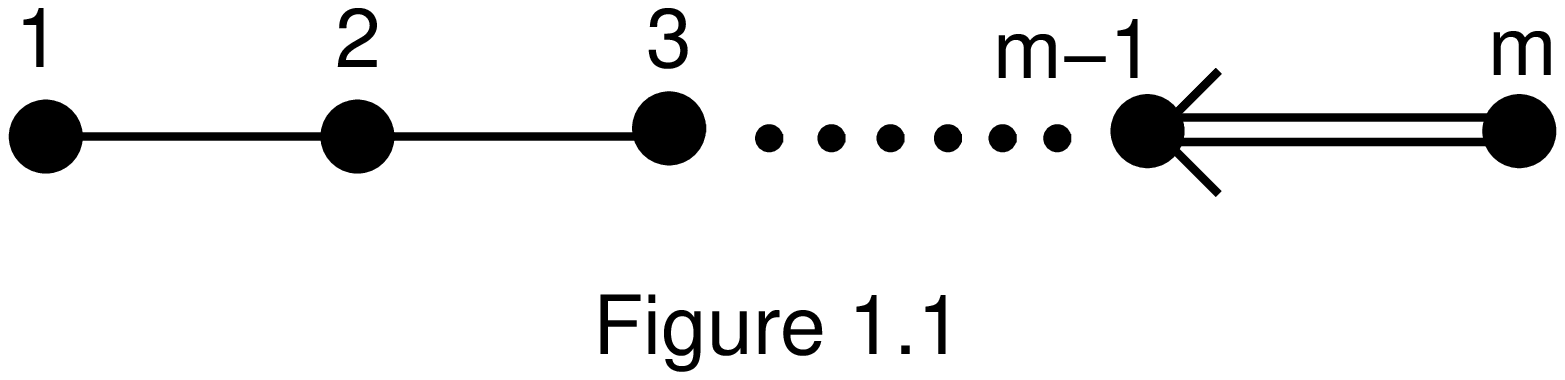}}\qquad ,
\end{center}
where each vertex labeled by $i$ represents a simple root
$\alpha_i$. For each integer $i$ with $1\leq i\leq m$, let
$\alpha^{\vee}_i$ be the corresponding simple co-root. Let $X$ be
the weight lattice of $\mathfrak{sp}_{2m}$. We realize $X$ as a free
$\Z$-module with basis $\varepsilon_1,\cdots,\varepsilon_m$. Then
$\alpha_m=2\varepsilon_m,\,\,\,\,\alpha_i=\varepsilon_i-\varepsilon_{i+1},\,\,i=1,2,\cdots,m-1$.

Throughout this paper, we shall identify the weight $\lam=\lam_1\varepsilon_1+\cdots+\lam_m\varepsilon_m\in X$
with the sequence $(\lam_1,\cdots,\lam_m)$ of integers. We shall also write $\lam=(\lam_1,\cdots,\lam_m)$ and
$|\lam|:=\lam_1+\cdots+\lam_m$. Let $$X^{+}:=\{\lam\in X|\langle\lam,\alpha_i\rangle\geq 0,\forall\,1\leq i\leq m\},$$ i.e., the set of dominant weights. Then $\lam=(\lam_1,\cdots,\lam_m)\in X^{+}$ if and only if $\lam$ is a partition.
Let $(,)$ be the symmetric bilinear form $X$ defined by
$(\varepsilon_i,\varepsilon_j)=\delta_{i,j}$ for any $i,j$. The
Cartan matrix $A=(a_{i,j})$ (where
$a_{i,j}:=2(\alpha_j,\alpha_i)/(\alpha_i,\alpha_i)$)
of $\fsp_{2m}$ is:%
\[
\left(
\begin{array}
[c]{cccccc}%
\text{ \ }2 & -1 &  &  &  & \\
-1 & \text{ \ }2 & -1 &  &  & \\
& -1 & . & . &  & \\
&  & . & . & . & \\
&  &  & . & \text{ \ }2 & -2\\
&  &  &  & -1 & \text{ \ }2
\end{array}
\right)
\]
with rows and columns indexed by $\{1,...,m\}$. Given a fixed
indeterminate
$q$ set%
\begin{gather*}
q_{i}=\left\{
\begin{tabular}
[c]{l}%
$q$,\,\,\, if $i\neq m$\\
$q^{2}$,\, if $i=m$%
\end{tabular}
\right.  \text{,}\\
\lbrack k]_{i}=\frac{q_{i}^{k}-q_{i}^{-k}}{q_{i}-q_{i}^{-1}}\text{
and }[k]_{i}^!=[k]_{i}[k-1]_{i}\cdot\cdot\cdot\lbrack1]_{i}.
\end{gather*}
The quantized enveloping algebra $\U_{\Q(q)}:=\U_{\Q(q)}(\fsp_{2m})$
is the associative unital algebra over $\Q(q)$ generated by
$e_{i},f_{i},k_i,k_i^{-1}$
$i=1,...,m$, subject to the relations:%
\begin{gather*}
k_ik_i^{-1}=k_i^{-1}k_i=1,\quad\,\,
k_i^{\pm 1}k_j^{\pm 1}=k_j^{\pm 1}k_i^{\pm 1},\\
k_ie_{j}k_i^{-1}=q^{a_{i,j}}e_{j},\quad\,\,\,
k_if_{j}k_i^{-1}=q^{-a_{i,j}}f_{j},\\
\lbrack e_{i},f_{i}]=\frac{\widetilde{k}_{i}-\widetilde{k}_{i}^{-1}}{q_{i}-q_{i}^{-1}}
\delta_{i,j}\text{ where }\widetilde{k}_{i}=\left\{
\begin{tabular}
[c]{l}%
$k_i$,\,\,\, if $i\neq m$\\
$k_i^{2}$,\,\, if $i=m$%
\end{tabular}
\right.  ,\\
\text{if }i\neq j\text{ },\,\,\overset{1-a_{i,j}}{\underset{k=0}{\sum}}%
(-1)^{k}e_{i}^{(k)}e_{j}e_{i}^{(1-a_{i,j}-k)}=0,\\
\text{if }i\neq j\text{ },\,\,\overset{1-a_{i,j}}%
{\underset{k=0}{\sum}}(-1)^{k}f_{i}^{(k)}f_{j}f_{i}^{(1-a_{i,j}-k)}=0,
\end{gather*}
where $e_{i}^{(k)}=e_{i}^{k}/[k]_{i}^!$ and $f_{i}^{(k)}=f_{i}^{k}/[k]_{i}^!$.\smallskip

$\U_{\Q(q)}$ is a Hopf algebra with coproduct $\Delta$, counit
$\varepsilon$ and antipode $S$ defined on generators by
$$\begin{aligned} &\Delta(e_i)=e_i\otimes 1+\widetilde{k}_{i}\otimes e_i,\quad
\Delta(f_i)=1\otimes f_i+f_i\otimes
\widetilde{k}_{i}^{-1},\quad \Delta(k_i)=k_i\otimes k_i,\\
& \varepsilon(e_i)=\varepsilon(f_i)=0,\quad \varepsilon(k_i)=1,\\
& S(e_i)=-\widetilde{k}_{i}^{-1}e_i,\quad S(f_i)=-f_i\widetilde{k}_{i},\quad S(k_i)=k_i^{-1}.
\end{aligned}$$

Recall our definition of $V_{\A}$ in the first paragraph below
Theorem \ref{ssqc}. For each $1\leq i\leq 2m$, let $i':=2m+1-i$. The
action of the generators of ${\U}_{\Q(q)}$ on
$V_{\Q(q)}:=V_{\A}\otimes_{\A}{\Q(q)}$ is as follows (cf.
\cite[(4.16)]{Ha})\footnote{Note that our $\widetilde{k}_m=k_m^{2}$
in this paper corresponds to $k_m$ in the notation of \cite{Ha}.}.
$$\begin{aligned}
e_{i}{v}_{j}&:=\begin{cases} {v}_i, &\text{if $j=i+1$,}\\
-{v}_{(i+1)'}, &\text{if $j=i'$,}\\
0, &\text{otherwise;}\end{cases}\,\,
e_{m}{v}_{j}:=\begin{cases} {v}_{m}, &\text{if $j=m'$,}\\
0, &\text{otherwise,}\end{cases}\\
f_{i}{v}_{j}&:=\begin{cases} {v}_{i+1}, &\text{if $j=i$,}\\
-{v}_{i'}, &\text{if $j=(i+1)'$,}\\
0, &\text{otherwise;}\end{cases}\,\,\,\,
f_{m}{v}_{j}:=\begin{cases} {v}_{m'}, &\text{if $j=m$,}\\
0, &\text{otherwise,}\end{cases}\\
k_{i}{v}_{j}&:=\begin{cases} q {v}_j, &\text{if $j=i$ or $j=(i+1)'$,}\\
q^{-1}{v}_{j}, &\text{if $j=i+1$ or $j=i'$,}\\
{v}_{j}, &\text{otherwise,}\end{cases}\cr
k_{m}{v}_{j}&:=\begin{cases} q{v}_j, &\text{if $j=m$,}\\
 q^{-1}{v}_{j}, &\text{if $j=m'$,}\\
{v}_{j}, &\text{otherwise,}\end{cases}\end{aligned}
$$
where $1\leq i<m,\,\, j\in\{1,\cdots,m\}\cup\{m',\cdots,1'\}$. Via
the coproduct, we get an action of ${\U}_{\Q(q)}$ on
$V_{\Q(q)}^{\otimes n}$. Let
$\mathbb{U}_{\A}:=\mathbb{U}_{\A}(\mathfrak{sp}_{2m})$ be the
Lusztig's $\A$-form in $\mathbb{U}_{\Q(q)}(\mathfrak{sp}_{2m})$. As
an $\A$-algebra, $\mathbb{U}_{\A}$ is generated by $$
e_i^{(a)},\,\,f_i^{(a)},\,\,k_i,\,\,k_i^{-1},\,\,a=0,1,2,\cdots,
1\leq i\leq m.
$$

\begin{lem} \label{Alattice} The above action of ${\U}_{\Q(q)}$ on
$V_{\Q(q)}^{\otimes n}$ naturally gives rise to an action of
$\U_{\A}$ on $V_{\A}^{\otimes n}$.
\end{lem}

\begin{proof} It is well-known that $\U_{\A}$ is an $\A$-Hopf
algebra. Hence it suffices to show that $\U_{\A}V_{\A}\subseteq
V_{\A}$. However, this follows from direct verification.
\end{proof}

For any commutative $\A$-algebra $K$, we define
$\mathbb{U}_{K}=\mathbb{U}_{K}(\mathfrak{sp}_{2m}):=\mathbb{U}_{\A}\otimes_{\A}K$.
By base change, we see that there is a representation
$$
{\psi}_C:{\U}_{K}\rightarrow\End_{K}\bigl(V_K^{\otimes n}\bigr).
$$

In Lusztig's book \cite[Part IV]{Lu3}, a ``modified form"
$\dot\U_{\Q(q)}$ of ${\U}_{\Q(q)}$ was introduced. The algebra
$\dot\U_{\Q(q)}$ in general does not have a unit element. But it
does have a family $\{1_\lambda\}_{\lambda\in X}$ of orthogonal
idempotents such that ${\dot{\U}}_{\Q(q)} = \oplus_{\lambda,\mu\in
X} 1_{\lambda} {\dot{\U}}_{\Q(q)} 1_\mu$. In a sense, the family
$\{1_\lambda\}$ serves as a replacement for the identity. Let
$\pi_{\lam,\mu}$ be the canonical projection from $\U_{\Q(q)}$ onto
$1_{\lambda} {\dot{\U}}_{\Q(q)} 1_\mu$ (see \cite[(23.1.1)]{Lu3}).
As a $\Q(q)$-algebra, $\dot\U_{\Q(q)}$ is generated by the elements
$e_i 1_{\lam}$, $f_i 1_{\lam}$ and $1_{\lam}$ with
$i\in\{1,2,\cdots,m\}$ and $\lam\in X$, where the following
relations are satisfied.
$$1_{\lam}1_{\mu}=1_{\mu}1_{\lam}=\delta_{\lam,\mu}1_{\lam},\quad
e_i 1_{\lam}=1_{\lam+\alpha_i} (e_i1_{\lam}),\quad f_i
1_{\lam}=1_{\lam-\alpha_i} (f_i1_{\lam}),$$
$$(e_i1_{\lam-\alpha_j})(f_j 1_{\lam})-(f_j 1_{\lam+\alpha_i})(e_i1_{\lam})
=\delta_{i,j}[\langle\lam,\alpha_i^{\vee}\rangle]_i 1_{\lam},$$
$$\sum_{k=0}^{1-a_{ij}}(-1)^k e_i^{(k)} e_j e_i^{(1-a_{ij}-k)}=
\sum_{k=0}^{1-a_{ij}}(-1)^k f_i^{(k)} f_j f_i^{(1-a_{ij}-k)}=0\quad
\text{if}\quad i\neq j,$$ where
$\langle\lam,\alpha_i^{\vee}\rangle:={2(\lam,\alpha_i)}/{(\alpha_i,\alpha_i)}$,
and the last identity is understood as its canonical image under
$\pi_{\lam,\mu}$ for any $\lam,\mu\in X$.\smallskip

Let $\dot{\U}_{\A}$ be the $\A$-subalgebra of $\dot{\U}_{\Q(q)}$
generated by $e_i^{(k)}1_{\lam}, f_i^{(k)}1_{\lam}$ for
$i=1,2,\cdots,$ $m$, $k=0,1,2,\cdots$, $\lam\in X$. Then by
\cite[(23.2)]{Lu3}, $\dot{\U}_{\A}$ is a free $\A$-module, and in
fact $\dot{\U}_{\A}$ is an $\A$-form of
$\dot{\U}_{\Q(q)}$.\smallskip

Since $V_{\Q(q)}^{\otimes n}$ is a finite dimensional integrable
module over $\U_{\Q(q)}$, it follows that $V_{\Q(q)}^{\otimes n}$
naturally becomes a unital $\dot{\U}_{\Q(q)}$-module in the sense of
\cite[(23.1.4)]{Lu3}. For each $\lam\in X$, we define $p_{\lam}$ to
be the projection operator from $V^{\otimes n}$ onto its
$\lam$-weight space (with respect to the subalgebra generated by $k_1^{\pm 1},\cdots,k_{m-1}^{\pm 1},k_m^{\pm 1}$).

\begin{lem} Let $\widetilde{\psi}_{C}$ be the map $$
1_{\lam}\mapsto p_{\lam},\quad
e_i1_{\lam}\mapsto\psi_C(e_i)p_{\lam},\quad
f_i1_{\lam}\mapsto\psi_C(f_i)p_{\lam},
\,\,i=1,2,\cdots,m,\,\,\lam\in X.$$ Then $\widetilde{\psi}_{C}$ can
be naturally extended to a representation of
$\dot{\U}_{\Q(q)}$ on $V^{\otimes n}$ such that $\widetilde{\psi}_C(P1_{\lam})=\psi_{C}(P)p_{\lam}$ for any
$P\in\U_{\Q(q)}$ and $\lam\in X$.
\end{lem}

\begin{proof} This follows directly from the definition of
 $\dot{\U}_{\Q(q)}$ and the fact that $V_{\Q(q)}^{\otimes n}$ is a
direct sum of its weight spaces.
\end{proof}

By restriction and applying Lemma \ref{Alattice}, we see that
$V_{\A}^{\otimes n}$ naturally becomes an $\dot{\U}_{\A}$-module.
For any commutative $\A$-algebra $K$, we define
$\dot{\U}_{K}=\dot{\U}_{K}(\mathfrak{sp}_{2m}):=\dot{\U}_{\A}\otimes_{\A}K$.
By base change, we get a representation
$$
\widetilde{\psi}_C:\dot{\U}_{K}\rightarrow\End_{K}\bigl(V_K^{\otimes
n}\bigr),
$$

\begin{lem} \label{keylem} With the above notation, for any
commutative $\A$-algebra $K$, $$ {\psi}_C\Bigl({\U}_{K}\Bigr)=
\widetilde{\psi}_C\Bigl(\dot{\U}_{K}\Bigr).
$$
\end{lem}

\begin{proof} It suffices to show that ${\psi}_C\Bigl({\U}_{\A}\Bigr)=
\widetilde{\psi}_C\Bigl(\dot{\U}_{\A}\Bigr)$. \smallskip

Let $X_n$ be the set of weights (with respect to the Cartan part of
$\mathfrak{sp}_{2m}$) in $V^{\otimes n}$. Obviously, $X_n$ is a
finite set. As linear operators on $V^{\otimes n}$, it is easy to
check that for $i=1,2,\cdots,m$ and $a=0,1,2,\cdots$,
$$\begin{aligned}
&\psi_C\bigl(e_i^{(a)}\bigr)=\widetilde{\psi}_{C}\Bigl(\sum_{\lam\in
X_n}e_i^{(a)}1_{\lam}\Bigr),\quad\,
\psi_C\bigl(f_i^{(a)}\bigr)=\widetilde{\psi}_{C}\Bigl(\sum_{\lam\in
X_n}f_i^{(a)}1_{\lam}\Bigr),\\
&\psi_{C}(k_i)=\widetilde{\psi}_{C}\Bigl(\sum_{\lam\in X_n}
q^{\langle\lam,\alpha_i^{\vee}\rangle}1_{\lam}\Bigr).\end{aligned}
$$ As a result, we deduce that
${\psi}_C\Bigl({\U}_{\A}\Bigr)\subseteq\widetilde{\psi}_C\Bigl(\dot{\U}_{\A}\Bigr)$.
It remains to show that
$\widetilde{\psi}_C\Bigl(\dot{\U}_{\A}\Bigr)\subseteq
{\psi}_C\Bigl({\U}_{\A}\Bigr)$.
\smallskip

It suffices to show that
$p_{\lam}=\widetilde{\psi}_C(1_{\lam})\in{\psi}_C\Bigl({\U}_{\A}\Bigr)$
for each $\lam\in X_n$. For each integer $i$ with $1\leq i\leq m$,
we set (following Lusztig)
$$
\left[\begin{matrix}K_i; c\\
t
\end{matrix}\right]=\prod_{s=1}^t\frac{\widetilde{k}_iq_i^{c-s+1}-\widetilde{k}_i^{-1}q_i^{-c+s-1}}{q_i^s-q_i^{-s}},\quad\,\,
\left[\begin{matrix}K_i; c\\
0
\end{matrix}\right]=1,
$$
for $t\geq 1$, $c\in\mathbb{Z}$. By \cite[Lemma 4.4]{Lu0}, we know
that
$$
\left[\begin{matrix}K_i; c\\
t
\end{matrix}\right]\in\U_{\A}.
$$
Let $\lam\in X_n$. We write
$\hat{\lam}_i:=\langle\lam,\alpha_i^{\vee}\rangle$ for each $i$.
Note that for each $i$,
$$
\hat{\lam}_i\in\begin{cases} \bigl\{-n,-n+1,\cdots,-1,0,1,\cdots,n-1,n\bigr\}, &\text{if $i\neq m$,}\\
\bigl\{-2n,-2n+2,\cdots,-2,0,2,\cdots,2n-2,2n\bigr\}, &\text{if
$i=m$.}
\end{cases}
$$
We define $$ \begin{aligned}
&p'_{\lam}:=\biggl(\prod_{i=1}^{m-1}\left[\begin{matrix}K_i; -\hat{\lam}_i-1\\
2n
\end{matrix}\right]
\left[\begin{matrix}K_i; -\hat{\lam}_i+2n\\
2n
\end{matrix}\right]
\biggr)\times\\
&\qquad\qquad\qquad\biggl(\left[\begin{matrix}K_m; -\hat{\lam}_m-1\\
4n
\end{matrix}\right]
\left[\begin{matrix}K_m; -\hat{\lam}_m+4n\\
4n
\end{matrix}\right]
\biggr). \end{aligned}$$

Clearly, $p'_{\lam}\in\U_{\A}$. For each $\mu\in X_n$, we use
$V_{\Q(q)}^{\otimes n}[\mu]$ to denote the $\mu$-weight space of
$V_{\Q(q)}^{\otimes n}$. One can verify directly that $$
\psi_C(p'_{\lam})(x)=\begin{cases}x,
&\text{if $x\in V_{\Q(q)}^{\otimes n}[\lam]$,}\\
0,&\text{if $x\in V_{\Q(q)}^{\otimes n}[\mu]$ with $\mu\neq\lam$.}
\end{cases}
$$
As a result, we deduce that
$p_{\lam}=\widetilde{\psi}_C(1_{\lam})=\psi_C(p'_{\lam})\in{\psi}_C\Bigl({\U}_{\A}\Bigr)$,
as required. This completes the proof of the lemma.\end{proof}
\smallskip

\noindent {\it Remark 2.4.} Lemma \ref{keylem} enables use to reduce
the proof of the equality concerning $\psi_C$ to the proof of the
equality concerning $\widetilde{\psi}_C$. Note that the arguments
used in the proof of Lemma \ref{keylem} actually work in all types.

\bigskip\bigskip

\section{BMW algebras and a generalized FRT construction}

In this section we shall first recall the definitions of specialized
BMW algebras and Oehms's results on a generalized
Faddeev--Reshetikhin--Takhtajan (FRT for short) construction. Then
we shall analyze the structure of each finite truncation
${A}_{\A}^{sy}(2m,\leq\!n)$ of the quantized coordinate
algebra ${A}_{\A}^{sy}(2m)$ of ${\rm SpM}_{2m}$. Using
Oehms's symplectic bideterminant basis for the quantized coordinate
algebra of symplectic monoid scheme $\SpM_{2m}$, we shall conclude
that ${A}_{\A}^{sy}(2m,\leq\!n)$ is a cellular coalgebra.
The two-sided simple comodule decomposition of the quantized
coordinate algebra $\widetilde{A}_{\Q(q)}^{sy}(2m)$ is obtained.

Let $x,r,z$ be three indeterminates over $\mathbb{Z}$. Let $R$ be
the ring $$
R:=\mathbb{Z}[r,r,^{-1},z,x]/\bigl((1-x)z+(r-r^{-1})\bigr).
$$
For simplicity, we shall use the same letters $r,r^{-1},z,x$ to
denote their images in $R$ respectively.

\begin{dfn} {\rm (\cite{BW}, \cite{MW}, \cite{M})}\label{bmw} The Birman--Murakami--Wenzl
algebra (or BMW algebra for short) $\bb_n(r,x,z)$ is a unital
associative $R$-algebra with generators $T_i, E_i$, $1\le i\le n-1$
and relations\begin{enumerate}
\item $T_i-T_i^{-1}=z(1-E_i)$, for $1\le i\le n-1$,
\item $E_i^2=xE_i$, for $1\le i\le n-1$,
\item $T_{i}T_{i+1}T_i=T_{i+1}T_iT_{i+1}$,  for $1\le i\le n-2$,
\item $T_iT_j=T_jT_i$,  for $|i-j|>1$,
\item $E_iE_{i+1}E_i=E_i,\,\,E_{i+1}E_{i}E_{i+1}=E_{i+1}$,\,\, for $1\le i\le n-2$,
\item $T_iT_{i+1}E_{i}=E_{i+1}E_i,\,\,T_{i+1}T_{i}E_{i+1}=E_{i}E_{i+1}$,\,\,for $1\le i\le n-2$,
\item $ E_i T_i=T_iE_i=r^{-1} E_i$,  for $1\le i\le n-1$.
\item $E_{i}T_{i+1}E_{i}=rE_{i},\,\,E_{i+1}T_{i}E_{i+1}=rE_{i+1}$,\,\,for $1\le i\le n-2$.
\end{enumerate}
\end{dfn}

In \cite{MW}, Morton and Wassermann proved that $\bb_n(r,x,z)$ is
isomorphic to the Kauffman's tangle algebra~\cite{Kau} whose
$R$-basis is indexed by Brauer diagrams. As a consequence, they show
that $\bb_n(r,x,z)$ is a free $R$-module with rank $(2n-1)!!$. In
fact, the same is still true if one replaces the ring $R$ by any
commutative $R$-algebra $K$. Note that if we specialize $r$ to $1$
and $z$ to $0$, then $\bb_n(r,x,z)$ will become the usual Brauer
algebra with parameter $x$.

We regard $\A$ as an $R$-algebra by sending $r$ to $-q^{2m+1}$, $z$
to $q-q^{-1}$ and $x$ to $1-\sum_{i=-m}^{m}q^{2i}$. The resulting
$\A$-algebra will be denoted by $\bb_n(-q^{2m+1},q)_{\A}$. We set
$\bb_n(-q^{2m+1},q)=\bb_n(-q^{2m+1},q)_{\A}\otimes_{\A}\Q(q)$, and
we call it {\it the specialized Birman--Murakami--Wenzl algebra} (or
{\it specialized BMW algebra} for short). Note that if we specialize
further $q$ to $1$, then $\bb_n(-q^{2m+1},q)$ will become the
specialized Brauer algebra $\bb_n(-2m)$ used in \cite{DDH} and
\cite{Hu1}.

There is an action of the algebra $\bb_n(-q^{2m+1},q)_{\A}$ on the
$n$-tensor space $V_{\A}^{\otimes n}$ which we now recall. We set $$
(\rho_1,\cdots,\rho_{2m}):=(m,m-1,\cdots,1,-1,\cdots,-m+1,-m),
$$
and $\epsilon_i:=\sign(\rho_i)$. For any $i,j\in\{1,2,\cdots,2m\}$,
we use $E_{i,j}$ to denote the corresponding basis of matrix units
for $\End_{\Q(q)}(V_{\Q(q)})$. Let $``-"$ be the ring involution defined on $\A$ by
$\overline{q^{\pm 1}}=q^{\mp 1}$, $\overline{k}=k$ for any $k\in\Z$. The involution $``-"$ can be
uniquely extended to an involution of $\End_{\A}\bigl(V_{\A}^{\otimes 2}\bigr)$ such that $$
\overline{E_{i,j}\otimes E_{k,l}}=E_{i,j}\otimes E_{k,l},\,\,\overline{rx}=\overline{r}\,\overline{x},
$$
for any integers $1\leq i,j,k,l\leq 2m$, any $r\in\A$ and any
$x\in\End_{\A}\bigl(V_{\A}^{\otimes 2}\bigr)$. Following
\cite[Section 2]{Oe2}, we set $$\begin{aligned} \beta &:=\sum_{1\leq
i\leq 2m}\Bigl(q^2E_{i,i}\otimes E_{i,i}+E_{i,i'}\otimes
E_{i',i}\Bigr)+q\sum_{\substack{1\leq
i,j\leq 2m\\ i\neq j,j'}} E_{i,j}\otimes E_{j,i}+\\
&\qquad\qquad (q^2-1)\sum_{1\leq j<i\leq 2m}\Bigl(E_{i,i}\otimes
E_{j,j}-q^{\rho_i-\rho_j}\epsilon_i\epsilon_j E_{i,j'}\otimes
E_{i',j}\Bigr),\\
\gamma &:=\sum_{1\leq i,j\leq
2m}q^{\rho_i-\rho_j}\epsilon_i\epsilon_j E_{i,j'}\otimes E_{i',j}.
\end{aligned}
$$
We define $\beta':=\overline{q\beta^{-1}}$, $
\gamma':=\overline{\gamma}$. By direct verification, we have that
$$\begin{aligned} \beta'&:=\sum_{1\leq i\leq
2m}\Bigl(qE_{i,i}\otimes E_{i,i}+q^{-1}E_{i,i'}\otimes
E_{i',i}\Bigr)+\sum_{\substack{1\leq
i,j\leq 2m\\ i\neq j,j'}} E_{i,j}\otimes E_{j,i}+\\
&\qquad\qquad (q-q^{-1})\sum_{1\leq i<j\leq 2m}\Bigl(E_{i,i}\otimes
E_{j,j}-q^{\rho_j-\rho_i}\epsilon_i\epsilon_j E_{i,j'}\otimes
E_{i',j}\Bigr),\\
\gamma'&:=\sum_{1\leq i,j\leq
2m}q^{\rho_j-\rho_i}\epsilon_i\epsilon_j E_{i,j'}\otimes E_{i',j}.
\end{aligned}
$$
Note that the operators\footnote{The reason we use the operators $\beta', \gamma'$ is because we want to let
$\bb_n(-q^{2m+1},q)_{\A}$ act on $V_{\A}^{\otimes n}$ from the right hand side instead of from the left hand side.} $\beta', \gamma'$ are related to each other by
the equation
\begin{equation}\label{BG}\beta'-(\beta')^{-1}=(q-q^{-1})(\id_{V^{\otimes
2}}-\gamma').
\end{equation}
For $i=1,2,\cdots,n-1$, we set $$ \beta'_i:=\id_{V^{\otimes
i-1}}\otimes\beta'\otimes\id_{V^{\otimes n-i-1}},\quad
\gamma'_i:=\id_{V^{\otimes i-1}}\otimes\gamma'\otimes\id_{V^{\otimes
n-i-1}}.
$$
By \cite[(10.2.5)]{CP} and \cite[Section 4]{Ha}, the map
$\varphi_{C}$ which sends each $T_i$ to $\beta'_i$ and each $E_i$ to
$\gamma'_i$ for $i=1,2,\cdots,n-1$ can be naturally extended to a
right action of $\bb_n(-q^{2m+1},q)_{\A}$ on $V_{\A}^{\otimes n}$
such that all the statements in Theorem \ref{ssqc} hold.  Note that
if we specialize $q$ to $1$, then $T_i$ degenerates to
$-s_i\in\bb_n(-2m)$ for each $i$ and this action of
$\bb_n(-q^{2m+1},q)$ becomes the action studied in \cite{DDH} of the
specialized Brauer algebra $\bb_n(-2m)$ on $V_{\Z}^{\otimes n}$.
\smallskip

Now we recall what Oehms called ``generalized
Faddeev--Reshetikhin--Takhtajan construction". The basic
references are \cite{FRT}, \cite{Ha} and \cite{Oe2}. We concentrate
on the quantized type $C$ case. Let $u_i, 1\leq i\leq 2m$, (resp.,
$X_{i,j}, 1\leq i,j\leq 2m$) be a basis of $V^{\ast}$ (resp., of
$V^{\ast}\otimes V$) satisfying $u_i(v_j)=\delta_{i,j}$ (resp.,
$X_{i,j}=u_i\otimes v_j$) for each $i,j$. Set $$
I(2m,n):=\bigl\{(i_1,\cdots,i_n)\bigm|\text{$i_j\in\{1,2,\cdots,2m\}$
for each $j$}\bigr\}.
$$
For each $\bi=(i_1,\cdots,i_n)\in I(2m,n)$, we set
$v_{\bi}:=v_{i_1}\otimes\cdots\otimes v_{i_n}$. An endomorphism
$\mu\in\End(V^{\otimes n})$ is uniquely determined by its
coefficient $\mu_{\bi,\bj}$ with respect to the basis
$\{v_{\bi}\}_{\bi\in I(2m,n)}$, that is, $$
\mu(v_{\bj})=\sum_{\bi\in I(2m,n)}\mu_{\bi,\bj}v_{\bi}.
$$

For any commutative $\A$-algebra $K$, we use $F_{K}(2m)$ to denote
the tensor algebra over $V^{\ast}\otimes V$, which can be identified
with the free $K$-algebra generated by the $(2m)^2$ symbols
$X_{i,j}$ for $i,j\in\{1,2,\cdots,2m\}$. For each $\bi,\bj\in
I(2m,n)$, we write $$ X_{\bi,\bj}:=X_{i_1,j_1}X_{i_2,j_2}\cdots
X_{i_n,j_n}.
$$
Following \cite[Section 2]{Oe2}, for an endomorphism $\mu\in\End(V^{\otimes n})$, we write $$ \mu\wr
X_{\bi,\bj}=\sum_{\bk\in I(2m,n)}\mu_{\bi,\bk}X_{\bk,\bj},\quad
X_{\bi,\bj}\wr\mu=\sum_{\bk\in I(2m,n)}X_{\bi,\bk}\mu_{\bk,\bj}.
$$
Note that the algebra $F_{K}(2m)$ possesses a structure of bialgebra
where comultiplication and augmentation on the generators
$x_{\bi,\bj}$ are given by $$ \Delta(X_{\bi,\bj})=\sum_{\bk\in
I(2m,n)}X_{\bi,\bk}\otimes X_{\bk,\bj},\quad
\varepsilon(X_{\bi,\bj})=\delta_{\bi,\bj}.
$$

Following Oehms \cite{Oe2}, we can apply the generalized FRT construction with
respect to the subset $\{\beta,\gamma\}\subseteq\End(V)\otimes
\End(V)$ to obtain a new bialgebra $A_K^{sy}(2m)$. Precisely, we
define $$ A_K^{sy}(2m):=F_{K}(2m)/\langle\beta\wr
X_{\bi,\bj}-X_{\bi,\bj}\wr\beta,\,\gamma\wr
X_{\bi,\bj}-X_{\bi,\bj}\wr\gamma,\,\,\bi,\bj\in I(2m,2)\rangle,
$$
which is necessarily a bialgebra as the ideal generated by $\beta\wr
X_{\bi,\bj}-X_{\bi,\bj}\wr\beta$ and $\gamma\wr
X_{\bi,\bj}-X_{\bi,\bj}\wr\gamma$ for $\bi,\bj\in I(2m,2)$ is
actually a bi-ideal. Clearly, $A_K^{sy}(2m)$ is a $\Z$-graded
algebra, that is,
$$ A_K^{sy}(2m)=\bigoplus_{0\leq n\in\Z}A_K^{sy}(2m,n).
$$
For each $\bi,\bj\in I(2m,n)$, let $x_{\bi,\bj}$ be the canonical
image of $X_{\bi,\bj}$ in $A^{sy}(2m,n)$. By the main result in
\cite{Oe2}, the natural map
$A_{\A}^{sy}(2m,n)\otimes_{\A}K\rightarrow A_K^{sy}(2m,n)$ is an
isomorphism for any commutative $\A$-algebra $K$. The algebra
$A_K^{sy}(2m)$ is a quantization of the coordinate ring of the
symplectic monoid scheme $\SpM_{2m}$ which is defined by (for any commutative
$\mathbb{Z}$-algebra $K$)
$$ \SpM_{2m}(K):=\bigl\{A\in M_{2m}(K)\bigm|\exists d(A)\in K,
A^tJA=AJA^t=d(A)J\bigr\},
$$
where $J$ is the Gram matrix of the given skew bilinear form with
respect to the basis $\{v_i\}_{i=1}^{2m}$.\smallskip

In \cite{Oe2}, Oehms constructed a basis for $A_{\A}^{sy}(2m,n)$ for
each $n$. To describe that basis, we need some more notations. For
$i=1,2,\cdots,n-1$, we set $$ \beta_i:=\id_{V^{\otimes
i-1}}\otimes\beta\otimes\id_{V^{\otimes n-i-1}},\quad
\gamma_i:=\id_{V^{\otimes i-1}}\otimes\gamma\otimes\id_{V^{\otimes
n-i-1}}.
$$ Recall that for $i=1,2,\cdots,n-2$, we have the braid relation
$\beta_i\beta_{i+1}\beta_i=\beta_{i+1}\beta_i\beta_{i+1}$.
Recall also that (see Section 1) we have used $\BS_n$ to denote the symmetric group on $n$ letters.
For each $w\in\BS_n$ there is a well--defined element
$\beta(w)\in\End(V^{\otimes n})$, where $\beta(w)=\beta_{i_1}\dots
\beta_{i_k}$ whenever $k$ is minimal such that
$w=(i_1,i_1+1)\dots(i_k,i_k+1)$. For each partition $\lam$ of $n$,
let $\lam^t$ be the transpose of $\lam$, and let $\BS_{\lam^t}$ be
the corresponding Young subgroup of $\BS_n$ (which is the subgroup
fixing the sets
$\{1,2,\cdots,\lam^t_1\},\{\lam^t_1+1,\lam^t_1+2,\cdots,\lam^t_1+\lam^t_2\},\cdots$).
For each $w\in\BS_n$ and each pair of multi-indices $\bi,\bj\in
I(2m,n)$, we set (following Oehms)
$$
T_q^{\lam}(\bi:\bj):=\sum_{w\in\BS_{\lam^t}}(-q^2)^{-\ell(w)}\beta(w)\wr
x_{\bi,\bj},
$$
and call $T_q^{\lam}(\bi:\bj)$ a {\it quantum symplectic
bideterminant}. \smallskip

Recall that for each partition $\lam=(\lam_1,\lam_2,\cdots,)$ one
can naturally associates a Young diagram reading row lengths out of
the components $\lam_i$. For example, $$ (3,3,2,1)\leftrightarrow
\begin{tabular}
[c]{|l|ll}\hline ${}$ & ${}$ & \multicolumn{1}{|l|}{${}$}\\
\hline ${}$ & ${}$ & \multicolumn{1}{|l|}{${}$}\\
\hline ${}$ & ${}$ & \multicolumn{1}{|l}{}\\
\cline{1-2} ${}$ &  & \\
\cline{1-1}
\end{tabular}\,\, .
$$
For each positive integer $k$, let $\Lambda^{+}(m,k)$ be the set of
partitions of $k$ into not more than $m$ parts.  Let
$\lam\in\Lambda^{+}(m,k)$ with $0\leq k\leq n$. For each $\bi\in
I(2m,k)$, one can construct a $\lam$-tableau $T_{\bi}^{\lam}$ by
inserting the components of $\bi$ column by column into the boxes of
the Young diagram of $\lam$. In the above example, $$
T_{\bi}^{\lam}=\begin{tabular} [c]{|l|ll}\hline $\mathtt{i_1}$ &
$\mathtt{i_5}$ & \multicolumn{1}{|l|}{$\mathtt{i_8}$}\\\hline
$\mathtt{i_2}$ & $\mathtt{i_6}$ &
\multicolumn{1}{|l|}{$\mathtt{i_9}$}\\\hline
$\mathtt{i_3}$ & $\mathtt{i_7}$ & \multicolumn{1}{|l}{}\\\cline{1-2}%
$\mathtt{i_4}$ &  & \\\cline{1-1}%
\end{tabular}\,\, .
$$
Let $\bi_{\lam}$ be the unique multi-index in $I(2m,k)$ such that in
the corresponding $\lam$-tableau $T_{\bi}^{\lam}$, the $j$th row is
filled with the number $j$  for each integer $j$ with $1\leq j\leq
m$. For each integer $1\leq i\leq \lam_1=\ell(\lam^t)$, we use
$w_{i,0}$ to denote the unique longest element in the symmetric
group $\mathfrak{S}_{\lam^t_i}$. Let $$
w_{0,\lam^t}:=\prod_{i=1}^{\lam_1}w_{i,0}.
$$
Let $\widehat{\bi}_{\lam}\in I(2m,k)$ be such that
$T_{\widehat{\bi}_{\lam}}^{\lam}=T_{\bi}^{\lam}w_{0,\lam^t}$.
Following \cite[Section 6]{Oe2}, we put a new order $``\prec"$ on
the set $\{1,2,\cdots,2m\}$, namely, $$ m\prec m'\prec (m-1)\prec
(m-1)'\prec\cdots\prec 1\prec 1'.
$$
We define $I_{\lam}^{mys}$ to be the set of $\bi\in I(2m,k)$ such
that the entries in $T_{\bi}^{\lam}$ are weakly increasing along
rows and strictly increasing down columns according to the order
$``\prec"$, and for each $1\leq i\leq m$, $i, i'$ are limited to the
first $m-i+1$ rows.

In the previous example, we have
$$ T_{\bi}^{\lam}=\begin{tabular} [c]{|l|ll}\hline
$\mathtt{1}$ & $\mathtt{1}$ &
\multicolumn{1}{|l|}{$\mathtt{1}$}\\\hline $\mathtt{2}$ &
$\mathtt{2}$ & \multicolumn{1}{|l|}{$\mathtt{2}$}\\\hline
$\mathtt{3}$ & $\mathtt{3}$ & \multicolumn{1}{|l}{}\\\cline{1-2}%
$\mathtt{4}$ &  & \\\cline{1-1}%
\end{tabular}\,\, ,\quad\,\,T_{\widehat{\bi}_{\lam}}^{\lam}=\begin{tabular} [c]{|l|ll}\hline
$\mathtt{4}$ & $\mathtt{3}$ &
\multicolumn{1}{|l|}{$\mathtt{2}$}\\\hline $\mathtt{3}$ &
$\mathtt{2}$ & \multicolumn{1}{|l|}{$\mathtt{1}$}\\\hline
$\mathtt{2}$ & $\mathtt{1}$ & \multicolumn{1}{|l}{}\\\cline{1-2}%
$\mathtt{1}$ &  & \\\cline{1-1}%
\end{tabular}\,\, .
$$

\begin{lem} \label{reversetableau} For any $\lam\in\Lambda^{+}(m,k)$,  we have $$
\widehat{\bi}_{\lam}\in I_{\lam}^{mys}.
$$
\end{lem}
\begin{proof} This follows directly from the definition of
$I_{\lam}^{mys}$.
\end{proof}

For each ${\lam}\in\Lambda^{+}(m,k)$, let $I_{\lam}^{<}$ be the set
of multi-indices $\bi\in I(2m,k)$ such that the entries in
$T_{\bi}^{\lam}$ are strictly increasing down columns with respect
to the usual order ``$<$" on $\{1,2,\cdots,2m\}$. For each integer
$j\geq 1$, we write
$$
\Delta^{(j)}:=\bigl(\Delta\otimes 1^{\otimes j-1}\bigr)\circ\cdots\circ\bigl(\Delta\otimes 1\bigr)\circ\Delta.
$$
The following result was used in \cite[Section 15]{Oe2} without proof.
Since we shall use it in this paper, we include a proof here.

\begin{lem} \label{comultirule} Let ${\lam}\in\Lambda^{+}(m,k)$,
$\bi,\bj\in I_{\lam}^{mys}$ and $j\geq 1$ be an integer. We have
that $$\begin{aligned}
&\quad\,\,\,\Delta^{(j)}\Bigl(T_q^{\lam}(\bi,\bj)\Bigr)\\
&=\sum_{\bh^{(1)},\cdots,\bh^{(j)}\in
I_{\lam}^{<}}T_q^{\lam}(\bi,\bh^{(1)})\otimes
T_q^{\lam}(\bh^{(1)},\bh^{(2)}) \otimes\cdots\otimes
T_q^{\lam}(\bh^{(j)},\bj).
\end{aligned} $$
\end{lem}

\begin{proof} We only prove the special case where $j=1$. The general case follows from the same argument.
For each positive integer $r$, we set $\omega_r:=(\underbrace{1,1,\cdots,1}_{\text{$r$
copies}})$. Let $\lam^t=(\mu_1,\cdots,\mu_p)$ be the transpose of
$\lam$, where $p=\lam_1$. We split $\bj$ into $p$ multi-indices
$\bj^l\in I(2m,\mu_l)$, where for each $l\in\{1,2,\cdots,p\}$, the
entries of $\bj^l$ are taken from the $l$th column of
$T_{\bj}^{\lam}$. The same thing can be done with $\bi$. Then $$
T_q^{\lam}(\bi,\bj)=T_q^{\omega_{\mu_1}}(\bi^1,\bj^1)T_q^{\omega_{\mu_2}}(\bi^2,\bj^2)\cdots
T_q^{\omega_{\mu_p}}(\bi^p,\bj^p),
$$
hence $$
\Delta\Bigl(T_q^{\lam}(\bi,\bj)\Bigr)=\Delta\Bigl(T_q^{\omega_{\mu_1}}(\bi^1,\bj^1)\Bigr)
\Delta\Bigl(T_q^{\omega_{\mu_2}}(\bi^2,\bj^2)\Bigr)\cdots
\Delta\Bigl(T_q^{\omega_{\mu_p}}(\bi^p,\bj^p)\Bigr).
$$
Therefore, to prove the lemma, it suffices to consider the case
where $p=1$.

Now assume that $p=1$. Then $\lam=\omega_k$. If $\bi,\bj\in
I_{\omega_k}^{<}$, then the lemma follows from \cite[(20)]{Oe2}. In
general, by the definition of $I_{\lam}^{mys}$ (\cite[Section
6]{Oe2}), we can find some simple reflections
$s_{q_1},\cdots,s_{q_a},s_{l_1},$ $\cdots,s_{l_b}\in\BS_{\lam^t}$
such that
\begin{enumerate}
\item $\widetilde{\bi}:=\bi s_{q_1}s_{q_2}\cdots s_{q_a}\in I_{\omega_k}^{<}$,
$\widetilde{\bj}:=\bj s_{l_1}s_{l_2}\cdots s_{l_b}\in
I_{\omega_k}^{<}$;
\item For each integer $1\leq c\leq a$, the action of $s_{q_c}$ on $\bi s_{q_1}s_{q_2}\cdots s_{q_{c-1}}$
does not exchange the indices $i,i'$ for any $i$;
\item For each integer $1\leq c\leq b$, the action of $s_{l_c}$ on $\bj s_{l_1}s_{l_2}\cdots s_{l_{c-1}}$
does not exchange the indices $i,i'$ for any $i$.
\end{enumerate}
Now using \cite[(13), (20), Corollary 11.9]{Oe2}, we deduce that there exist integer $a(\bi), b(\bj)$, such that $$
T_q^{\lam}({\bi},{\bj})=(-q)^{a(\bi)+b(\bj)}T_q^{\lam}(\widetilde{\bi},\widetilde{\bj}).
$$
Therefore, applying \cite[(20)]{Oe2}, we get that
$$\begin{aligned} \Delta\bigl(T_q^{\lam}(\bi,\bj)\bigr)&=(-q)^{a(\bi)+b(\bj)}\Delta\bigl(T_q^{\lam}(\widetilde{\bi},
\widetilde{\bj})\bigr)\\
&=\sum_{\bh\in I_{\lam}^{<}}(-q)^{a(\bi)}T_q^{\lam}(\widetilde{\bi},\bh)\otimes
(-q)^{b(\bj)}T_q^{\lam}(\bh,\widetilde{\bj})\\
&=\sum_{\bh\in I_{\lam}^{<}}T_q^{\lam}(\bi,\bh)\otimes
T_q^{\lam}(\bh,\bj),
\end{aligned}
$$
as required.
\end{proof}

For each integer $n\geq 0$, set $$
\Lambda_n:=\bigl\{\underline{\lam}:=(\lam,l)\bigm|l\in\Z, 0\leq
l\leq n/2, \lam\in\Lambda^{+}(m,n-2l)\bigr\}.
$$
Let $d_q\in A_{\A}^{sy}(2m)$ be the central group-like element as
defined in \cite[(7)]{Oe2} and \cite[Corollary 6.3]{Ha}. By
definition,
$$
d_q=-q^{-\rho_k-\rho_l}\epsilon_k\epsilon_lx_{(k,k'),(l,l')}\wr\gamma\in
A_{\A}^{sy}(2m,2),
$$
which is independent of the choices of $k,l\in\{1,2,\cdots,2m\}$.
For each $\underline{\lam}:=(\lam,l)\in\Lambda_n$ and each
$\bi,\bj\in I_{\lam}^{mys}$, we set
$D_{\bi,\bj}^{\underline{\lam}}:=d_q^lT_q^{\lam}(\bi,\bj)$.  Oehms
(\cite[(7.1)]{Oe2}) proved that $A_{\A}^{sy}(2m,n)\otimes_{\A}K\cong
A_K^{sy}(2m,n)$ for any commutative $\A$-algebra $K$. Indeed,
$A_{\A}^{sy}(2m,n)$ is a free $\A$-module and the elements in the
following set
\begin{equation}\label{basis1}
\Bigl\{D_{\bi,\bj}^{\underline{\lam}}\Bigm|\text{$\underline{\lam}=(\lam,l)\in\Lambda_n$,
$\bi,\bj\in I_{\lam}^{mys}$}\Bigr\}
\end{equation}
form an $\A$-basis of $A_{\A}^{sy}(2m,n)$.
\smallskip

For each integer $k\geq 0$, we put an order on the set
$\Lambda^{+}(m,k)$, write $\lam\prec\mu$ if $\lam^t_i=\mu^t_i$ for
$i=1,2,\cdots,s-1$ and $\lam^t_s<\mu^t_s$ for some $s$. For example,
for $\lam=(2,2,1),\mu=(3,1,1)\in\Lambda^{+}(3,5)$, we have
$\mu\prec\lam$. Next, we put an order on $\Lambda_n$. For any
$\underline{\lam}=(\lam,l),\underline{\mu}=(\mu,b)\in\Lambda_n$,
write $\underline{\lam}\prec\underline{\mu}$ if $l<b$ or $l=b$ and
$\lam\prec\mu$. For each $\underline{\lam}=(\lam,l)\in\Lambda_n$, we
set $M(\underline{\lam})=I_{\lam}^{mys}$. Let ``$\ast$" be the
$\A$-linear involution of ${A}_{\A}^{sy}(2m,n)$ which is defined on
generators by
$\Bigl(D_{\bi,\bj}^{\underline{\lam}}\Bigr)^{\ast}=D_{\bj,\bi}^{\underline{\lam}}$
for all $\underline{\lam}\in\Lambda_n, \bi,\bj\in\Lambda_n$.

\begin{lem} \label{cellco} {\rm (\cite[Theorem 7.1]{Oe2})} With respect to the
ordered set $(\Lambda_n,\prec)$, the finite set
$M(\underline{\lam})$ and the $\A$-linear involution ``$\ast$", the
coalgebra ${A}_{\A}^{sy}(2m,n)$ is a cellular coalgebra in the sense of \cite[Section 5, page 860]{Oe2} with
cellular basis given by $$
\Bigl\{D_{\bi,\bj}^{\underline{\lam}}\Bigm|\text{$\underline{\lam}=(\lam,l)\in\Lambda_n$,
$\bi,\bj\in I_{\lam}^{mys}$}\Bigr\}.
$$
\end{lem}

For each $\underline{\lam}=(\lam,l)\in\Lambda_n$, we set
$$\begin{aligned}
{A}_{\A}^{sy}(2m,n)^{\succ\underline{\lam}}&:=\text{$\A$-Span}\Bigl\{D_{\bi,\bj}^{\underline{\mu}}\Bigm|
\text{$\underline{\lam}\prec\underline{\mu}=(\mu,b)\in\Lambda_n$,\,\,$\bi,\bj\in
I_{\mu}^{mys}$}\Bigr\},\\
{A}_{\A}^{sy}(2m,n)^{\succeq\underline{\lam}}&:=\text{$\A$-Span}\Bigl\{D_{\bi,\bj}^{\underline{\mu}}\Bigm|
\text{$\underline{\lam}\preceq\underline{\mu}=(\mu,b)\in\Lambda_n$,\,\,$\bi,\bj\in
I_{\mu}^{mys}$}\Bigr\}.\end{aligned}
$$
By general theory of cellular coalgebra (\cite[Section 5, page 860]{Oe2}), we know that both
${A}_{\A}^{sy}(2m,n)^{\succ\underline{\lam}}$ and
${A}_{\A}^{sy}(2m,n)^{\succeq\underline{\lam}}$ are two-sided
coideal of ${A}_{\A}^{sy}(2m,n)$, and we have a two-sided
${A}_{\A}^{sy}(2m,n)$ comodule isomorphism $$
{A}_{\A}^{sy}(2m,n)^{\succeq\underline{\lam}}/{A}_{\A}^{sy}(2m,n)^{\succ\underline{\lam}}\cong
\nabla^r(\underline{\lam})\otimes
\nabla(\underline{\lam}),
$$
where $\nabla(\underline{\lam})$ (resp.,
$\nabla^r(\underline{\lam})$) is the cell right
(resp., left) comodule corresponding to $\underline{\lam}$. Note
that
$\nabla^r(\underline{\lam})=\nabla(\underline{\lam})$
as $\A$-module, while the left ${A}_{\A}^{sy}(2m,n)$-coaction is
obtained by twisting the right ${A}_{\A}^{sy}(2m,n)$-coaction using
``$\ast$". If we extend the base ring $\A$ to the rational
function field $\Q(q)$, then $\nabla(\underline{\lam})_{\Q(q)}$ is
an irreducible right ${A}_{\Q(q)}^{sy}(2m,n)$-comodule.
\smallskip

For any commutative $\A$-algebra $K$, we set $$
S^{sy}_{K}(2m,n):=\End_{\bb_n(-\zeta^{2m+1},\zeta)}\Bigl(V_{K}^{\otimes
n}\Bigr),
$$
where $\zeta$ is the image of $q$ in $K$. The algebra
$S_K^{sy}(m,n)$ is called ``symplectic $\zeta$-Schur algebra" by Oehms
(\cite{Oe2}). By \cite[Proposition 2.1]{Ha}, there is a
non-degenerate pairing between $F_{\Q(q)}(2m)$ and
$\End_{\Q(q)}(V_{\Q(q)}^{\otimes n})$ such that for any
$X_{\bi,\bj}\in F_{\Q(q)}(2m,n)$, $f\in
\End_{\Q(q)}(V_{\Q(q)}^{\otimes \widehat{n}})$, where $\bi,\bj\in
I(2m,n), n,\widehat{n}\in\Z^{\geq 0}$, $$ \langle
X_{\bi,\bj},f\rangle_0:=\begin{cases}
u_{\bi}\Bigl(f\bigl(v_{\bj}\bigr)\Bigr),&\text{if $n=\widehat{n}$;}\\
0, &\text{otherwise.}
\end{cases}
$$
where $u_{\bi}:=u_{i_1}\otimes\cdots\otimes u_{i_n}$, $\{u_i\}$ is
the basis of $V^{\ast}$ dual to $\{v_i\}$. By \cite[Proposition
2.1]{Ha}, it induces a non-degenerate pairing $\langle,\rangle_0$
between $A^{sy}_{\Q(q)}(2m)$ and $$ \bigoplus_{0\leq
n\in\Z}S^{sy}_{\Q(q)}(2m,n).
$$
Furthermore, it induces a pairing $\langle,\rangle_0$ between
$A^{sy}_{\Q(q)}(2m)$ and ${\U}_{\Q(q)}$. Precisely, for any
$x_{\bi,\bj}\in A^{sy}_{\Q(q)}(2m,n)$, $u\in {\U}_{\Q(q)}$, where
$\bi,\bj\in I(2m,n), n\in\Z^{\geq 0}$, $$ \langle
x_{\bi,\bj},u\rangle_0:=\langle x_{\bi,\bj},\psi_C(u)\rangle_0,
$$
where $\psi_C$ is the canonical homomorphism $$
\psi_C:{\U}_{\Q(q)}\rightarrow
S^{sy}_{\Q(q)}(2m,n):=\End_{\bb_n(-q^{2m+1},q)}\Bigl(V_{\Q(q)}^{\otimes
n}\Bigr).$$ In \cite{Oe1} and \cite{Oe2}, Oehms proved the pairing
$\langle,\rangle_0$ actually induces an $\A$-algebra isomorphism
$S^{sy}_{\A}(2m,n)\cong \bigl(A^{sy}_{\A}(2m,n)\bigr)^{\ast}$ for
each $n\in\Z^{\geq 0}$. Note that our $F(2m)$ is just $T(E)$ in the
notation of \cite[Section 2]{Ha}, while the ideal generated by
$\beta\wr X_{\bi,\bj}-X_{\bi,\bj}\wr\beta$ for $\bi,\bj\in I(2m,2)$
in our paper is the same as the ideal generated by
$\image\bigl(\id_{E\otimes E}-\beta_{E}\bigr)$ in \cite[Section
2]{Ha}. Note also that over $\A$, the algebra $A^{sy}_{\A}(2m)$ is
only a homomorphic image of $S(E)$ in \cite[Proposition 2.1]{Ha}.
However, if we work over $\Q(q)$, then $A^{sy}_{\Q(q)}(2m)$
coincides with $S(E)$ in the notation of \cite[Proposition 2.1]{Ha}
because of the relation (\ref{BG}). Therefore, for each $\underline{\lam}=(\lam,l)\in\Lambda_n$,
$\nabla(\underline{\lam})_{\Q(q)}$ is also an irreducible
left ${S}_{\Q(q)}^{sy}(2m,n)$-module, and hence an irreducible left
$\U_{\Q(q)}$-module.\smallskip

For any commutative $\A$-algebra $K$,  let $\zeta$ be the natural
image of $q$ in $K$, we define
$\widetilde{A}_K^{sy}(2m):=A_K^{sy}(2m)/\langle d_{\zeta}-1\rangle$,
where $d_{\zeta}$ is the natural image of $d_q$ in $A_K^{sy}(2m)$.
Note that $A_K^{sy}(2m)$ is a quantized version of the coordinate algebra of the symplectic monoid $\SpM_{2m}(K)$,
while $\widetilde{A}_K^{sy}(2m)$ is a quantized version of the coordinate algebra of the symplectic group ${\rm Sp}_{2m}(K)$.
The algebra $\widetilde{A}_K^{sy}(2m)$ will play a key role in the proof of Theorem \ref{mainthm1} in the next section.
We use $\pi_C$ to denote the natural projection
$$
\pi_{C}:\,\,{A}_K^{sy}(2m)\twoheadrightarrow\widetilde{A}_K^{sy}(2m).
$$

For each integer $n\geq 0$ and each commutative $\A$-algebra $K$, we
define
$$
\Lambda^{+}(m,\leq\! n):=\bigsqcup_{0\leq k\leq
n}\Lambda^{+}(m,k),\quad
{A}_{K}^{sy}(2m,\leq\! n):=\bigoplus_{0\leq
k\leq n}{A}_{K}^{sy}(2m,k).
$$
{\it We shall call ${A}_{\A}^{sy}(2m,\leq\!n)$ a finite
truncation of the quantized coordinate algebra
${A}_{\A}^{sy}(2m)$ of $\SpM_{2m}$}. Clearly, ${A}_{\A}^{sy}(2m,\leq\!\! n)$ is a sub-coalgebra of
${A}_{\A}^{sy}(2m)$ as well as a free $\A$-submodule with basis $$
\Bigl\{T_{q}^{\lam}(\bi,\bj)\Bigm|\text{$\lam\in\Lambda^{+}(m,\leq\!
n)$,\,\, $\bi,\bj\in I_{\lam}^{mys}$}\Bigr\}.
$$

The author is grateful to Professor S. Doty and the referee for the first part of the following corollary.

\begin{cor} \label{Doc} Let $n\geq 0$ be an integer and $K$ be a commutative $\A$-algebra.

1) The maps
$$\begin{aligned}\pi_{C}\!\!\downarrow_{{A}_{K}^{sy}(2m,n)}:\,\, {A}_{K}^{sy}(2m,n)&\rightarrow\pi_{C}\bigl({A}_K^{sy}(2m,n)\bigr)\\
\pi_{C}\!\!\downarrow_{{A}_{K}^{sy}(2m,\leq\! n)}:\,\, {A}_{K}^{sy}(2m,\leq\! n)&\rightarrow\pi_{C}\bigl({A}_K^{sy}(2m,\leq\! n)\bigr)\end{aligned}$$ are both isomorphisms.

2) $\widetilde{A}_{\A}^{sy}(2m)$ is a free $\A$-module and the elements
in the following set $$
\biggl\{\pi_C\Bigl(T_q^{\lam}(\bi,\bj)\Bigr)\biggm|\begin{matrix}\text{$\lam\in\Lambda^{+}(m,n-2l)$
for some integer $n,l$}\\ \text{with $n\geq 0$, $0\leq l\leq n/2$,
$\bi,\bj\in I_{\lam}^{mys}$}\end{matrix}\biggr\}
$$
form an $\A$-basis of $\widetilde{A}_{\A}^{sy}(2m)$. Moreover, the canonical map
$$\widetilde{A}_{\A}^{sy}(2m)\otimes_{\A}K\rightarrow
\widetilde{A}_K^{sy}(2m)$$ is an isomorphism.
\end{cor}

\begin{proof} Suppose that $0\neq x\in\Ker\bigl(\pi_{C}\!\!\downarrow_{{A}_{K}^{sy}(2m,n)}\bigr)$. Since $d_q$ is central, it follows that $x=(d_q-1)y$ for some $y\in A_K^{sy}(2m)$. Note that $d_q\neq 0$ is homogeneous of degree $2$, while $x$ is homogeneous of degree $n$. By Lemma \ref{cellco}, the elements in the following set $$
\Bigl\{d_q^lT_q^{\lam}(\bi,\bj)\Bigm|\text{$\underline{\lam}=(\lam,l)\in\Lambda_k$ for some $k\geq 0$ and $0\leq l\leq k/2$, $\bi,\bj\in I_{\lam}^{mys}$}\Bigr\}
$$
form a homogeneous basis of $A_K^{sy}(2m)$. Expressing $y$ into a linear combination of this basis and comparing the degree of each homogeneous component, we get a contradiction. This proves that $\pi_{C}\!\!\downarrow_{{A}_{K}^{sy}(2m,n)}$ is injective and hence an isomorphism. As a result, we deduce that
$\pi_{C}\!\!\downarrow_{{A}_{K}^{sy}(2m,\leq\! n)}$ is also an isomorphism. This proves 1). The statement 2) is an immediate consequence of the statement 1).
\end{proof}

{\it Henceforth, we shall use Corollary \ref{Doc} to identify ${A}_{K}^{sy}(2m,n)$ and ${A}_{K}^{sy}(2m,\leq\! n)$ as subspaces
of $\widetilde{A}_{\A}^{sy}(2m)$ without further comments.} The involution ``$\ast$" gives rise to an $\A$-linear involution of ${A}_{\A}^{sy}(2m,\leq\! n)$, which will be still denoted
by ``$\ast$". Recall that in the paragraph below (\ref{basis1}) we have introduced an order ``$\prec$" on the set of partitions of a fixed integer. Now we generalize it to the case of partitions of possibly different integers.
For any ${\lam},{\mu}\in\Lambda^{+}(m,\leq\! n)$,
write ${\lam}\prec{\mu}$ if $$ \text{$|\lam|-|\mu|\in
2\mathbb{N}$\quad or\quad $|\lam|=|\mu|$ and $\lam\prec\mu$.}$$

\begin{cor}  \label{cellco2} With respect to the
ordered set $(\Lambda^{+}(m,\leq\!\! n),\prec)$, the
finite set $I_{\lam}^{mys}$ (for each $\lam\in
\Lambda^{+}(m,\leq\!\! n)$) and the $\A$-linear involution
``$\ast$", the coalgebra ${A}_{\A}^{sy}(2m,\leq\!\! n)$ is
a cellular coalgebra with cellular basis given by $$
\Bigl\{T_{q}^{\lam}(\bi,\bj)\Bigm|\text{$\lam\in\Lambda^{+}(m,\leq\!
n)$,\,\, $\bi,\bj\in I_{\lam}^{mys}$}\Bigr\}.
$$
Furthermore, for each commutative $\A$-algebra $K$, the canonical
map $$ {A}_{\A}^{sy}(2m,\leq\!\! n)\otimes_{\A}
K\rightarrow {A}_{K}^{sy}(2m,\leq\!\! n)
$$
is an isomorphism.
\end{cor}

Let $\Lambda^{+}(m):=\bigcup_{k\geq 0}\Lambda^{+}(m,k)$. For each
${\lam}\in\Lambda^{+}(m,\leq n)$, we set
$$\begin{aligned}
{A}_{\A}^{sy}(2m,\leq\! n)^{\succ{\lam}}&=\text{$\A$-Span}\bigl\{T_q^{\mu}(\bi,\bj)\bigm|
\text{${\lam}\prec{\mu}\in\Lambda^{+}(m,\leq\! n)$, $\bi,\bj\in
I_{\mu}^{mys}$}\bigr\},\\
{A}_{\A}^{sy}(2m,\leq\! n)^{\succeq{\lam}}&=\text{$\A$-Span}\bigl\{T_q^{\mu}(\bi,\bj)\bigm|
\text{${\lam}\preceq{\mu}\in\Lambda^{+}(m,\leq\! n)$, $\bi,\bj\in
I_{\mu}^{mys}$}\bigr\}.\end{aligned}
$$
By general theory of cellular coalgebra, we have that both
${A}_{\A}^{sy}(2m,\leq\! n)^{\succ{\lam}}$ and
${A}_{\A}^{sy}(2m,\leq\! n)^{\succeq{\lam}}$ are two-sided coideal
of ${A}_{\A}^{sy}(2m,\leq\! n)$, and we have a two-sided
${A}_{\A}^{sy}(2m,\leq\! n)$-comodule isomorphism
$$
{A}_{\A}^{sy}(2m,\leq\! n)^{\succeq{\lam}}/{A}_{\A}^{sy}(2m,\leq\! n)^{\succ{\lam}}\cong
\nabla^r({\lam})\otimes \nabla({\lam}),
$$
where $\nabla({\lam})$ (resp., $\nabla^r({\lam})$) is the cell right
(resp., left) comodule corresponding to ${\lam}$. Note that
$\nabla^r({\lam})=\nabla({\lam})$ as $\A$-module, while the left
${A}_{\A}^{sy}(2m,\leq\! n)$-coaction is obtained by
twisting the right ${A}_{\A}^{sy}(2m,\leq\! n)$-coaction
using ``$\ast$". If we extend the base ring $\A$ to the rational
function field $\Q(q)$, then $\nabla({\lam})_{\Q(q)}$ is an
irreducible right ${A}_{\Q(q)}^{sy}(2m,\leq\!
n)$-comodule.

Suppose that $\lam\in\Lambda^{+}(m,k)$ with $0\leq k\leq n$. We set
$\underline{\lam}=(\lam,0)\in\Lambda_{k}$. Note that the
${A}_{\A}^{sy}(2m,\leq\! n)$-comodule $\nabla(\lam)_{\A}$
is isomorphic to the restriction of the
${A}_{\A}^{sy}(2m,k)$-comodule $\nabla(\underline{\lam})_{\Q(q)}$.
In particular, every simple ${A}_{\Q(q)}^{sy}(2m,\leq\!
n)$-comodule comes from the restriction of a simple
${A}_{\Q(q)}^{sy}(2m,k)$-comodule, or equivalently, of a simple
${S}_{\Q(q)}^{sy}(2m,k)$-module for some $0\leq k\leq n$. Therefore,
by the surjection from $\dot{\U}_{\Q(q)}$ onto
${S}_{\Q(q)}^{sy}(2m,k)$, we see that every simple
${A}_{\Q(q)}^{sy}(2m,\leq\! n)$-comodule comes from the
restriction of a simple $\dot{\U}_{\Q(q)}$-module. For any field $K$
which is an $\A$-algebra, we define $$
\nabla_K(\lam):=\nabla(\lam)\otimes_{\A}K,\,\,\,
{S}_{K}^{sy}(2m,k):={S}_{\A}^{sy}(2m,k)\otimes_{\A}K.
$$
Then $\nabla_K(\lam)\cong \nabla_K(\lam,0)$ can be regarded as a
${S}_{K}^{sy}(2m,k)$-module.

\begin{cor} Let $K$ be a field which is an $\A$-algebra. Let $k\geq 0$ be an integer and $\lam\in\Lambda^{+}(m,k)$.
Then as a ${S}_{K}^{sy}(2m,k)$-module, $\nabla_K(\lam)$ has a unique
simple socle.
\end{cor}
\begin{proof} By \cite{Oe2}, ${S}_{K}^{sy}(2m,k)$ is a cellular quasi-hereditary algebra. By
definition, the dual of $\nabla_K(\lam)$ is a cell module of
${S}_{K}^{sy}(2m,k)$. Thus, the corollary follows from the general
theory of cellular quasi-hereditary algebra.
\end{proof}

Recall that $\psi_C$ induces a natural morphism from
$\U_K(\mathfrak{g})$ to ${S}_{K}^{sy}(2m,k)$, via which
${S}_{K}^{sy}(2m,k)$-module $\nabla_K(\lam)$ can be regarded as an
$\U_K$-module. Note that the above corollary does not immediately
imply that $\nabla_K(\lam)$ has a unique simple socle as an
$\U_K$-module because at this moment we did not know if that natural
morphism from $\U_K$ to ${S}_{K}^{sy}(2m,k)$ is surjective or not.
However, in the next section we shall show that this is indeed the
case.

\begin{cor} \label{codecomp} For each integer $n\geq 0$, we have a two-sided
${A}_{\A}^{sy}(2m,\leq\! n)$-comodule isomorphism $$
\theta_n: {A}_{\Q(q)}^{sy}(2m,\leq\!
n)\cong\bigoplus_{\lam\in\Lambda^{+}(m,\leq
n)}\nabla^r({\lam})_{\Q(q)}\otimes
\nabla({\lam})_{\Q(q)}.
$$
Furthermore, we have a two-sided
$\widetilde{A}_{\A}^{sy}(2m)$-comodule isomorphism $$
{\theta}:\widetilde{A}_{\Q(q)}^{sy}(2m)\cong\bigoplus_{k\geq
0,
\lam\in\Lambda^{+}(m,k)}\nabla^r({\lam})_{\Q(q)}\otimes
\nabla({\lam})_{\Q(q)},
$$
and a commutative diagram $$\begin{CD}
{A}_{\Q(q)}^{sy}(2m,\leq\! n)
@>{\theta_n}>>\bigoplus_{\lam\in\Lambda^{+}(m,\leq
n)}\nabla^r({\lam})_{\Q(q)}\otimes
\nabla({\lam})_{\Q(q)} \\
@V{} VV @V{}VV\\
\widetilde{A}_{\Q(q)}^{sy}(2m) @>{{\theta}}>>
\bigoplus_{k\geq 0,
\lam\in\Lambda^{+}(m,k)}\nabla^r({\lam})_{\Q(q)}\otimes
\nabla({\lam})_{\Q(q)}
\end{CD},
$$
where the two vertical maps are natural embedding.
\end{cor}

\begin{proof} This follows from the cellular structure of
${A}_{\A}^{sy}(2m,\leq\!\! n)$ given in Corollary
\ref{cellco2} and the fact $$
\widetilde{A}_{\A}^{sy}(2m)=\bigcup_{n\geq
0}{A}_{\A}^{sy}(2m,\leq\! n),
$$
and the following commutative diagram $$\begin{CD}
{A}_{\Q(q)}^{sy}(2m,\leq\! n)
@>{\theta_n}>>\bigoplus_{\lam\in\Lambda^{+}(m,\leq
n)}\nabla^r({\lam})_{\Q(q)}\otimes
\nabla({\lam})_{\Q(q)} \\
@V{} VV @V{}VV\\
{A}_{\Q(q)}^{sy}(2m,\leq\! \widehat{n})
@>{{\theta_{\widehat{n}}}}>> \bigoplus_{\lam\in\Lambda^{+}(m,\leq
\widehat{n})}\nabla^r({\lam})_{\Q(q)}\otimes
\nabla({\lam})_{\Q(q)}
\end{CD},
$$
where $\widehat{n}\geq n\geq 0$ and the two vertical maps are
natural embedding.
\end{proof}

\begin{cor} \label{ModuleBase} For each integer $n\geq 0$ and each
$\lam\in\Lambda^{+}(m,n)$, let $\pr_{\lam}$ be the natural
projection from $\widetilde{A}_{\Q(q)}^{sy}(2m)$ onto
$\nabla^r({\lam})_{\Q(q)}\otimes
\nabla({\lam})_{\Q(q)}$. Then the elements in the following set $$
\Bigl\{\pr_{\lam}\bigl(\pi_{C}(T_q^{\lam}(\bi,\bj))\bigr)\Bigm|\bi,\bj\in
I_{\lam}^{mys}\Bigr\}
$$
form a $\Q(q)$-basis of
$\nabla^r({\lam})_{\Q(q)}\otimes
\nabla({\lam})_{\Q(q)}$.
\end{cor}

We end this section with the following lemma. Recall the definition of $\bi_{\lam}$, $\widehat{\bi}_{\lam}$
given in the paragraph above Lemma \ref{reversetableau}.

\begin{lem} \label{Base2} For each $\lam\in\Lambda^{+}(m,n)$,  we have $$
\begin{aligned}
&\pi_{C}\Bigl(T_q^{\lam}({\bi_{\lam}},{\bi_{\lam}})\Bigr)\equiv
q^{a_{\lam}}\pi_{C}\Bigl(T_q^{\lam}(\widehat{\bi}_{\lam},\widehat{\bi}_{\lam})\Bigr)+\sum_{\substack{\bj,\kbk\in
I_{\lam}^{mys}\\\bj,\kbk\triangleleft\widehat{\bi}_{\lam}}}C_{\bj,\kbk}^{\lam}\pi_{C}\Bigl(T_q^{\lam}({\bj},{\kbk})\Bigr)\\
&\qquad\qquad\qquad\qquad\Bigl(\!\!\!\!\mod{{A}_{\A}^{sy}(2m,\leq\! n)^{\succ{\lam}}}\Bigr),\end{aligned}
$$
where $a_{\lam}\in\Z$, $C_{\bj,\kbk}^{\lam}\in\A$, and
``$\triangleleft$" is the same as defined in \cite[Proposition 8.4]{Oe2}. In
particular, we have that $$0\neq\pr_{\lam}\bigl(\pi_{C}(T_q^{\lam}(\bi_{\lam},\bi_{\lam}))\bigr)\in
\nabla^r({\lam})\otimes \nabla({\lam}).$$
\end{lem}
\begin{proof} This follows directly from \cite[(8), Proposition 8.4, Corollary 9.12]{Oe2} and the facts that
any quantum symplectic bideterminant in $A_{\Q(q)}^{sy}(2m)$ is
homogeneous and the central group-like element $d_q$ is homogeneous
of degree $2$.
\end{proof}

\bigskip\bigskip
\section{A comparison of two quantized coordinate algebras}

In \cite{Ka2}, Kashiwara introduced a version of quantized coordinate algebras (which was
denoted by $A_q^{\Z}(\mathfrak{g})$ there) for any symmetrizable Kac--Moody Lie algebras $\mathfrak{g}$.
In this section we shall first show that in the case of type $C$, the $\mathbb{Z}[q,q^{-1}]$ algebra $A_q^{\Z}(\mathfrak{g})$ and the quantized coordinate algebra $\widetilde{A}_{\A}^{sy}(2m)$ are isomorphic to
each other. Then we shall give a proof of Theorem \ref{mainthm1}. Throughout this section, we set
$$\mathfrak{g}:=\mathfrak{sp}_{2m}(\C),\,\,
\U_{\Q(q)}:=\U_{\Q(q)}(\mathfrak{g}),\,\,
\dot{\U}_{\Q(q)}:=\dot{\U}_{\Q(q)}(\mathfrak{g}).$$

Following \cite{Ka2}, we use $O_{int}(\mathfrak{g})$ to denote the
category of $\U_{\Q(q)}$-modules $M$ such that \begin{enumerate}
\item $M=\oplus_{\lam\in X}M_{\lam}$, where
$M_{\lam}:=\bigl\{x\in
M\bigm|k_ix=q^{\langle\lam,\alpha_i^{\vee}\rangle}x, \forall\,1\leq
i\leq m\bigr\}$,
\item for any $i$, $M$ is a union of finite dimensional
$\U_{\Q(q)}(\mathfrak{g}_i)$-modules, where
$\U_{\Q(q)}(\mathfrak{g}_i)$ denotes the $\Q(q)$-subalgebra
generated by $$e_i,f_i,k_1^{\pm 1},k_2^{\pm 1},\cdots,k_m^{\pm 1}.$$
\item for any $u\in M$, there exists $l\geq 0$ satisfying
$e_{i_1}\cdots e_{i_l}u=0$ for any $i_1,\cdots,i_l\in
\{1,2,\cdots,m\}$.
\end{enumerate}
Then $O_{int}(\mathfrak{g})$ is semisimple and any simple object is
isomorphic to the irreducible module $V(\lam)$ with highest weight
$\lam$. Note that $\U_{\Q(q)}$ has a structure of
bi-$\U_{\Q(q)}$-module. Hence $\U_{\Q(q)}^{\ast}$ was naturally
endowed with a structure of bi-$\U_{\Q(q)}$-module.

\begin{dfn} {\rm (\cite[(7.2.1)]{Ka2})} We set $$\begin{aligned}
A_q(\mathfrak{g})&:=\biggl\{u\in
\U_{\Q(q)}^{\ast}\biggm|\begin{matrix}\text{$\U_{\Q(q)}u$ belongs to $O_{int}(\mathfrak{g})$ and $u\U_{\Q(q)}$}\\
\text{ belongs to $O_{int}(\mathfrak{g^{opp}})$}\end{matrix}
\biggr\},\\
A_q^{\Z}(\mathfrak{g})&:=\Bigl\{u\in A_q(\mathfrak{g})\Bigm|\langle
u,\U_{\A}\rangle\subseteq\A\Bigr\},\\
A_q^{\Q}(\mathfrak{g})&:=\Bigl\{u\in A_q(\mathfrak{g})\Bigm|\langle
u,\U_{\A}\rangle\subseteq\Q[q,q^{-1}]\Bigr\},
\end{aligned}
$$
where $\langle,\rangle$ is the natural pairing.
\end{dfn}

Recall the pairing $\langle,\rangle_0$ between ${A}^{sy}_{\Q(q)}(2m)$ and ${\U}_{\Q(q)}$ (see Section 3, the second paragraph below Lemma \ref{cellco}).

\begin{lem} {\rm (\cite[Theorem 6.4(2)]{Ha})} \label{cqpairing} We have that $$
\langle d_q-1, y\rangle_0=0,\quad\text{for any $y\in{\U}_{\Q(q)}$}.
$$
and the pairing $\langle,\rangle_0$ induces a non-degenerate Hopf
pairing $\langle,\rangle_0$ between $\widetilde{A}^{sy}_{\Q(q)}(2m)$
and the quantized enveloping algebra ${\U}_{\Q(q)}$
$$
\widetilde{A}^{sy}_{\Q(q)}(2m)\times{\U}_{\Q(q)}\rightarrow \Q(q).
$$
As a result, we have two Hopf algebra injections $$
\iota_A:\widetilde{A}^{sy}_{\Q(q)}(2m)\hookrightarrow\Bigl({\U}_{\Q(q)}\Bigr)^{\circ},\,\,
\iota_U:{\U}_{\Q(q)}\hookrightarrow\Bigl(\widetilde{A}^{sy}_{\Q(q)}(2m)\Bigr)^{\circ},\,\,$$
where for any Hopf algebra $H$, $H^{\circ}$ denotes the Hopf dual of
$H$.
\end{lem}

\begin{lem} \label{mqpairing} With the above notations, the pairing $\langle,\rangle_0$ naturally induces
a non-degenerate Hopf pairing $\langle,\rangle_0$ between
$\widetilde{A}^{sy}_{\Q(q)}(2m)$ and the modified quantized
enveloping algebra $\dot{\U}_{\Q(q)}$
$$
\widetilde{A}^{sy}_{\Q(q)}(2m)\times\dot{\U}_{\Q(q)}\rightarrow
\Q(q).
$$
As a result, we have two Hopf algebra injections
$$
\widetilde{\iota}_A:\widetilde{A}^{sy}_{\Q(q)}(2m)\hookrightarrow\Bigl(\dot{\U}_{\Q(q)}\Bigr)^{\circ},\,\,
\widetilde{\iota}_U:\dot{\U}_{\Q(q)}\hookrightarrow\Bigl(\widetilde{A}^{sy}_{\Q(q)}(2m)\Bigr)^{\circ},\,\,$$
\end{lem}

\begin{proof} This is an easy consequence of Lemma \ref{keylem}, Lemma \ref{cqpairing} as well as the
following two standard facts:\begin{enumerate}
\item[(a)] any simple $\dot{\U}_{\Q(q)}$-module is a submodule of
$V_{\Q(q)}^{\otimes n}$ for some $n\in\Z^{\geq 0}$;
\item[(b)] if $u\in \dot{\U}_{\Q(q)}$ acts as $0$ on every simple
$\dot{\U}_{\Q(q)}$-module, then $u=0$.
\end{enumerate}
\end{proof}

For each $\bi\in I(2m,n)$, we define
$\wt(\bi)=(\mu_1,\cdots,\mu_{m})$, where $$ \mu_s:=\#\bigl\{1\leq
j\leq n\bigm|i_{j}=s\bigr\}-\#\bigl\{1\leq j\leq
n\bigm|i_{j}=s'\bigr\},\,\,s=1,2,\cdots,m.$$ We identify
$\wt(\bi)=(\mu_1,\cdots,\mu_{m})$ with the weight
$\mu_1\varepsilon_1+\cdots+\mu_m\varepsilon_m\in X$.

\begin{lem} \label{Apairing} Let $\bi, \bj\in I(2m,n)$. Set
$\mu=\wt(\bj)$. Then for any integer $a$ with $1\leq a\leq m$,
$$ \langle\pi_C\bigl(
T_q^{\lam}(\bi,\bj)\bigr),k_a\rangle_0=q^{\langle\mu,\alpha_a^{\vee}\rangle}\varepsilon\bigl(\pi_C\bigl(
T_q^{\lam}(\bi,\bj)\bigr)\bigr).
$$
\end{lem}

\begin{proof} This follows from direct verification.
\end{proof}

Note that ${\U}_{\Q(q)}(\mathfrak{g})\cong
{\U}_{\Q(q)}(\mathfrak{g}^{opp})$ via the anti-automorphism $\phi$
defined on generators by:
$$ e_i\mapsto f_i,\quad f_i\mapsto e_i,\quad k_i\mapsto k_i,\quad
i=1,2,\cdots,m. $$ We identify
${\U}_{\Q(q)}(\mathfrak{g}\oplus\mathfrak{g})$ with
${\U}_{\Q(q)}(\mathfrak{g})\otimes{\U}_{\Q(q)}(\mathfrak{g})$. Using
$\phi$, the bi-${\U}_{\Q(q)}(\mathfrak{g})$ structure on
$\bigl({\U}_{\Q(q)}\bigr)^{\ast}$ can be interpreted as a left
${\U}_{\Q(q)}(\mathfrak{g}\oplus\mathfrak{g})$-structure, i.e., $$
\bigl((a\otimes
b)f\bigr)(x):=f\bigl(\phi(b)xa\bigr),\,\,\,\forall\,a,b,x\in{\U}_{\Q(q)}(\mathfrak{g}),
f\in\bigl({\U}_{\Q(q)}(\mathfrak{g})\bigr)^{\ast}.
$$

Let $W_m=W(C_m)$ be the Weyl group of type $C_m$. Let $w_0$ be the
longest element in $W_m$. If $\lam=(\lam_1,\cdots,\lam_m)\in X$,
then $w_0\lam=(-\lam_1,\cdots,-\lam_m)$. Let $k\in\Z^{\geq 0}$ and
$\lam\in\Lambda^{+}(m,k)$. Recall our definitions of $\bi_{\lam}$
given above Lemma \ref{reversetableau}. We have the following
observation.

\begin{cor} \label{highestweight} Let $k\in\Z^{\geq 0}$ and
$\lam\in\Lambda^{+}(m,k)$. Then
$\pi_C\bigl(T_q^{\lam}({\bi}_{\lam},{\bi}_{\lam})\bigr)$ is a weight
vector of weight $\bigl(\lam,\lam\bigr)$ satisfying
$$e_i\pi_C\bigl(T_q^{\lam}({\bi}_{\lam},{\bi}_{\lam})\bigr)=0=\pi_C\bigl(T_q^{\lam}({\bi}_{\lam},{\bi}_{\lam})\bigr)f_i,\,\,\,
\forall\,1\leq i\leq m. $$
\end{cor}

\begin{proof} Note that $\bi_{\lam}\in I_{\lam}^{<}$. We identify
$\pi_C\bigl(T_q^{\lam}({\bi}_{\lam},{\bi}_{\lam})\bigr)$ as an element in $\bigl({\U}_{\Q(q)}\bigr)^{\ast}$ via $\iota_A$.
Recall that the $\U_{\Q(q)}(\mathfrak{g}\oplus\mathfrak{g})$-structure on $\bigl({\U}_{\Q(q)}\bigr)^{\ast}$ comes from its
bi-$\U_{\Q(q)}(\mathfrak{g})$ structure.

We first look at the left $\U_{\Q(q)}$-action on
$\bigl({\U}_{\Q(q)}\bigr)^{\ast}$. Recall that $\pi_C$ is a
bialgebra map. For each integer $j\geq 1$, by Lemma
\ref{comultirule} and \cite[(20)]{Oe2}, we know that
$$\begin{aligned}
&\quad\,\,\,\Delta^{(j)}\Bigl(\pi_C\bigl(T_q^{\lam}(\bi_{\lam},\bi_{\lam})\bigr)\Bigr)\\
&=\sum_{\bh^{(1)},\cdots,\bh^{(j)}\in
I_{\lam}^{<}}\pi_C\bigl(T_q^{\lam}(\bi_{\lam},\bh^{(1)})\bigr)\otimes\pi_C\bigl(T_q^{\lam}(\bh^{(1)},\bh^{(2)})\bigr)
\otimes\cdots\\
&\qquad\qquad\otimes\pi_C\bigl(T_q^{\lam}(h^{(j)},\bi_{\lam})\bigr).
\end{aligned} $$
In particular, we have that $$
\pi_C\bigl(T_q^{\lam}(\bi_{\lam},\bi_{\lam})\bigr)=\sum_{\bh\in I_{\lam}^{<}}
\varepsilon\Bigl(\pi_C\bigl(T_q^{\lam}(\bh,\bi_{\lam})\bigr)\Bigr)\pi_C\bigl(T_q^{\lam}(\bi_{\lam},\bh)\bigr).
$$
With these in mind and using Lemma \ref{Apairing}, for any $f\in{\U}_{\Q(q)}$ and any integer $a$ with $1\leq a\leq m$,
we get that $$
\begin{aligned}&\quad\,\,\langle
k_a\pi_C\bigl(T_q^{\lam}(\bi_{\lam},\bi_{\lam})\bigr),f\rangle_0\\
&=\langle
\pi_C\bigl(T_q^{\lam}(\bi_{\lam},\bi_{\lam})\bigr),fk_a\rangle_0\\
&=\sum_{\bh\in
I_{\lam}^{<}}\langle\Delta\Bigl(\pi_C\bigl(T_q^{\lam}(\bi_{\lam},\bi_{\lam})\bigr)\Bigr),f\otimes
k_a\rangle_0\\
&=\sum_{\bh\in
I_{\lam}^{<}}\langle\pi_C\bigl(T_q^{\lam}(\bi_{\lam},\bh)\bigr),f\rangle_0
\langle\pi_C\bigl(T_q^{\lam}(\bh,\bi_{\lam})\bigr),k_a\rangle_0\\
&=q^{\langle\lam,\alpha_a^{\vee}\rangle}\sum_{\bh\in
I_{\lam}^{<}}\langle\varepsilon\Bigl(\pi_C\bigl(T_q^{\lam}(\bh,\bi_{\lam})\bigr)\Bigr)
\pi_C\bigl(T_q^{\lam}(\bi_{\lam},\bh)\bigr),f\rangle_0\\
&=q^{\langle\lam,\alpha_a^{\vee}\rangle}\langle
\pi_C\bigl(T_q^{\lam}(\bi_{\lam},\bi_{\lam})\bigr),f\rangle_0,
\end{aligned}$$
which implies that
$k_a\pi_C\bigl(T_q^{\lam}(\bi_{\lam},\bi_{\lam})\bigr)=q^{\langle\lam,\alpha_a^{\vee}\rangle}
\pi_C\bigl(T_q^{\lam}(\bi_{\lam},\bi_{\lam})\bigr)$. In a similar way, one can prove
that if we regard $\bigl({\U}_{\Q(q)}\bigr)^{\ast}$ as a right
$\U_{\Q(q)}$-module, then
$\pi_C\bigl(T_q^{\lam}(\bi_{\lam},\bi_{\lam})\bigr)$ is a weight
vector of weight $\lam$.

It remains to show that $e_i\pi_C\bigl(T_q^{\lam}(\bi_{\lam},\bi_{\lam})\bigr)=0=\pi_C\bigl(T_q^{\lam}(\bi_{\lam},\bi_{\lam})\bigr)f_i$ for any $1\leq i\leq m$. It suffices to show that for any $f\in{\U}_{\Q(q)}$,

\begin{equation}
\label{4twozero}
\langle
e_i\pi_C\bigl(T_q^{\lam}(\bi_{\lam},\bi_{\lam})\bigr),f\rangle_0=\langle
\pi_C\bigl(T_q^{\lam}(\bi_{\lam},\bi_{\lam})\bigr)f_i,f\rangle_0.\end{equation}
By definition, $$\begin{aligned}
&\langle
e_i\pi_C\bigl(T_q^{\lam}(\bi_{\lam},\bi_{\lam})\bigr),f\rangle_0=
\langle
\pi_C\bigl(T_q^{\lam}(\bi_{\lam},\bi_{\lam})\bigr),fe_i\rangle_0,\\
&\langle
\pi_C\bigl(T_q^{\lam}(\bi_{\lam},\bi_{\lam})\bigr)f_i,f\rangle_0=
\langle
\pi_C\bigl(T_q^{\lam}(\bi_{\lam},\bi_{\lam})\bigr),f_if\rangle_0.
\end{aligned}
$$
By direct verification, one can show that for any $\bh\in I_{\lam}^{<}$, $$
\langle
\pi_C\bigl(T_q^{\lam}(\bh,\bi_{\lam})\bigr),e_i\rangle_0=0=\langle
\pi_C\bigl(T_q^{\lam}(\bi_{\lam},\bh)\bigr),f_i\rangle_0,
$$
from which the equality (\ref{4twozero}) follows immediately. This completes the proof of the corollary.
\end{proof}

By Corollary \ref{codecomp} and the discussion above it, we see that
every simple $\widetilde{A}_{\Q(q)}^{sy}(2m)$ comodule comes from
the restriction of a simple ${\U}_{\Q(q)}$-module. For each $\lam\in
X^{+}$, let $V(\lam)$ (resp., $V^{r}(\lam)$) denotes the left
(resp., the right) simple module with highest weight $\lam$.
By Corollary \ref{ModuleBase}, Lemma \ref{Base2} and Lemma
\ref{highestweight}, it is easy to see that $\nabla(\lam)_{\Q(q)}$ is
identified with $V(\lam)$ as left ${\U}_{\Q(q)}$-module, and
$\nabla^r(\lam)_{\Q(q)}$ is identified with
$V^{r}(\lam)$ as right ${\U}_{\Q(q)}$-module. By \cite[Proposition 7.2.2]{Ka2}, we have a
Peter--Weyl decomposition
\begin{equation}\label{PeterWey}
A_q(\mathfrak{g})\cong\bigoplus_{\lam\in X^{+}}V^{r}(\lam)\otimes
V(\lam),
\end{equation}
from which the following result follows easily.

\begin{lem} \label{Qembed} With the above notations, we have that
$$\iota_{A}\Bigl(\widetilde{A}^{sy}_{\Q(q)}(2m)\Bigr)=
A_q(\mathfrak{g}).$$
\end{lem}

For later use, we denote by $\Phi_{\lam}$ the canonical embedding induced from the isomorphism (\ref{PeterWey})
from $V^{r}(\lam)\otimes V(\lam)$ into $A_q(\mathfrak{g})$ for each
$\lam\in X^{+}$.

\smallskip

In \cite{Ka2}, Kashiwara introduced a crystal basis
$B(A_q(\mathfrak{g}))$ of $A_q(\mathfrak{g})$. He proved that
$B(A_q(\mathfrak{g}))$ has the crystal structure $$
B(A_q(\mathfrak{g}))=\bigsqcup_{\lam\in X^{+}}\widetilde{B}(\lam),
$$
where $\widetilde{B}(\lam):=B^{r}(\lam)\otimes B(\lam)$, and
$B^{r}(\lam)$ (resp., $B(\lam)$) denotes the crystal basis of
$V^{r}(\lam)$ (resp., of $V(\lam)$). For each
$b\in\widetilde{B}(\lam)$, let $\widetilde{G}(b)$ be
the corresponding upper global crystal base of $A_q(\mathfrak{g})$. Recall that
$\mathfrak{g}=\mathfrak{sp}_{2m}$ in this paper. By the results in
\cite{Lu4}, $\dot{\U}_{\Q(q)}$ has also a crystal base
$\sqcup_{\lam\in X^{+}}\widetilde{B}(\lam)$ as well as a canonical
base (or lower global crystal base)
$\{\widehat{G}(b)|b\in\widetilde{B}(\lam),\lam\in X^{+}\}$. This
canonical base is an $\A$-basis of $\dot{\U}_{\A}$. There exists a
canonical coupling $\langle,\rangle_1$ between $A_q(\mathfrak{g})$
and $\dot{\U}_{\Q(q)}$ (cf. \cite[(10.1.1)]{Ka3}) defined by $$
\langle\Phi_{\lam}(u\otimes
v),P\rangle_1:=(u,Pv),\,\,\forall\,\lam\in X^{+}, u\in V^{r}(\lam),
v\in V(\lam), P\in\dot{\U}_{\Q(q)},
$$
where $(,)$ is the pairing between $V^{r}(\lam)$ and $V(\lam)$
introduced in \cite[(7.1.2)]{Ka2}. By the results in \cite{Ka3}, for
each $\lam\in X^{+}$, there is a bijection $\widetilde{\Psi}:
\widetilde{B}(\lam)\rightarrow\widetilde{B}(\lam)$, such that

\begin{equation}\label{dualbase}
\langle\widetilde{G}(b'),\widehat{G}(b)\rangle_1=\delta_{b,\widetilde{\Psi}(b')},\,\,\,\forall\,b,b'\in\widetilde{B}(\lam).
\end{equation}
In \cite{Ka2}, Kashiwara proved that the upper global crystal bases
$\{\widetilde{G}(b)|b\in\widetilde{B}(\lam),\lam\in X^{+}\}$ form a
$\Q[q,q^{-1}]$ basis of $A_q^{\Q}(\mathfrak{g})$, and he remarked
that it is actually an $\A$-basis of $A_q^{\Z}(\mathfrak{g})$. For the reader's convenience, we include a
proof here.

\begin{lem} With the above notations, we have that the elements in the following set $$
\Bigl\{\widetilde{G}(b)\Bigm|b\in\widetilde{B}(\lam),\lam\in
X^{+}\Bigr\}
$$
form an $\A$-basis of $A_q^{\Z}(\mathfrak{g})$.
\end{lem}

\begin{proof} First we show that for each $\lam\in X^{+}$, and each
$b\in\widetilde{B}(\lam)$, $$\widetilde{G}(b)\in
A_q^{\Z}(\mathfrak{g}).$$ Let $P\in\U_{\A}$. By definition, for any
$\mu\in X, \lam\in X^{+}, u\in V^{r}(\lam), v\in V(\lam)$, $$
\langle \Phi_{\lam}(u\otimes v),P1_{\mu}\rangle_1=(u,P1_{\mu}v).
$$
It follows that for any given $\lam\in X^{+}, u\in V^{r}(\lam), v\in
V(\lam)$, there are only finitely many $\mu\in X$ such that $$
\langle \Phi_{\lam}(u\otimes v),P1_{\mu}\rangle_1\neq 0.
$$
Therefore, applying \cite[Proposition 7.2.2]{Ka2} and
(\ref{dualbase}), we get that
$$ \langle \widetilde{G}(b),P\rangle=\sum_{\mu\in X}\langle
\widetilde{G}(b),P1_{\mu}\rangle_1\in\A,
$$
which implies that $\widetilde{G}(b)\in A_q^{\Z}(\mathfrak{g})$.

Second, we want to show that $$
A_q^{\Z}(\mathfrak{g})\subseteq\Bigl\{x\in
A_q(\mathfrak{g})\Bigm|\langle x,\dot{\U}_{\A}\rangle_1\in\A\Bigr\},
$$
from which and together with (\ref{dualbase}) the lemma would
follow immediately.

Let $f=\sum_{i=1}^{t}x_i\otimes y_i\in A_q^{\Z}(\mathfrak{g})$, where for each $1\leq i\leq t$,
$x_i\in V^{r}(\lam^{(i)})$, $y_i\in V(\lam^{(i)})$, and $\lam^{(i)}\in X^{+}$.
Let $P\in\U_{\A}$, $\mu\in X$. Our purpose is to show that $\langle
f,P1_{\mu}\rangle_1\in\A$. Let $n$ be an integer
which is large enough such that
\begin{enumerate}
\item $|\mu|\leq n$, $|\mu|\!\!\equiv n\pmod{2\Z}$, and
\item for each $1\leq i\leq t$, either
$V(\lam^{(i)})\subseteq V^{\otimes n-1}$ or $V(\lam^{(i)})\subseteq
V^{\otimes n}$.
\end{enumerate}
We write $\{1,2,\cdots,t\}=I_1\sqcup I_2$, where $$
I_1:=\bigl\{1\leq i\leq t\bigm|V(\lam^{(i)})\subseteq V^{\otimes
n}\bigr\},\,\,\,I_2:=\bigl\{1\leq i\leq
t\bigm|V(\lam^{(i)})\subseteq V^{\otimes n-1}\bigr\}.
$$
Let $p'_{\mu}\in\U_{\A}$ be as defined in the proof of Lemma
\ref{keylem} with respect to our fixed $n$ and $\mu$. By the definition of
$A_q^{\Z}(\mathfrak{g})$, we have that $\langle f, Pp'_{\mu}\rangle\in\A$.
Note that $p'_{\mu}V^{\otimes n-1}=0=1_{\mu}V^{\otimes n-1}$. It follows that $$\begin{aligned}
\langle f,P1_{\mu}\rangle_1
&=\sum_{i\in I_1}(x_i, P1_{\mu}y_i)
=\sum_{i\in I_1}(x_i, Pp'_{\mu}y_i)=\sum_{i\in
I_1\sqcup I_2}(x_i, Pp'_{\mu}y_i)\\
&=\langle f, Pp'_{\mu}\rangle\in\A, \end{aligned}$$
as required. This completes the proof of the lemma.
\end{proof}

Let $\lam\in X^{+}$. Let $\Delta_{\A}(\lam)$ denote the standard
$\A$-form of $V(\lam)$, i.e., the $\U_{\A}(\mathfrak{g})$-submodule
generated by the highest weight vector $v_{\lam}$. Then
$\Delta_{\A}(\lam)$ is spanned by  Lusztig's {\it canonical basis}
as in \cite[\S14.4]{Lu3}. Note that the upper global crystal basis
$\bigl\{G(b)\bigm|b\in B(\lam)\bigr\}$ is Lusztig's dual canonical
basis. The dual basis to the upper global crystal basis under the
canonical contravariant form $(.,.)$ on $V(\lam)$ is the {\it lower
global crystal basis}, i.e., Lusztig's canonical basis, cf.
\cite{GLu} and \cite[\S3.3, 4.2.1]{Ka2}. Let $$
V_{\A}(\lam):=\bigl\{v\in V(\lam)\bigm|\text{$(v,w)\in\A$ for all
$w\in\Delta_{\A}(\lam)$}\bigr\}.
$$
Then $V_{\A}(\lam)\cong\Hom_{\A}(V_{\A},\A))$ is  $\U_{\A}(\mathfrak{g})$-stable and spanned by the upper global
crystal basis of $V(\lam)$.
Similar statement is true for $V^r_{\A}(\lam), V^r(\lam)$. For any field $K$ which is an $\A$-algebra, we define
$$\begin{aligned}V^r_{K}(\lam)\otimes V_K(\lam)&:=(V^r_{\A}(\lam)\otimes V_{\A}(\lam))\otimes_{\A}K,\\
\Delta^r_{K}(\lam)\otimes \Delta_K(\lam)&:=(\Delta^r_{\A}(\lam)\otimes\Delta_{\A}(\lam))\otimes_{\A}K.
\end{aligned}
$$
Since $\Delta^r_{K}(\lam)\otimes \Delta_K(\lam)$ is a highest weight module generated by its highest weight vector
(cf. \cite[Theorem 14.4.11]{Lu3}) and has the same dimension as $V^r(\lam)\otimes V(\lam)$, it follows that
$\Delta^r_{K}(\lam)\otimes \Delta_K(\lam)$ is isomorphic to the Weyl module of $\U_{K}(\mathfrak{g}\oplus\mathfrak{g})$
associated to $(\lam,\lam)$. Note that $V^r_{K}(\lam)\otimes V_K(\lam)\cong\bigl(\Delta^r_{K}(\lam)\otimes
\Delta_K(\lam)\bigr)^{\ast}$. Therefore, we have

\begin{lem} \label{coWeyl1} Let $\lam\in X^{+}$. For each field $K$ which is an $\A$-algebra, $V^{r}_{K}(\lam)\otimes V_{K}(\lam)$ is isomorphic to the co-Weyl module of $\U_{K}(\mathfrak{g}\oplus\mathfrak{g})$ associated to $(\lam,\lam)$.
\end{lem}

We recall the Bruhat order ``$<$" on the set $X^{+}$ of dominant weights.
Namely, $\lam<\mu$ if and only if $\mu-\lam\in\sum_{i=1}^{m}\Z^{\geq
0}\alpha_i$.  Note that $\lam<\mu$ implies that $\lam\lhd\mu$, where $``\lhd"$ is the usual dominance order defined on the set of partitions (cf. \cite{DJ1}). In particular, $|\lam|\leq|\mu|$. If $|\lam|<|\mu|$, then $\lam<\mu$ implies that
$|\mu|-|\lam|\in 2\mathbb{N}$ and hence $\lam\succ\mu$; if $|\lam|=|\mu|$, then $\lam\lhd\mu$ implies that
$\lam^{t}\rhd\mu^{t}$, which also implies that
$\lam^t$ is bigger than $\mu^t$ under the lexicographical order, hence we still have $\lam\succ\mu$. For each $\lam\in X^{+}$, we define
$$\begin{aligned}
A_q^{\Z}(\mathfrak{g})^{\leq\lam}&:=\text{$\A$-Span}\Bigl\{\widetilde{G}(b)\Bigm|\lam\geq\mu\in
X^{+},b\in\widetilde{B}(\mu)\Bigr\},\\
A_q^{\Z}(\mathfrak{g})^{<\lam}&:=\text{$\A$-Span}\Bigl\{\widetilde{G}(b)\Bigm|\lam>\mu\in
X^{+},b\in\widetilde{B}(\mu)\Bigr\},
\end{aligned}
$$
and
$A_q(\mathfrak{g})^{\leq\lam}:=A_q^{\Z}(\mathfrak{g})^{\leq\lam}\otimes_{\A}\Q(q)$,
$A_q(\mathfrak{g})^{<\lam}:=A_q^{\Z}(\mathfrak{g})^{<\lam}\otimes_{\A}\Q(q)$. For each
$b\in\widetilde{B}(\lam)$, let $G(b)$ be
the corresponding upper global crystal base of $V^{r}(\lam)\otimes
V(\lam)$. By construction, we know that $V^{r}(\lam)\otimes
V(\lam)$ is spanned by $\bigl\{G(b)\bigm|b\in\widetilde{B}(\lam)\bigr\}$.
The following result is implicit in the proof of \cite[Section
5,6]{Ka2}. The author is indebted to Professor Kashiwara for
pointing out this to him.

\begin{lem} \label{basemodule1} {\rm (\cite{Ka2})} With the decomposition (\ref{PeterWey}) in mind, we have
that \begin{enumerate}
\item for each $\lam\in X^{+}, b\in\widetilde{B}(\lam)$, $$
\widetilde{G}(b)\in G(b)+\sum_{\substack{b'\in\widetilde{B}(\mu)\\
\lam>\mu\in X^{+}}}\Q(q)G(b');
$$
\item for each $\lam\in X^{+}$, $$ A_q(\mathfrak{g})^{\leq\lam}\cong\oplus_{\lam\geq\mu\in X^{+}}V^{r}(\mu)\otimes
V(\mu),\quad A_q(\mathfrak{g})^{<\lam}\cong\oplus_{\lam>\mu\in
X^{+}}V^{r}(\mu)\otimes V(\mu). $$ In particular,
$A_q(\mathfrak{g})^{\leq\lam}/A_q(\mathfrak{g})^{<\lam}\cong
V^{r}(\lam)\otimes V(\lam)$;
\item both $A_q^{\Z}(\mathfrak{g})^{\leq\lam}$ and $A_q^{\Z}(\mathfrak{g})^{<\lam}$ are
${\U}_{\A}(\mathfrak{g}\oplus\mathfrak{g})$-stable.\end{enumerate}
\end{lem}

For each integer $k\geq 0$, we define $$
A_q^{\Z}(\mathfrak{g})^{\leq
k}:=\text{$\A$-Span}\Bigl\{\widetilde{G}(b)\Bigm|b\in\widetilde{B}(\mu),
\mu\in \Lambda^{+}(m,\leq\! k)\Bigr\}.
$$
Then it is clear that $$ A_q^{\Z}(\mathfrak{g})^{\leq
k}=\sum_{\lam\in\Lambda^{+}(m,\leq
k)}A_q^{\Z}(\mathfrak{g})^{\leq\lam}.
$$
In particular, $A_q^{\Z}(\mathfrak{g})^{\leq k}$ is
${\U}_{\A}(\mathfrak{g}\oplus\mathfrak{g})$-stable. Recall that for
any $\lam\in X^{+}$, $w_0\lam=(-\lam_1,\cdots,-\lam_m)$. Recall also our definitions of
$\nabla^r_K(\lam)$ and $\nabla_K(\lam)$ in Section 3. When
regarded as a left $\U_K(\mathfrak{g}\oplus\mathfrak{g})$-module,
$\nabla^r_K(\lam)\otimes\nabla_K(\lam)$ is isomorphic to
$\nabla_K(\lam)\otimes\nabla_K(\lam)$.

\begin{lem} \label{coWeyl2} Let $K$ be a field which is an $\A$-algebra and $n\geq 0$ be an integer. Let $\lam\in\Lambda^{+}(m,n)$.
If $N$ is a nonzero $\U_K(\mathfrak{g}\oplus\mathfrak{g})$-submodule
of $\nabla^r_K(\lam)\otimes\nabla_K(\lam)$, then $
\pr_{\lam}\Bigl(\pi_{C}\bigl(T_q^{\lam}(\bi_{\lam},\bi_{\lam})\bigr)\Bigr)\otimes_{\A}\!1_{K}\in
N$. In particular, $\nabla^r_K(\lam)\otimes\nabla_K(\lam)$ is
isomorphic to the co-Weyl module of\,
$\U_K(\mathfrak{g}\oplus\mathfrak{g})$ associated to $(\lam,\lam)$.
\end{lem}

\begin{proof} We divide the proof into four steps:\smallskip

{\it Step 1.} First, we claim that for any integer $t\geq 0$ and any
$\bi,\bj\in I_{\lam}^{<}\cup I_{\lam}^{mys}$, $$\begin{aligned}
e_i^{(t)}T_q^{\lam}(\bi,\bj)&=\sum_{\substack{\bh\in I_{\lam}^{<}\\
\bh\overset{(i,t)}{\rightsquigarrow}\bj}}(-1)^{a'(\bh,\bj)}q^{a(\bh,\bj)}T_q^{\lam}(\bi,\bh),\\
f_i^{(t)}T_q^{\lam}(\bi,\bj)&=\sum_{\substack{\bh\in I_{\lam}^{<}\\
\bh\overset{(i,t)}{\rightsquigarrow}\bj}}(-1)^{b'(\bh,\bj)}q^{b(\bh,\bj)}T_q^{\lam}(\bi,\bh),\\
T_q^{\lam}(\bi,\bj)f_i^{(t)}&=\sum_{\substack{\bh\in I_{\lam}^{<}\\
\bh\overset{(i,t)}{\rightsquigarrow}\bi}}(-1)^{c'(\bi,\bh)}q^{c(\bi,\bh)}T_q^{\lam}(\bh,\bj),\\
T_q^{\lam}(\bi,\bj)e_i^{(t)}&=\sum_{\substack{\bh\in I_{\lam}^{<}\\
\bh\overset{(i,t)}{\rightsquigarrow}\bi}}(-1)^{d'(\bi,\bh)}q^{d(\bi,\bh)}T_q^{\lam}(\bh,\bj),
\end{aligned}
$$
where $a'(\bh,\bj),a(\bh,\bj),b'(\bi,\bh),b(\bi,\bh),c'(\bh,\bj),c(\bh,\bj),d'(\bi,\bh),d(\bi,\bh)
\in\Z$ and $\bh\overset{(i,t)}{\rightsquigarrow}\bi$ means that $\bh$ differs
from $\bi$ on exactly $t$ indices, on which each index $i$ is changed into $i+1$ or $i-1$.

We only prove the first equality as the others can be proved in a
similar way. By the same argument used in the proof of Lemma
\ref{comultirule}, we can assume without loss of generality that
$\bi,\bj\in I_{\lam}^{<}$. For any $f\in\U_{\Q(q)}$, we have that $$
\langle e_i^{(t)}T_q^{\lam}(\bi,\bj),f\rangle_0=\langle
T_q^{\lam}(\bi,\bj),fe_i^{(t)}\rangle_0 =\sum_{\bh\in
I_{\lam}^{<}}\langle T_q^{\lam}(\bi,\bh),f\rangle_0 \langle
T_q^{\lam}(\bh,\bj),e_i^{(t)}\rangle_0.
$$
It suffices to show that $\langle T_q^{\lam}(\bh,\bj),e_i^{(t)}\rangle_0\neq 0$ only if $\bh\overset{(i,t)}{\rightsquigarrow}\bj$, and in that case it is equal to $(-1)^{a'(\bh,\bj)}q^{a(\bh,\bj)}$ for some
$a'(\bh,\bj),a(\bh,\bj)\in\Z$.

By definition, $$
\langle T_q^{\lam}(\bh,\bj),e_i^{(t)}\rangle_0=\sum_{\substack{\bk\in I(2m,n)\\ w\in\BS_{\lam^t}}}
(-q^2)^{-\ell(w)}\beta(w)_{\bh,\bk}u_{\bk}\Bigl(e_i^{(t)}v_{\bj}\Bigr).
$$
Note that $\bh,\bj\in I_{\lam}^{<}$. Recall the definition of $e_i^{(t)}$ (cf. \cite[3.1.5]{Lu3}) and the action of $e_i$ given in Section 2. To calculate the above sum, it suffices to consider only those $\bk\in I(2m,n)$ such that $\bk\overset{(i,t)}{\rightsquigarrow}\bj$ and the entries in each column of $T_{\bk}^{\lam}$ are weakly increasing from top to bottom. By the definition of $\beta(w)$, we know that $\beta(w)_{\bh,\bk}\neq 0$ only if each column of $\bh$ has the same set of entries as the corresponding  column of $\bk$. It follows that we can further restrict ourselves to those $\bk\in I_{\lam}^{<}$.
Now applying \cite[Lemma 11.8]{Oe2}, we deduce that only the case when $w=1, \bk=\bh$ can make contribution to the above sum, from which our claim follows immediately.
\smallskip

{\it Step 2.} We define $P(N):=\{(\mu,\nu)\in X\times X|N_{\mu,\nu}\neq 0\}$. We claim that for some $k_2,\cdots,k_{m}\in\Z$ and some $\nu\in X$, $(\nu,\lam_1\varepsilon_1+k_2\varepsilon_2+\cdots+k_{m}\varepsilon_{m})\in P(N)$.

Recall that $W_m$ is the Weyl group of type $C_m$. By \cite[Lemma
1.13]{APW}, we know that for any $\lam\in X$,
\begin{equation}\label{weightstable}
\text{$\lam\in P(N)$ implies that $(w_1\lam,w_2\lam)\in P(N)$ for
any $w_1,w_2\in W_m$.}
\end{equation}

Therefore, it is equivalent to show that for some $k_1,\cdots,k_{m-1}\in\Z$ and some $\nu\in X$, $(\nu,k_1\varepsilon_1+\cdots+k_{m-1}\varepsilon_{m-1}+\lam_1\varepsilon_m)\in P(N)$. For simplicity, for each $\bi,\bj\in I_{\lam}^{mys}$, we write $$
v(\bi,\bj):=\pr_{\lam}\Bigl(\pi_{C}\bigl(T_q^{\lam}(\bi,\bj)\bigr)\Bigr).
$$
Then the elements in $\{v(\bi,\bj)|\bi,\bj\in I_{\lam}^{mys}\}$ form a $K$-basis of $\nabla^r_K(\lam)\otimes\nabla_K(\lam)$.

Since $N$ is a submodule of $\nabla^r_K(\lam)\otimes\nabla_K(\lam)$,
any weight $(\nu,\mu)$ of $N$ satisfies $\lam\geq\nu, \lam\geq\mu$.
In particular, $\lam_1\geq \nu_1, \lam_1\geq \mu_1$. By
(\ref{weightstable}), we can assume without loss of generality that
for some integers $k_1,\cdots,k_{m}$ with $0<k_m\leq\lam_1$ and some
$\nu\in X$, $(\nu,k_1\varepsilon_1+\cdots+k_{m}\varepsilon_{m})\in
P(N)$. Furthermore, we assume that our $k_m$ is chosen such that
$k_m$ is as big as possible. For each weight vector $0\neq x\in N$
with weight $(\nu,k_1\varepsilon_1+\cdots+k_{m}\varepsilon_{m})$, we
can write $$ x=\sum_{\bi,\bj\in
I_{\lam}^{mys}}C_{\bi,\bj}v(\bi,\bj),
$$
for some $C_{\bi,\bj}\in K$. Set $$J(x)=\{\bj\in
I_{\lam}^{mys}|C_{\bi,\bj}\neq 0, \text{for some $\bi\in
I_{\lam}^{mys}$}\}.$$ For each $\bj\in J(x)$, the assumption $k_m>0$
and the fact that $\bj\in I_{\lam}^{mys}$ imply that $j_1=m$. Let
$\mathfrak{t}_{\lam}$ be the standard $\lam$-tableau such that the
numbers $1,2,\cdots,k$ entered in usual order along the successive
columns of $\lam$. We define $$\begin{aligned}
c_{\bj}&:=\#\bigl\{1\leq t\leq k\bigm|\text{$j_t=j_1$ and $t$ sits in the first row of $\mathfrak{t}_{\lam}$}\bigr\},\\
\widehat{c}_{\bj}&:=\#\bigl\{1\leq t\leq k\bigm|\text{$j_t=j'_1$ and $t$ sits in the first row of $\mathfrak{t}_{\lam}$}\bigr\}.
\end{aligned}
$$
We assert that $c_{\bj}=k_m$ and $\widehat{c}_{\bj}=0$, $\forall\,\bj\in
J(x)$. In fact, if this is not true, then we can find some $\bj_0\in
J(x)$ such that $0<\widehat{c}_{\bj_0}\geq \widehat{c}_{\bj}$ for any $\bj\in J(x)$.
It is easy to see that $e_m^{(\widehat{c}_{\bj_0})}x$ is a nonzero weight
vector with weight
$\biggl(\nu,\sum_{t=1}^{m-1}k_t\varepsilon_t+(k_{m}+2\widehat{c}_{\bj_0})\varepsilon_{m}\biggr)$,
a contradiction to the maximality of $k_m$.\smallskip

If $k_m=\lam_1$, then we are done. Henceforth we assume
$0<k_m<\lam_1$. For each $\bj\in J(x)$, we define
$b_{\bj}\in\{1,2,\cdots,2m\}$ to be the least integer (with respect
to the order $``\prec"$) in the first row of $T_{\bj}^{\lam}$ such
that $b_{\bj}\succ m$. In particular, $b_{\bj}\succ m'$ (because
$\widehat{c}_{\bj}=0$). We set $$J_1(x):=\bigl\{\bj\in
J(x)\bigm|b_{\bj}\preceq b_{\bh},\,\,\forall\,\bh\in
J(x)\bigr\}.$$ Let $b=b_{\bj}$ for any $\bj\in J_1(x)$. For any
$\bj\in J_1(x)$, we define $$ c''_{\bj}:=\#\bigl\{1\leq t\leq
n\bigm|\text{$j_t=b$ and $t$ sits in the first row of
$\mathfrak{t}_{\lam}$}\bigr\}.
$$
Let $$
J_2(x):=\bigl\{\bj\in J_1(x)\bigm|c''_{\bj}\geq c''_{\bh},\,\,\forall\,\bh\in J_1(x)\bigr\}.
$$
Let $c''=c''_{\bj}$ for any $\bj\in J_2(x)$. Then $c''>0$. We define
$$
z:=\begin{cases}f_m^{(c''+k_m)}f_{m-1}^{(c'')}\cdots f_{b+1}^{(c'')}f_{b}^{(c'')}, &\text{if $1\leq b<m$;}\\
e_m^{(k_m+c'')}\cdots e_{b'+1}^{(c'')}e_{b'}^{(c'')}f_m^{(k_m)},
&\text{if $m<b<2m$.}
\end{cases}
$$
Using the first two formulae we have given in Step 1 and the fact
that $T_q^{\lam}(\bi,\bh)=0$ whenever there are two identical
indices appearing in an adjacent position in a column of
$T_{\bh}^{\lam}$ (\cite[Corollary 9.2]{Oe2}) as well as the
definition of $I_{\lam}^{mys}$ and the action of each $e_i, f_i$
given in Section 2, we deduce that $zx\in N$ is a nonzero weight vector with weight equal to either
$$\biggl(\nu,\sum_{\substack{1\leq s\leq m\\ s\neq b, s\neq m}}k_s\varepsilon_s+
(k_b-c'')\varepsilon_b-(k_{m}+c'')\varepsilon_{m}\biggr),$$
or $$\biggl(\nu,\sum_{\substack{1\leq s\leq m\\ s\neq b', s\neq m}}k_s\varepsilon_s+
(k_{b'}+c'')\varepsilon_{b'}+(k_{m}+c'')\varepsilon_{m}\biggr).$$
In both case, applying (\ref{weightstable}) if necessary, we get a contradiction to
the maximality of $k_m$. This proves our claim.
\smallskip

{\it Step 3.} We claim that for any integer $1\leq t\leq m$, there
exist some integers $k''_{t+1},\cdots,k''_m\in\Z$ and some $\nu''\in
X$, such that
$\biggl(\nu'',\sum_{j=1}^{t}\lam_j\varepsilon_j+\sum_{s=t+1}^{m}k''_{s}\varepsilon_{s}\biggr)\in
P(N)$.\smallskip

We make induction on $t$. If $t=1$, this is true by the result
obtained in Step 2. Suppose that the claim is true for $t-1$. That
is, for some integers $k_{t},\cdots,k_m\in\Z$ and some $\nu\in X$,
we have that
$$(\nu,\widetilde{\lam}):=\biggl(\nu,\sum_{j=1}^{t-1}\lam_j\varepsilon_j+\sum_{s=t}^{m}k_{s}\varepsilon_{s}\biggr)\in
P(N).$$ Using (\ref{weightstable}), we can further assume that the weight $\widetilde{\lam}$ is chosen
such that \begin{equation}\label{max4}
\text{$k_m=\max\bigl\{|k_s|\bigm|t\leq s\leq m\bigr\}\geq 0$, and $k_m$ is as big as possible.}
\end{equation}
We are going to show that for some integers $\widehat{k}_{t},\cdots,\widehat{k}_{m-1}\in\Z$, $$
\biggl(\nu,\sum_{j=1}^{t-1}\lam_j\varepsilon_j+\lam_t\varepsilon_m+\sum_{s=t}^{m-1}\widehat{k}_{s}\varepsilon_{s}\biggr)\in
P(N). $$ If this is true, then we can apply (\ref{weightstable})
again to get our claim.

If $\ell(\lam)\leq t-1$, then there is nothing to prove. Henceforth, we assume that $\ell(\lam)\geq t$. In particular,
$\lam_t>0$. By (\ref{weightstable}), we know that
$(\nu,\sum_{j=1}^{t-1}\lam_j\varepsilon_j+k_m\varepsilon_t+\sum_{s=t+1}^{m}k_{s-1}\varepsilon_{s})\in P(N)$, which implies that
$\sum_{j=1}^{t-1}\lam_j\varepsilon_j+k_m\varepsilon_t+\sum_{s=t+1}^{m}k_{s-1}\varepsilon_{s}\leq\lam$. It follows that
$0\leq k_m\leq\lam_t$. We assert that $k_m=\lam_t$, from which our claim will follows immediately.

Suppose $k_m<\lam_t$. For each nonzero weight vector $x$ of weight $(\nu,\widetilde{\lam})$ in $N$, we write $$ x=\sum_{\bi,\bj\in I_{\lam}^{mys}}C_{\bi,\bj}(x)v(\bi,\bj).
$$
for some $C_{\bi,\bj}(x)\in K$. We define $J(x), c_{\bj}, \widehat{c}_{\bj}, b_{\bj}, J_1(x), b, c''_{\bj}, J_2(x), c''$ as in Step 2.
We first show that $k_m>0$. In fact, if $k_m=0$, then by (\ref{max4}) we know that $k_i=0$ for any
$t\leq i\leq m$. Using the definition of $I_{\lam}^{mys}$, it is easy to see that for each $\bj\in J(x)$, \begin{equation}\label{extre4}
\text{$\forall\,1\leq l\leq t-1$, $l$ appears $\lam_l$ times in $\bj$ and $l'$ does not appear in $\bj$.}
\end{equation}
Since $\lam_t>0$, it follows that for any $\bj\in J(x)$, $$j_1\in\{m,m',m-1,(m-1)',\cdots,t,t'\}.$$
If there exists $\bj\in J(x)$ such that $j_1=m$, then (as $k_m=0$) $b_{\bj}=m'$. In this case $b=m'$ and it is easy to see that
$e_m^{(c'')}x\in N$ is a nonzero weight vector of weight $(\nu,\sum_{j=1}^{t-1}\lam_j\varepsilon_j+2c''\varepsilon_m)$. Since $c''>0$, we get a contradiction to (\ref{max4}).

Henceforth we assume that $j_1\neq m$ for any $\bj\in J(x)$. In particular, $j_1\succ m'$.
Note that $j_1\preceq t'$. We define $j$ to be the unique integer such that $j=j_1$ for some $\bj\in J(x)$ and
$j\preceq h_1,\,\forall\,\bh\in J(x)$. Set $c=\max\{c_{\bj}|\bj\in J(x), j_1=j\}$. Then $c>0$.
If $t\leq j<m$ then it is easy to see that $f_{m}^{(c)}\cdots f_{j+1}^{(c)}f_j^{(c)}x\in N$ is a nonzero weight vector of weight $(\nu,\sum_{s=1}^{t-1}\lam_s\varepsilon_s-c\varepsilon_j-c\varepsilon_m)$; while if $m'<j\leq t'$ then it is easy to see that $e_{m}^{(c)}\cdots e_{j'+1}^{(c)}e_{j'}^{(c)}x\in N$ is a nonzero weight vector of weight $(\nu,\sum_{s=1}^{t-1}\lam_s\varepsilon_s+c\varepsilon_j+c\varepsilon_m)$. In both cases using (\ref{weightstable}), we get a contradiction to (\ref{max4}). Therefore, we must have $k_m>0$. In particular, for any $\bj\in J(x)$, $j_1=m$.

The remaining argument is similar to that used in Step 2 with some
slight modification. First, by the same argument used in Step 2, we
can show that $j_1=m$, $c_{\bj}=k_m$ and $\widehat{c}_{\bj}=0$. In
particular, $b\succ m'$. We claim that $b\preceq t'$. In fact, if
$b\succeq t-1$, then for any $\bj\in J(x)$, the first row of the
remaining tableau after deleting all the entries of $T_{\bj}^{\lam}$
in $\{t-1,(t-1)',\cdots,2,2',1,1'\}$ has length $k_m$. On the other
hand, we know that for any $\bj\in I_{\lam}^{mys}$ satisfying
(\ref{extre4}), the first row of the remaining tableau after
deleting all of the entries of $T_{\bj}^{\lam}$ in
$\{t-1,(t-1)',\cdots,2,2',1,1'\}$ must have length $\lam_1\geq
\lam_t$. It follows that $k_m=\lam_t$, a contradiction to our
assumption. This proves that $m'\prec b\preceq t'$.

Now we follow exactly the same argument used in Step 2 to define an element $z$. Note that the condition $m'\prec b\preceq t'$ ensures that $zx\neq 0$ is a weight vector in $N$ with weight $$
\biggl(\nu,\sum_{j=1}^{t-1}\lam_j\varepsilon_j+\sum_{s=t}^{m}\widehat{k}_{s}\varepsilon_{s}\biggr)
$$
such that $|\widehat{k}_m|>k_m$. Applying (\ref{weightstable}) if necessary, we get a contradiction to (\ref{max4}).
This proves our assertion that $k_m=\lam_t$.
\smallskip

By induction and set $t=m$, we get that for some $\nu\in X$, $(\nu, \lam)\in P(N)$.
\smallskip

{\it Step 4.} Starting from a nonzero weight vector $x\in N$ with weight $(\nu, \lam)$ and Using a similar argument as used in Step 3, we can prove that $(\lam,\lam)\in P(N)$. Note that the $(\lam,\lam)$-weight space of $\nabla^r_K(\lam)\otimes\nabla_K(\lam)$ is one-dimensional and is spanned by $
\pr_{\lam}\Bigl(\pi_{C}\bigl(T_q^{\lam}(\bi_{\lam},\bi_{\lam})\bigr)\Bigr)\otimes_{\A}1_{K}$.
As $N$ is a submodule of $\nabla^r_K(\lam)\otimes\nabla_K(\lam)$, we can deduce that $$
\pr_{\lam}\Bigl(\pi_{C}\bigl(T_q^{\lam}(\bi_{\lam},\bi_{\lam})\bigr)\Bigr)\otimes_{\A}1_{K}\in N.$$ This completes the proof of the lemma.
\end{proof}

\begin{cor} Let $K$ be a field which is an $\A$-algebra and $n\geq 0$ be an integer. Let $\lam\in\Lambda^{+}(m,n)$.
Then $\nabla_K(\lam)$ is isomorphic to the co-Weyl module of\,\,$\U_K(\mathfrak{g})$ associated to $\lam$.
\end{cor}

\begin{proof} If $L$ is a simple $\U_K(\mathfrak{g})$-submodule of $\nabla_K(\lam)$, then $L\otimes L$ is a simple
$\U_K(\mathfrak{g}\oplus\mathfrak{g})$-submodule of $\nabla^r_K(\lam)\otimes\nabla_K(\lam)$. By Lemma \ref{coWeyl2},
$\nabla^r_K(\lam)\otimes\nabla_K(\lam)$ is isomorphic to the co-Weyl module of $\U_K(\mathfrak{g}\oplus\mathfrak{g})$ associated to $(\lam,\lam)$. So it must have a unique simple $\U_K(\mathfrak{g}\oplus\mathfrak{g})$-socle. This implies that
$\nabla_K(\lam)$ also has a unique simple $\U_K(\mathfrak{g})$-socle. By the universal property of co-Weyl module, we see that there exists an embedding from $\nabla_K(\lam)$ into the co-Weyl module $V_K(\lam)$. Comparing their dimensions, we deduce that this embedding must be an isomorphism.
\end{proof}

\begin{thm} \label{coincide} With the above notations, we have that
$$\iota_{A}\Bigl(\widetilde{A}^{sy}_{\A}(2m)\Bigr)=
A_q^{\Z}(\mathfrak{g}),\,\,\,\iota_{A}\Bigl(\widetilde{A}^{sy}_{\A}(2m,\leq\! k)\Bigr)=A_q^{\Z}(\mathfrak{g})^{\leq k},\,\,\forall\,k\geq 0.
$$ In other words, the quantized coordinate
algebra defined by Kashiwara and the quantized coordinate algebra
$\widetilde{A}_{\A}^{sy}(2m)$ arising from a generalized FRT
construction are isomorphic to each other as $\A$-algebras.
Furthermore, we have the following commutative diagram $$
\begin{CD}
\dot{\U}_{\Q(q)} @>{\id}>>\dot{\U}_{\Q(q)} \\
@V{} VV @V{}VV\\
A_q(\mathfrak{g})^{\ast}@
>{\iota_{A}^{\ast}}>>
(\widetilde{A}^{sy}_{\Q(q)}(2m))^{\ast}
\end{CD},
$$
where the two vertical maps are induced by the two natural pairings
$\langle,\rangle_0$, $\langle,\rangle_1$ respectively.
\end{thm}

\begin{proof} We first show that $\iota_{A}\Bigl({A}^{sy}_{\A}(2m,\leq\! k)\Bigr)=
A_q^{\Z}(\mathfrak{g})^{\leq k}$ for any $k\geq 0$, from which the equality $\iota_{A}\Bigl(\widetilde{A}^{sy}_{\A}(2m)\Bigr)=A_q^{\Z}(\mathfrak{g})$ follows at once. We divide the proof into two steps:
\smallskip

\noindent {\it Step 1.} We claim that
$\iota_{A}\Bigl({A}^{sy}_{\A}(2m,\leq\! k)\Bigr)\subseteq
A_q^{\Z}(\mathfrak{g})^{\leq k}$.\smallskip

By Lemma \ref{Qembed} and the bimodules decomposition we have discussed before, we have $$
\iota_{A}\Bigl({A}^{sy}_{\A}(2m,\leq\! k)\Bigr)\subseteq\iota_{A}\Bigl({A}^{sy}_{\Q(q)}(2m,\leq\! k)\Bigr)\subseteq A_q(\mathfrak{g})^{\leq k}.
$$
For any integer $\lam\in\Lambda^{+}(m,\leq\! k)$, and any
$\bi,\bj\in I_{\lam}^{mys}$, it is easy to verify directly that
$$\langle \pi_C\bigl(T_q^{\lam}(\bi,\bj)\bigr),\U_{\A}\rangle_0\in\A.$$
It follows that
$\iota_{A}\Bigl({A}^{sy}_{\A}(2m,\leq\! k)\Bigr)\subseteq
A_q^{\Z}(\mathfrak{g})$. Hence by Lemma \ref{basemodule1}, $$
\iota_{A}\Bigl({A}^{sy}_{\A}(2m,\leq\! k)\Bigr)\subseteq
A_q(\mathfrak{g})^{\leq k}\cap A_q^{\Z}(\mathfrak{g})=A_q^{\Z}(\mathfrak{g})^{\leq k},
$$ as required.

\smallskip
\noindent {\it Step 2.} We now show that
$\iota_{A}\Bigl({A}^{sy}_{\A}(2m,\leq\! k)\Bigr)=
A_q^{\Z}(\mathfrak{g})^{\leq\! k}$. Our strategy is to show that for any field $K$ which is an $\A$-algebra, $
\iota_{K}:=\iota_{A}\otimes_{\A}1_K$ is an injection from ${A}^{sy}_{K}(2m,\leq\! k)$ into $
A_q^{K}(\mathfrak{g})^{\leq\! k}$.

For each $\lam\in\Lambda^{+}(m,k)$, let $b_{\lam}$ be the unique element in $\widetilde{B}(\lam)$ such
that $G(b_{\lam})\in V^{r}(\lam)\otimes V(\lam)$ is a highest weight
vector of weight $(\lam,\lam)$. Note that $$
e_iG(b_{\lam})=0,\,\,\,\forall\,e_i\in\U_{\Q(q)}(\mathfrak{g}\oplus\mathfrak{g}).
$$

We use induction on dominant weights with respect to the order ``$\prec$". If $\lam\in X^{+}$ is maximal
with respect to the order ``$\prec$", then $\lam$ is minimal
with respect to the order ``$<$". Then $$
\nabla^r(\lam)\otimes\nabla(\lam)\subseteq {A}^{sy}_{\A}(2m,\leq\! k),\,\,\,
V_{\A}^r(\lam)\otimes V_{\A}(\lam)\subseteq A_q^{\Z}(\mathfrak{g})^{\leq k}.
$$
By Lemma \ref{basemodule1}, we also know that $\widetilde{G}(b_{\lam})=G(b_{\lam})$.
In that case, as both $\widetilde{G}(b_{\lam})$ and
$\iota_{A}\Bigl(\pi_{C}\bigl(T_q^{\lam}(\bi_{\lam},\bi_{\lam})\bigr)\Bigr)$
are the highest weight vectors of
weight $(\lam,\lam)$ in $A_q^{\Z}(\mathfrak{g})$, we deduce (by
Corollary \ref{highestweight} and Lemma \ref{Base2}) that
$$ \widetilde{G}(b_{\lam})=\pm
q^{a}\iota_{A}\Bigl(\pi_{C}\bigl(T_q^{\lam}(\bi_{\lam},\bi_{\lam})\bigr)\Bigr),
$$ for some $a\in\Z$. Since $$
\iota_{A}\Bigl(\nabla^r(\lam)\otimes\nabla(\lam)\Bigr)\subseteq A_q^{\Z}(\mathfrak{g})^{\leq k}\bigcap
\bigl(V^{r}(\lam)\otimes V(\lam)\bigr)=V_{\A}^r(\lam)\otimes V_{\A}(\lam),
$$
we deduce that for any field $K$ which is an $\A$-algebra, $\iota$ induces a nonzero
$\U_K(\mathfrak{g}\oplus\mathfrak{g})$-homomorphism $\iota_K:=\iota_{A}\otimes_{\A}1_K$ from $\nabla^r_K(\lam)\otimes\nabla_K(\lam)$ to
$V_{K}^r(\lam)\otimes V_{K}(\lam)$. By Lemma \ref{coWeyl1} and Lemma \ref{coWeyl2}, we know that both modules are co-Weyl modules of $\U_K(\mathfrak{g}\oplus\mathfrak{g})$ associated to $(\lam,\lam)$. It follows that $\iota_K$ must always be an isomorphism. Hence $\iota_A$ must be an isomorphism as well. In particular, $$
\iota_{A}\Bigl(\nabla^r(\lam)\otimes\nabla(\lam)\Bigr)=V_{\A}^r(\lam)\otimes V_{\A}(\lam).
$$

In general, let $\lam\in X^{+}$, assume that for any field $K$ which is an $\A$-algebra,
$\iota_{K}$ is an injection from ${A}^{sy}_{K}(2m,\leq k)^{\succ\lam}$ into
$$A_q^{K}(\mathfrak{g})^{\leq k}(\ngeq\lam):=\sum_{\substack{\mu\in\Lambda^{+}(m,\leq k)\\
\mu\ngeq\lam}}A_q^{K}(\mathfrak{g})^{\leq\mu}.
$$
We want to prove that $\iota_{K}$ is also an injection from ${A}^{sy}_{K}(2m,\leq k)^{\succeq\lam}$
into $$A_q^{K}(\mathfrak{g})^{\leq k}(\ngtr\lam):=\sum_{\substack{\mu\in\Lambda^{+}(m,\leq k)\\
\mu\ngtr\lam}}A_q^{K}(\mathfrak{g})^{\leq\mu}.
$$
By bimodules decomposition, Lemma \ref{basemodule1} and definition, we know that $$
\iota_{\A}\Bigl({A}^{sy}_{\A}(2m,\leq k)^{\succeq\lam}\Bigr)\subseteq
A_q^{\Z}(\mathfrak{g})\bigcap A_q(\mathfrak{g})^{\leq k}(\ngtr\lam):=A_q^{\Z}(\mathfrak{g})^{\leq k}(\ngtr\lam).
$$
By the same argument as before, we know that $$ \widetilde{G}(b_{\lam})=\pm
q^{a}\iota_{A}\Bigl(\pi_{C}\bigl(T_q^{\lam}(\bi_{\lam},\bi_{\lam})\bigr)\Bigr),
$$
for some $a\in\Z$.
Therefore, $\iota_{\A}$ induces a nonzero $\U_{\A}(\mathfrak{g}\oplus\mathfrak{g})$-homomorphism $\overline{\iota}_{\A}$ from $$
\nabla^r(\lam)\otimes\nabla(\lam)\cong {A}^{sy}_{\A}(2m,\leq k)^{\succeq\lam}/{A}^{sy}_{\A}(2m,
\leq k)^{\succ\lam}$$ to
$$
V_{\A}^r(\lam)\otimes V_{\A}(\lam)\cong A_q^{\Z}(\mathfrak{g})^{\leq k}(\ngtr\lam)/A_q^{\Z}(\mathfrak{g})^{\leq k}(\ngeq\lam).
$$
For any field $K$ which is an $\A$-algebra, we get by base change a nonzero
$\U_K(\mathfrak{g}\oplus\mathfrak{g})$-homomorphism $\overline{\iota}_K$ from $$
\nabla^r_K(\lam)\otimes\nabla_K(\lam)\cong {A}^{sy}_{K}(2m,\leq k)^{\succeq\lam}/{A}^{sy}_{K}(2m,
\leq k)^{\succ\lam}$$ to
$$
V_{K}^r(\lam)\otimes V_{K}(\lam)\cong A_q^{K}(\mathfrak{g})^{\leq k}(\ngtr\lam)/A_q^{K}(\mathfrak{g})^{\leq k}(\ngeq\lam).
$$
By Lemma \ref{coWeyl1} and Lemma \ref{coWeyl2}, we know that both modules are co-Weyl modules of $\U_K(\mathfrak{g}\oplus\mathfrak{g})$ associated to $(\lam,\lam)$. It follows that $\overline{\iota}_K$ must always be an isomorphism. It follows that ${\iota}_K$ must always be an injection, as required.

By induction, we know that for any field $K$ which is an $\A$-algebra,
$\iota_{K}$ is an injection from ${A}^{sy}_{K}(2m,\leq\! k)$ into $
A_q^{K}(\mathfrak{g})^{\leq\! k}$. Comparing their dimensions, we can deduce that it must be an isomorphism, from which
we can deduce that $\iota_{A}$ must be an isomorphism as well. This proves the first part of this Theorem.
\smallskip

It remains to prove the commutativity of the diagram. Note that as $\dot\U_{\Q(q)}(\mathfrak{g}\oplus\mathfrak{g})$-module, $\widetilde{A}^{sy}_{\Q(q)}(2m)$ is generated by the elements $\pi_C\bigl(T_q^{\lam}(i_{\lam},i_{\lam})\bigr)$, $\lam\in\Lambda^{+}(m)$. Therefore, it suffices to show that for any $\lam\in\Lambda^{+}(m,k)$, $k\geq 0$ and any $P\in\U_{\A}, \mu\in X$, $\langle\pi_C\bigl(T_q^{\lam}(\bi_{\lam},\bi_{\lam})\bigr), P1_{\mu}\rangle_0=\langle\pi_C\bigl(T_q^{\lam}(\bi_{\lam},\bi_{\lam})\bigr), P1_{\mu}\rangle_1$.

Note that there exists also a canonical coupling $\langle,\rangle_1$ between $A_q(\mathfrak{g})$
and ${\U}_{\Q(q)}$ (cf. \cite{Ka3}) defined by $$
\langle\Phi_{\lam}(u\otimes
v),P\rangle_1:=(u,Pv),\,\,\forall\,\lam\in X^{+}, u\in V^{r}(\lam),
v\in V(\lam), P\in{\U}_{\Q(q)},
$$
where $(,)$ is the pairing between $V^{r}(\lam)$ and $V(\lam)$
introduced in \cite[(7.1.2)]{Ka2}. By direct verification, we see that $$\begin{aligned}
\langle \pi_C\bigl(T_q^{\lam}(\bi_{\lam},\bi_{\lam})\bigr), P1_{\mu}\rangle_0&=\delta_{\lam,\mu}\langle \pi_C\bigl(T_q^{\lam}(\bi_{\lam},\bi_{\lam})\bigr), P\rangle_0,\\
\langle \iota_A\bigl(\pi_C\bigl(T_q^{\lam}(\bi_{\lam},\bi_{\lam})\bigr)\bigr), P1_{\mu}\rangle_1&=\delta_{\lam,\mu}\langle \iota_A\bigl(\pi_C\bigl(T_q^{\lam}(\bi_{\lam},\bi_{\lam})\bigr)\bigr), P\rangle_1.
\end{aligned}$$
Therefore, it suffices to show that $$
\pi_C\bigl(T_q^{\lam}(\bi_{\lam},\bi_{\lam})\bigr), P\rangle_0=
\langle \iota_A\bigl(\pi_C\bigl(T_q^{\lam}(\bi_{\lam},\bi_{\lam})\bigr)\bigr), P\rangle_1. $$
Using the PBW basis of $\U_{\Q(q)}$ and Lemma \ref{highestweight}, we can reduce to the proof to the case where $P$ is generated by $k_1^{\pm 1},\cdots,k_m^{\pm 1}$. In that case, the proof follows from an easy verification. This completes the proof of the Theorem.\end{proof}

We remark that each integer $k\geq 0$, the dual of the
$\U_{\Q(q)}\bigl(\mathfrak{g}\oplus\mathfrak{g}\bigr)$-module $A_q(\mathfrak{g})^{\leq k}$
together with the dual basis of its upper global crystal basis actually forms a based module
in the sense of \cite[27.1.2]{Lu3}.\medskip

Henceforth, we shall identify $A_q^{\Z}(\mathfrak{g})$ with
$\widetilde{A}_{\A}^{sy}(2m)$ via $\iota_A$. By Theorem \ref{coincide}, the quantized coordinate algebra
$A_q^{\Z}(\mathfrak{g})$ was
equipped with two bases. One is
$\bigl\{\widetilde{G}(b)\bigr\}_{\lam\in
X^{+},\,\,b\in\widetilde{B}(b)}$, another is
$\bigl\{\pi_C(T_q^{\lam}(\bi,\bj))\bigr\}_{\lam\in
X^{+},\,\,\bi,\bj\in I_{\lam}^{mys}}$. The transition matrix between
these two bases must be invertible as a matrix over $\A$. Combining
this with (\ref{dualbase}), we can find an $\A$-basis
$\bigl\{\widehat{G}_{\bi,\bj}^{\lam}\bigr\}_{\lam\in
X^{+},\,\,\bi,\bj\in I_{\lam}^{mys}}$ of $\dot{\U}_{\A}$, such that
\begin{equation}\label{newbase}
\langle\pi_C\bigl(T_q^{\mu}(\bk,\bl)\bigr),\widehat{G}_{\bi,\bj}^{\lam}\rangle=\begin{cases}
1, &\text{if $\lam=\mu,\bi=\bk,\bj=\bl$,}\\
0, &\text{otherwise.}
\end{cases}
\end{equation}
for any $\lam,\mu\in X^{+}$, $\bi,\bj\in I_{\lam}^{mys}$ and
$\bk,\bl\in I_{\mu}^{mys}$.
\bigskip\bigskip

\noindent {\bf Proof of Theorem \ref{mainthm1}:} For each integer
$0\leq l\leq [n/2]$ and each $\lam\in\Lambda^{+}(m,n-2l)$,
$\bi,\bj\in I_{\lam}^{mys}$, we use
$\Bigl(D_{\bi,\bj}^{\lam,l}\Bigr)^{\ast}$ to denote the base element
of $S^{sy}_{\A}(2m,n)$ dual to the base element $D_{\bi,\bj}^{\lam,l}$
of $A^{sy}_{\A}(2m,n)$. We have the following commutative diagram:
$$
\begin{CD}
\dot{\U}_{\Q(q)}
@>{\widetilde{\psi}_C}>>S^{sy}_{\Q(q)}(2m,n) @>{\id}>>S^{sy}_{\Q(q)}(2m,n)\\
@V{\text{$\widetilde{\iota}_U$}} VV  @.  @V{\wr}VV \\
\Bigl(\widetilde{A}^{sy}_{\Q(q)}(2m)\Bigr)^{\ast}
@>{\pi_{C}^{\ast}}>> \Bigl(A^{sy}_{\Q(q)}(2m)\Bigr)^{\ast} @>{}>>
\Bigl(A^{sy}_{\Q(q)}(2m,n)\Bigr)^{\ast}
\end{CD}
$$
By Theorem \ref{coincide}, (\ref{newbase}) and the fact $$ \langle
d_qf, \pi_C^{\ast}\widetilde{\iota}_U(u)\rangle_0=\langle \pi_C(f),
\widetilde{\iota}_U(u)\rangle_0,\quad\forall\,u\in\dot{\U}_{\Q(q)},f\in
A^{sy}_{\Q(q)}(2m),
$$
we deduce that $$
\widetilde{\psi}_{C}\bigl(\widehat{G}_{\bi,\bj}^{\lam}\bigr)=\begin{cases}
\Bigl(D_{\bi,\bj}^{\lam,l}\Bigr)^{\ast},&\text{if $|\lam|\leq n$ and $2l:=n-|\lam|$ is even;}\\
0, &\text{otherwise.}
\end{cases}
$$
In particular, this shows that
$\widetilde{\psi}_{C}\bigl(\dot{\U}_{\A}\bigr)=S^{sy}_{\A}(2m,n)$.
By base change, we know that for any commutative $\A$-algebra $K$,
$$\Bigl(\widetilde{\psi}_{C}\downarrow_{{\dot\U}_{\A}}\otimes_{\A}K\Bigr)\bigl(\dot{\U}_{K}\bigr)=S^{sy}_{K}(2m,n).$$
This completes the proof of Theorem \ref{mainthm1}.

\bigskip
\begin{cor} \label{kernelbasis} With the above notations, we have that $$\begin{aligned}
\Ker\bigl(\widetilde{\psi}_{C}\downarrow_{\dot{\U}_{\A}}\bigr)&=\text{$\A$-Span}\Bigl\{\widehat{G}_{\bi,\bj}^{\lam}\Bigm|\begin{matrix}
\text{$\bi,\bj\in I_{\lam}^{mys}$, $\lam\in\Lambda^{+}(m,k)$,
$k>n$}\\\text{or\,\,\,$k<n$ and $n-k$ is odd}\end{matrix}\Bigr\},\\
&=\text{$\A$-Span}\Bigl\{\widehat{G}(b)\Bigm|\begin{matrix}\text{$b\in\widetilde{B}(\lam),
\lam\in\Lambda^{+}(m,k)$, $k>n$}\\\text{or\,\,\,$k<n$ and $n-k$ is
odd}\end{matrix}\Bigr\}. \end{aligned}$$ As a result, this is still
true if we replace $\A$ by any commutative $\A$-algebra $K$.
\end{cor}

\begin{proof} It suffices to prove the second statement. Let $\pi$ be the set of dominant weights in $V^{\otimes n}$.
Let $\lsub{\bS}{\Q(q)}(\pi)$ be the generalized Schur algebra associated to $\pi$ defined by Doty \cite{Dt2}. Then it is easy to check that the homomorphism $\widetilde{\psi}_C$ factors through $\lsub{\bS}{\Q(q)}(\pi)$. Let $\lsub{\bS}{\A}(\pi)$ be the
$\A$-form of $\lsub{\bS}{\Q(q)}(\pi)$ defined in \cite{Dt2}. For any field $K$ which is an $\A$-algebra, let
$\lsub{\bS}{K}(\pi):=\lsub{\bS}{\A}(\pi)\otimes_{\A}K$. Applying Theorem \ref{mainthm1} and comparing dimensions we deduce that the natural homomorphism from $\lsub{\bS}{K}(\pi)$ to $\End_{\bb_n(\zeta^{2m+1},\zeta)}\bigl(V_K^{\otimes n}\bigr)$ is an isomorphism. So the same is true if we replace $K$ by $\A$. Now the second statement follows directly from
the definition of $\lsub{\bS}{\A}(\pi)$ given in \cite{Dt2}.
\end{proof}

Note that in the proof of the above corollary, we have also given a proof of
Corollary \ref{maincor}.

\bigskip\bigskip
\bigskip
\section{Proof of Theorem \ref{mainthm2} in the case where $m\geq n$}

The purpose of this and the next section is to give a proof of
Theorem \ref{mainthm2}. Before starting the proof, we make some
reduction. By the results in \cite{Oe2}, we know that the symplectic
$q$-Schur algebra is stable under base change. That is, for any
commutative $\A$-algebra $K$, there is a canonical isomorphism
$$ S_{\A}^{sy}(2m,n)\otimes_{\A}K\cong S_{K}^{sy}(2m,n).
$$
Furthermore, $S_{\A}^{sy}(2m,n)$ is an integral quasi-hereditary
algebra. For any field $K$ which is an $\A$-algebra,
$V_K\cong\Delta_K(\varepsilon_1)\cong\nabla_K(\varepsilon_1)\cong
L_K(\varepsilon_1)$ is a tilting module over $S_{K}^{sy}(2m,n)$. It
follows that $V_{K}^{\otimes n}$ is also a tilting module over
$S_{K}^{sy}(2m,n)$. Applying Theorem \ref{mainthm1} and using
\cite[Lemma 4.4 (c)]{DPS}, we get that
$$\begin{aligned} \End_{\U_{\A}}\Bigl(V_{\A}^{\otimes
n}\Bigr)\otimes_{\A}K&=
\End_{S_{\A}^{sy}(2m,n)}\Bigl(V_{\A}^{\otimes
n}\Bigr)\otimes_{\A}K\\
&\cong\End_{S_{K}^{sy}(2m,n)}\Bigl(V_{K}^{\otimes
n}\Bigr)=\End_{\U_{K}}\Bigl(V_{K}^{\otimes n}\Bigr).
\end{aligned}$$
In other words, the endomorphism algebra
$\End_{\U_{K}}\Bigl(V_{K}^{\otimes n}\Bigr)$ is stable under base
change. Therefore, to prove Theorem \ref{mainthm2}, it suffices to
show that the natural homomorphism from
$(\bb_n(-q^{2m+1},q)_{\A})^{\rm op}$ to
$\End_{\U_{\A}}\Bigl(V_{\A}^{\otimes n}\Bigr)$ is surjective.
Equivalently, it suffices to prove this is true with $\A$ replaced
by any field $K$ which is an $\A$-algebra.\smallskip

In this section we shall give a proof of Theorem \ref{mainthm2} in
the case where $m\geq n$. {\it Henceforth, we shall assume that $K$
is field, and $\zeta$ is the image of $q$ in $K$ and $m\geq n$.}
Note that, in this case, by \cite[Proposition 4.2]{Ha},
$$\begin{aligned}
&\quad\dim\End_{\U_{K}}\Bigl(V_{K}^{\otimes
n}\Bigr)=\End_{\U_{\Q(q)}}\Bigl(V_{\Q(q)}^{\otimes n}\Bigr)\\
&=\sum_{\substack{0\leq f\leq [n/2]\\ \lam\vdash n-2f}}\Bigl(\dim
D(\lam^t)\Bigr)^{2}=\dim\bb_n(-q^{2m+1},q)=\dim\bb_n(-\zeta^{2m+1},\zeta).
\end{aligned}
$$
Therefore, in order to prove Theorem
\ref{mainthm2} in the case $m\geq n$, it suffices to show that
$\varphi_C$ is injective. \smallskip

Our strategy to prove the injectivity of $\varphi_{C}$ is similar to
that used in \cite[Section 3]{DDH}, but some extra technical
difficulties do arise due to the complexity of the action on
$n$-tensor space in this quantized case. First, we make some
convention on the left and right place permutation actions.
Throughout the rest of this paper, for any $\sigma,\tau\in\BS_n,
a\in\{1,2,\cdots,n\}$, we set
$$ (a)(\sigma\tau)=\bigl((a)\sigma\bigr)\tau,\quad
(\sigma\tau)(a)=\sigma\bigl(\tau(a)\bigr).
$$
In particular, we have $\sigma(a)=(a)\sigma^{-1}$. Therefore, for
any $\bi=(i_1,i_2,\cdots,i_n)\in I(2m,n), w\in\BS_n$,
$$
\bi w=(i_1,i_2,\cdots,i_n)w=(i_{w(1)},i_{w(2)},\cdots,i_{w(n)}), $$
which gives the so-called right place permutation action: $$
v_{\bi}w=(v_{i_1}\otimes\cdots\otimes v_{i_n})w=
v_{i_{w(1)}}\otimes\cdots\otimes v_{i_{w(n)}}=v_{\bi w}.
$$
For each $w\in\BS_n$, the element $T_w$ (resp., $\widehat{T}_w$) is
well defined in the BMW algebra $\bb_n(-\zeta^{2m+1},\zeta)$ (resp.,
in the Hecke algebra $\HH_{K}(\BS_n)$) because of the braid
relations. Precisely,
$$ T_w=T_{j_1}T_{j_2}\cdots T_{j_k}\in \bb_n(-\zeta^{2m+1},\zeta),\quad
\widehat{T}_w=\widehat{T}_{j_1}\widehat{T}_{j_2}\cdots
\widehat{T}_{j_k}\in \HH_{K}(\BS_n),
$$
for any reduced expression $s_{j_1}s_{j_2}\cdots s_{j_k}$ of $w$.

Set $$\begin{aligned} \widehat{\beta} &:=\sum_{1\leq i\leq
2m}\Bigl(qE_{i,i}\otimes E_{i,i}\Bigr)+\sum_{\substack{1\leq
i,j\leq 2m\\ i\neq j}} E_{i,j}\otimes E_{j,i}+\\
&\qquad\qquad\qquad\qquad (q-q^{-1})\sum_{1\leq i<j\leq
2m}\Bigl(E_{i,i}\otimes E_{j,j}\Bigr).
\end{aligned}
$$
For $i=1,2,\cdots,n-1$, we set $$ \widehat{\beta}_i:=\id_{V^{\otimes
i-1}}\otimes\widehat{\beta}\otimes\id_{V^{\otimes n-i-1}}.
$$
By \cite{J2}, the map $\widehat{\varphi}$ which sends each
$\widehat{T}_i$ to $\widehat{\beta}_i$ for $i=1,2,\cdots.n-1$ can be
naturally extended to a representation of $\HH_{\A}(\BS_n)$ on
$V_{\A}^{\otimes n}$.

\begin{lem} \label{obser2} Let $\bi=(i_1,i_2,\cdots,i_n)\in
I(2m,n)$. Suppose that $i_j\neq i'_k$ for any $1\leq j,k\leq n$.
Then for any $w\in\BS_n$, $$ v_{\bi}T_w=v_{\bi}\widehat{T}_w ;
$$
if furthermore $i_1>i_2>\cdots>i_n$, then $ v_{\bi}T_w=v_{\bi w}.$
\end{lem}

\begin{proof} This follows directly from the definition of action (see the formulae
given above (\ref{BG})).
\end{proof}

Let $q$ be an indeterminate over $\mathbb{Z}$. Let
$\widetilde{R}$ be the ring
$$ \widetilde{R}:=\mathbb{Z}[r,r,^{-1},q,q^{-1},x]/\bigl((1-x)(q-q^{-1})+(r-r^{-1})\bigr).
$$
$\widetilde{R}$ naturally becomes an $R$-algebra (with $z$ acting as
$q-q^{-1}$). We regard $\A$ as an $\widetilde{R}$-algebra by sending $r$ to
$-q^{2m+1}$ and $x$ to $1-\sum_{i=-m}^{m}q^{2i}$. The resulting
$\A$-algebra is exactly $\bb_n(-q^{2m+1},q)_{\A}$. We refer the reader to the beginning of Section 3 to understand
how $\A$ is an $R$-algebra. Let
$\bb_{n}(r,q):=\bb_n(r,x,z)\otimes_{R}\widetilde{R}$.
In \cite{E}, a cellular basis
for $\bb_n(r,q)$ indexed by certain bitableaux was constructed by
Enyang. The advantage of that basis is that it is explicitly
described in terms of generators and amenable to computation. In the
remaining part of this section we shall use Enyang's results in
\cite{E}. We first recall some notations and notions.\smallskip

For each natural number $n$ and each integer $f$ with $0\leq f\leq
[n/2]$, we set $\nu=\nu_{f}:=((2^f), (n-2f))$, where
$(2^f):=(\underbrace{2,2,\cdots,2}_{\text{$f$ copies}})$ and
$(n-2f)$ are considered as partitions of $2f$ and $n-2f$
respectively. So $\nu$ is a bipartition of $n$. Let $\ft^{\nu}$ be
the standard $\nu$-bitableau in which the numbers $1,2,\cdots,n$
appear in order along successive rows of the first component
tableau, and then in order along successive rows of the second
component tableau. We define
$$
\mathcal{D}_{\nu}:=\Bigl\{d\in\BS_n\Bigm|\begin{matrix}\text{$\ft^{\nu}d$
is row standard and the first column of $\ft^{(1)}$ is an}\\
\text{increasing sequence when read from top to bottom}\\
\end{matrix}\Bigr\}.
$$
For each partition $\lambda$ of $n-2f$, we denote by $\Std(\lambda)$
the set of all the standard $\lambda$-tableaux with entries in
$\{2f+1,\cdots,n\}$. The initial tableau $\ft^{\lam}$ in this case
has the numbers $2f+1,\cdots,n$ in order along successive rows.

\begin{lem} \label{lm52} Let $d\in\mathcal{D}_{\nu_f}$. Assume that $d=d's_j$ with $\ell(d)=\ell(d')+1$, where $1\leq j\leq n-1$.
Then $d'\in\mathcal{D}_{\nu_f}$.
\end{lem}

\begin{proof} Since $d=d's_j$ and $\ell(d)=\ell(d')+1$, we get
$(j)(d')^{-1}<(j+1)(d')^{-1}$. It follows that $j, j+1$ can not both
sit in the second component of $\ft^{\nu} d'$. If $j, j+1$ sits in
different components of $\ft^{\nu} d'$, then the lemma follows
immediately. So it suffices to consider the case where both $j, j+1$
sits in the first component of $\ft^{\nu}d'$. But
$d\in\mathcal{D}_f$, we deduce that $j$, $j+1$ must be located in
different rows and can not be both located in the first column of
$\ft^{(2^f)}d'$, which implies that $d'\in\mathcal{D}_{\nu_f}$ (as
$\ft^{\nu}d'$ and $\ft^{\nu}d$ differ only in the positions of $j,
j+1$).
\end{proof}

Recall that (cf. \cite{E}) the map $T_i\mapsto T_{i}, E_i\mapsto E_i, \forall\,1\leq
i\leq n-1$ extends naturally to an algebra anti-automorphism of
$\bb_n(-q^{2m+1},q)_{\A}$. We denote this anti-automorphism by
``$\ast$".

\begin{lem} \text{\rm (\cite{E})} \label{Enyangbasis} For each $\lam\vdash n-2f$, $\fs,
\ft\in\Std(\lam)$, let $m_{\fs,\ft}$ denote the canonical image in
$\bb_n(-q^{2m+1},q)_{\A}$ of the corresponding Murphy basis element
(cf. \cite{Mu}) of the Hecke algebra
$\HH_{\A}(\BS_{\{2f+1,\cdots,n\}})$. Then the set
$$ \biggl\{T_{d_1}^{\ast}E_1E_3\cdots
E_{2f-1}m_{\fs\ft}T_{d_2}\biggm|\begin{matrix}
\text{$0\leq f\leq [n/2]$, $\lam\vdash n-2f$, $\fs, \ft\in\Std(\lam)$,}\\
\text{$d_1, d_2\in\mathcal{D}_{\nu}$, where $\nu:=((2^f), (n-2f))$}
\end{matrix}\biggr\}$$
is a cellular basis of the BMW algebra $\bb_n(-q^{2m+1},q)_{\A}$.
\end{lem}

As a consequence, by combining Lemma \ref{Enyangbasis} and
\cite[(3.3)]{E}, we have

\begin{cor} \label{Enyangbasis2} With the above notations, the set
$$ \biggl\{T_{d_1}^{\ast}E_1E_3\cdots E_{2f-1}T_{\sigma}
T_{d_2}\biggm|\begin{matrix}\text{$0\leq f\leq [n/2]$,
$\sigma\in\BS_{\{2f+1,\cdots,n\}}$, $d_1, d_2\in\mathcal{D}_{\nu}$,}\\
\text{where $\nu:=((2^f), (n-2f))$}
\end{matrix}\biggr\}$$
is a basis of the BMW algebra $\bb_n(-q^{2m+1},q)_{\A}$.
\end{cor}

By base change, we can apply the previous results to the specialized
algebra $\bb_n(-\zeta^{2m+1},\zeta)$. The main result in this
section is

\begin{thm} Suppose $m\geq n$. Then the natural homomorphism
$$\varphi_C:\,\,
\bb_n(-\zeta^{2m+1},\zeta)\rightarrow\End_{K}\bigl(V^{\otimes
n}\bigr)$$ is injective. In particular, $\varphi_C$ maps
$\bb_n(-\zeta^{2m+1},\zeta)$ isomorphically onto
$$\End_{\U_{K}(\mathfrak{sp}_{2m})}\bigl(V^{\otimes n}\bigr).$$
\end{thm}

To prove the theorem, it suffices to show the annihilator
$\ann_{\bb_n(-\zeta^{2m+1},\zeta)}(V^{\otimes n})$ is $(0)$. Note
that
$$ \ann_{\bb_n(-\zeta^{2m+1},\zeta)}(V^{\otimes n})=\bigcap_{v\in V^{\otimes
n}}\ann_{\bb_n(-\zeta^{2m+1},\zeta)}(v).
$$
Thus it is enough to calculate
$\ann_{\bb_n(-\zeta^{2m+1},\zeta)}(v)$ for some set of chosen
vectors $v\in V^{\otimes n}$ such that the intersection of
annihilators is $(0)$. We write $$
\ann(v)=\ann_{\bb_n(-\zeta^{2m+1},\zeta)}(v):=\bigl\{x\in
\bb_n(-\zeta^{2m+1},\zeta) \bigm|vx=0\bigr\}.
$$

For each integer $f$ with $0\leq f\leq [n/2]$, we denote by
$B^{(f)}$ the two-sided ideal of $\bb_n(-q^{2m+1},q)_{\A}$ generated
by $E_1E_3\cdots E_{2f-1}$. Note that $B^{(f)}$ is spanned by all
the basis elements whose indexing diagrams contain at least $2f$
horizontal edges ($f$ edges in each of the top and the bottom rows
in the diagrams). We recall a notion introduced in \cite{DDH}. For $\bi\in I(2m,n)$,
an ordered pair $(s,t)$ ($1\leq s<t\leq n$) is called a {\it
symplectic pair} in $\bi$ if $i_s=(i_t)'$. Two ordered pairs $(s,t)$
and $(u,v)$ are called disjoint if
$\bigl\{s,t\bigr\}\cap\bigl\{u,v\bigr\}=\emptyset$. We define the
{\it symplectic length} $\ell_s(v_{\bi})=\ell_s(\bi)$ to be the
maximal number of disjoint symplectic pairs $(s,t)$ in $\bi$. Note
that if $f>\ell_s(v_{\bi})$, then clearly
$B^{(f)}\subseteq\ann(v_{\bi})$.

\begin{lem} \label{step1} $\ann_{\bb_n(-\zeta^{2m+1},\zeta)}\bigl(V^{\otimes
n}\bigr)\subseteq B^{(1)}$.
\end{lem}

\begin{proof} Since $m\geq n$, the tensor $v:=v_n\otimes
v_{n-1}\otimes\cdots\otimes v_1$ is defined. Note that $i\neq j'$
for any $i,j\in\{1,2,\cdots,n\}$. Applying Lemma \ref{obser2}, we
deduce that $vT_w=v\widehat{T}_w$ for any $w\in\BS_n$.  Now
$B^{(1)}$ is contained in the annihilator of $v\widehat{T}_w$, hence
is contained in the intersection of all annihilators of
$v\widehat{T}_w$, as $w$ ranges over $\BS_n$. Hence $B^{(1)}$
annihilates the subspace $S$ spanned by the $vT_w=v\widehat{T}_w$,
where $w$ runs through $\BS_n$.

On the other hand, since $m\geq n$, it is well known (cf.
\cite{DPS}) that the annihilator of $v$ in the Hecke algebra
$\HH_{K}(\BS_n)$ is $\{0\}$. Therefore, we conclude that
$\ann_{\bb_n(-\zeta^{2m+1},\zeta)}\bigl(V^{\otimes n}\bigr)\subseteq
B^{(1)}$.
\end{proof}

Suppose that we have already shown
$\ann_{\bb_n(-\zeta^{2m+1},\zeta)}\bigl(V^{\otimes n}\bigr)\subseteq
B^{(f)}$ for some natural number $f\geq 1$. We want to show
$\ann_{\bb_n(-\zeta^{2m+1},\zeta)}\bigl(V^{\otimes n}\bigr)\subseteq
B^{(f+1)}$. If $f>[n/2]$ then $\ann_{\bb_n(-\zeta^{2m+1},\zeta)}\bigl(V^{\otimes n}\bigr)\subseteq
B^{(f)}=0$ implies that $\ann_{\bb_n(-\zeta^{2m+1},\zeta)}\bigl(V^{\otimes n}\bigr)=0\subseteq B^{(f+1)}$
and we are done. Thus we may assume $f\leq [n/2]$.\smallskip

For $\bi:=(i_1,\cdots,i_n)\in I(2m,n)$, we define
$\bwt(\bi)=\lam=(\lam_1,\cdots,\lam_{2m})$, where $\lam_j$ is the
number of times that $v_j$ occurs as tensor factor in $v_{\bi}$ for
each $1\leq j\leq 2m$. We call $\bwt(\bi)$ the $GL_{2m}$-weight of
$v_{\bi}$. Note that for a given composition $\lam$ of $n$, the
simple tensors of $GL_{2m}$-weight $\lam$ span a
$\HH_{K}(\BS_n)$-submodule $M^{\lam}$ of $V^{\otimes n}$, thus $$
V^{\otimes n}=\bigoplus_{\lam\in\Lambda(2m,n)}M^{\lam}
$$
as $\HH_{K}(\BS_n)$-module, where $\Lambda(2m,n)$ denotes the set of
compositions of $n$ into not more than $(2m)$ parts. It is
well-known that $M^{\lam}$ is isomorphic to the permutation
representation of $\HH_{K}(\BS_n)$ corresponding to $\lam$.

As a consequence, each element $v\in V^{\otimes n}$ can be written
as a sum $$ v=\sum_{\lam\in\Lambda(2m,n)}v_{\lam}
$$
for uniquely determined $v_{\lam}\in M^{\lam}$.\smallskip

Following \cite{DDH}, we consider the subgroup $\Pi$ of
$\BS_{\{1,\cdots,2f\}}\leq\BS_n$ permuting the rows of
$\ft^{\nu^{(1)}}$ but keeping the entries in the rows fixed. The
group $\Pi$ normalizes the stabilizer $\BS_{(2^f)}$ of
$\ft^{\nu^{(1)}}$ in $\BS_{2f}$. We set
$\Psi:=\BS_{(2^f)}\rtimes\Pi$. By \cite[Lemma 3.7]{DDH}, we have
$$\BS_{2f}=\bigsqcup_{d\in\mathcal{D}_f}\Psi d,$$ where ``$\,\sqcup$"
means a disjoint union. We set
$\mathcal{D}_f:=\mathcal{D}_{\nu_f}\bigcap\BS_{2f}$.

\begin{lem} \label{length} Let $d\in\mathcal{D}_f$. Then for any $w\in\Psi$,
$\ell(wd)\geq\ell(d)$.
\end{lem}

\begin{proof} Let $w\in\Psi$. By definition, we can write $w=w''w'$, where
$w''\in\BS_{(2^f)}, w'\in\Pi$. Note that $\BS_{(2^f)}$ is generated
by $s_1,s_3,\cdots,s_{2f-1}$, and $\Pi$ is generated
$\widetilde{s}_1, \widetilde{s}_2,\cdots,$ $\widetilde{s}_{f-1}$, where
$\widetilde{s}_i:=s_{2i}s_{2i-1}s_{2i+1}s_{2i}$ for
$i=1,2,\cdots,f-1$.\smallskip

We claim that $\ell(w)=\ell(w''w'd)\geq\ell(w'd)$. In fact, this
follows easily from the counting of the number of inversions and the
fact that for any $\sigma\in\BS_n$,
$$ \ell(\sigma)=\bigl\{(i,j)\bigm|1\leq i<j\leq n,
(i)\sigma>(j)\sigma\bigr\}.
$$
Therefore, it remains to show that
$\ell(w'd)\geq\ell(d)$.\smallskip

Note that the subgroup generated by $\widetilde{s}_1,
\widetilde{s}_2,\cdots,\widetilde{s}_{f-1}$ is isomorphic to the
symmetric group $\BS_f$. We use $\widetilde{\ell}$ to denote the
length function of $\BS_f$ with respect to the generators
$\widetilde{s}_i, i=1,\cdots,f-1$. We use induction on
$\widetilde{\ell}(w')$. If $\widetilde{\ell}(w')=1$, then
$w'=\widetilde{s}_i=s_{2i}s_{2i-1}s_{2i+1}s_{2i}$. In this case, our
claim $\ell(w'd)\geq\ell(d)$ follows directly from the counting of
the number of inversions. Suppose that for any $w'\in\Pi$ with
$\widetilde{\ell}(w')=k-1$, we have $\ell(w'd)\geq\ell(d)$. Let
$w'\in\Pi$ with $\widetilde{\ell}(w')=k$. We can write
$w'=\widetilde{s}_ju'$, where $1\leq j\leq f-1$, such that
$\widetilde{\ell}(w')=\widetilde{\ell}(u')+1$. Now counting the
number of inversions, it is easy to see that
$\ell(w'd)=\ell(\widetilde{s}_ju'd)\geq \ell(u'd)$. On the other
hand, by induction hypothesis, $\ell(u'd)\geq\ell(d)$. Therefore,
$\ell(w'd)\geq\ell(d)$, as required. This completes the proof of the
lemma.
\end{proof}

Let $\mathcal{P}_f:=\{(i_1,\cdots,i_{2f})|1\leq
i_1<\cdots<i_{2f}\leq n\}$. For each $J\in\mathcal{P}_f$, we use
$d_J$ to denote the unique element in $\mathcal{D}_{\nu_f}$ such
that the first component of $\ft^{\nu}d_J$ is the tableau obtained
by inserting the integers in $J$ in increasing order along
successive rows in $\ft^{\nu^{(1)}}$. Let
$\widetilde{\mathcal{D}}_{(2f,n-2f)}$ be the set of distinguished
right coset representatives of $\BS_{(2f,n-2f)}$ in $\BS_n$. Clearly
$d_{J}\in\widetilde{\mathcal{D}}_{(2f,n-2f)}$, and every element of
$\widetilde{\mathcal{D}}_{(2f,n-2f)}$ is of the form $d_J$ for some
$J\in\mathcal{P}_f$. The following result is well known.

\begin{lem} \label{EXP1} Let $J=(i_1,i_2,\cdots,i_{2f})\in\mathcal{P}_f$. Then $$
(s_{2f}s_{2f+1}\cdots s_{i_{2f}-1})(s_{2f-1}s_{2f}\cdots
s_{i_{2f-1}-1})\cdots (s_2s_3\cdots s_{i_2-1})(s_1s_2\cdots
s_{i_1-1})
$$
is a reduced expression of $d_J$.
\end{lem}

By \cite[Lemma 3.8]{DDH},
$\mathcal{D}_{\nu_f}=\bigsqcup_{J\in\mathcal{P}_f}\mathcal{D}_f
d_{J}$.

\begin{dfn} We define $$\begin{aligned}
\bc&=(c_1,c_2,\cdots,c_{2f})\\
&=\Bigl((m-f+1)',\cdots,(m-1)',m',m,m-1,\cdots,m-f+1\Bigr).\end{aligned}$$
\end{dfn}

Note that $m-f+1<m-f+2<\cdots<m<m'<\cdots<(m-f+2)'<(m-f+1)'$. Let
$d_{0}$ be the unique element in $\BS_{2f}$ such that $$
(a)d_{0}=\begin{cases}(a+1)/2,&\text{if
$a\in\{1,3,\cdots,2f-1\}$,}\\
2f+1-a/2, &\text{if $a\in\{2,4,\cdots,2f\}$.}
\end{cases}
$$
Then $d_{0}\in\mathcal{D}_f$. Counting the number of inversions, we
deduce that $\ell(d_0)=f(f-1)$. On the other hand, by direct
verification, we know that
$$d_0=(s_{2f-2}s_{2f-1})(s_{2f-4}s_{2f-3}s_{2f-2}s_{2f-1})\cdots(s_2s_3\cdots
s_{2f-2}s_{2f-1}).$$ It follows that
\begin{equation}\label{EXP2}
(s_{2f-2}s_{2f-1})(s_{2f-4}s_{2f-3}s_{2f-2}s_{2f-1})\cdots(s_2s_3\cdots
s_{2f-2}s_{2f-1})
\end{equation}
is a reduced expression of $d_0$.

\begin{dfn} We define $$v_{\bc_0}=v_{\bc}d_{0}^{-1}=v_{(m-f+1)'}\otimes v_{m-f+1}\otimes \cdots\otimes
v_{(m-1)'}\otimes v_{m-1}\otimes v_{m'}\otimes v_{m}.$$
\end{dfn}

\begin{lem} \label{maximal} Let $d\in\mathcal{D}_f$, $J_0:=\{n-2f+1,n-2f+2,\cdots,n\}$.
\begin{enumerate}
\item There exists $w\in\BS_{n}$, such that $d_0=dw$ and
$\ell(d_0)=\ell(d)+\ell(w)$;
\item For any $J\in\mathcal{P}_f$, there exists $w'\in\BS_{n}$, such that $d_{J_0}=d_J w'$ with
$\ell(d_{J_0})=\ell(d_J)+\ell(w')$;
\item for any $d\in\mathcal{D}_{\nu_f}$ with $d\neq d_0d_{J_0}$,
there exists integer $1\leq j\leq n-1$ such that
$ds_j\in\mathcal{D}_{\nu_f}$ and $\ell(ds_j)=\ell(d)+1$.
\end{enumerate}
\end{lem}

\begin{proof} The statement (2) is a well-known result, see e.g., \cite{DJ1}. We only give the
proof of the statements (1) and (3).

First, we claim that there exists an element $w_1\in\BS_{2f}$, such
that $dw_1\in\mathcal{D}_f$, $\ell(dw_1)=\ell(d)+\ell(w_1)$ and the
numbers $1,2,\cdots,f$ are located in the first column of
$\ft^{(2^f)}$. In fact, let $1\leq a\leq f$ be the smallest integer
which is not located in the first column of $\ft^{(2^f)}d$, then for
any integer $b$ which is located in the first column of
$\ft^{(2^f)}d$, we must have $b\geq a+1$. Furthermore, any integer
between $a$ and $b-1$ can only be located in the first $a-1$ rows of
the second column of $\ft^{(2^f)}$. Now let
$w_1:=s_{b-1}s_{b-2}\cdots s_a$. It is easy to see that
$dw_1\in\mathcal{D}_f$, $\ell(dw_1)=\ell(d)+\ell(w)$, and
$1,2,\cdots,a$ are located in the first column of $\ft^{(2^f)}$.
Using induction on $a$, we can find an element $w'\in\BS_{2f}$, such
that $dw'\in\mathcal{D}_f$, $\ell(dw')=\ell(d)+\ell(w')$ and the
numbers $1,2,\cdots,f$ are located in the first column of
$\ft^{(2^f)}$. Let $w_{0,f}$ be the unique element in $\BS_{2f}$
such that $$ (a)w_{0,f}=\begin{cases}(a+1)/2,&\text{if
$a\in\{1,3,\cdots,2f-1\}$,}\\
f+a/2, &\text{if $a\in\{2,4,\cdots,2f\}$.}
\end{cases}
$$
Then, $dw'=w_{0,f}w'_1$ for some
$w'_1\in\BS_{\{f+1,f+2,\cdots,2f\}}$.

Let $w'_0\in\BS_{\{f+1,f+2,\cdots,2f\}}$ be defined by $$
(f+1,f+2,\cdots,2f)w'_0=(2f,2f-1,\cdots,f+1).
$$
Then $w'_0$ is the unique longest element in
$\BS_{\{f+1,f+2,\cdots,2f\}}$. It is well known that there exists
$w''\in\BS_{\{f+1,f+2,\cdots,2f\}}$ such that $w'_0=w'_1w''$ and
$\ell(w'_0)=\ell(w'_1)+\ell(w'')$. It is clear that
$$\begin{aligned}\ell(dw'w'')&=\ell(w_{0,f}w'_0)=\ell(w_{0,f})+\ell(w'_0)
=\ell(w_{0,f})+\ell(w'_1)+\ell(w'')\\
&=\ell(w_{0,f}w'_1)+\ell(w'')=\ell(dw')+\ell(w'')=\ell(d)+\ell(w')+\ell(w'').
\end{aligned}$$
Therefore, $$
\ell(d)+\ell(w')+\ell(w'')=\ell(dw'w'')\leq\ell(d)+\ell(w'w'')\leq\ell(d)+\ell(w')+\ell(w''),
$$
which forces $\ell(w'w'')=\ell(w')+\ell(w'')$. Since
$d_0=dw_{0,f}w'_0=dw_{0,f}w'_1w''=d(w'w'')$. The statement (1) is
proved.\smallskip

Let $d\in\mathcal{D}_{\nu_f}$ with $d\neq d_0d_{J_0}$. We can write
$d=d_1d_{J_1}$, where $d_1\in\mathcal{D}_f, J_1\in\mathcal{P}_f$.
For any $d'\in\widetilde{\mathcal{D}}_{(2f,n-2f)}$ and any integer
$1\leq j\leq n-1$ satisfying $d'=d''s_j$ and $\ell(d')=\ell(d'')+1$,
it is well known that $d''\in\widetilde{\mathcal{D}}_{(2f,n-2f)}$.
If $J_1\neq J_0$, then the statement (3) follows directly from this
well-known fact and the statement (2). Now we assume $J_1=J_0$. Then
$d_1\neq d_0$. By statement (1), we can find $s_l\in\BS_{2f}$ such
that $d_1s_l\in\mathcal{D}_f$ and $\ell(d_1s_l)=\ell(d_1)+1$. Then
$d_{J_0}^{-1}s_ld_{J_0}=s_j$ for some $s_{j}\in\BS_{(n-2f,2f)}$.
Note that $$\begin{aligned}
\ell(ds_j)&=\ell(d_1d_{J_0}s_j)=\ell(d_1s_ld_{J_0})=\ell(d_1s_l)+\ell(d_{J_0})
=\ell(d_1)+1+\ell(d_{J_0})\\
&=\ell(d)+1,\end{aligned}$$ as required.
\end{proof}

We define
$$ I_f:=\Bigl\{\bk=(b_{1},\cdots,b_{n-2f})\Bigm|\text{$1\leq
b_{n-2f}<\cdots<b_{2}<b_1\leq m-f$}\Bigr\}. $$ It is clear that
$\ell_s(v_{\bc}\otimes v_{\bk})=f$ for all $\bk\in I_f$.
\smallskip

For an arbitrary element $v\in V^{\otimes n}$, we say the simple
tensor $v_{\bi}=v_{i_1}\otimes\cdots\otimes v_{i_n}$ is involved in
$v$, if $v_{\bi}$ has nonzero coefficient in writing $v$ as linear
combination of the basis $\bigl\{v_{\bj}\bigm|\bj\in I(2m,n)\bigr\}$
of $V^{\otimes n}$. For later use, we note the following very useful
fact: for any $(i_1,i_2), (j_1,j_2)\in I(2m,2)$,
\begin{equation}\label{useful}
\text{$v_{j_1}\otimes v_{j_2}$ is involved in $(v_{i_1}\otimes
v_{i_2})\beta'$ only if $j_1\leq i_2$ and $j_2\geq i_1$.}
\end{equation}

\begin{lem} \label{Tech1} Let $s, i_1,\cdots,i_a$ be integers such that \begin{enumerate}
\item $1\leq s\leq f$;
\item $\ell_s(i_1,\cdots,i_a)=0$;
\item for each integer $1\leq t\leq a$, either $1\leq i_t<m-f+1$ or $m'\leq i_t\leq
(m-f+s+1)'$.
\end{enumerate}
Let $d$ be a distinguished right coset representative of
$$\BS_{(1,2,\cdots,2s)}\times\BS_{(2s+1,\cdots,2s+a)}$$ in
$\BS_{2s+a}$. Let $J:=\{a+1,a+2,\cdots,a+2s\}$. Let
$$\begin{aligned}
\widetilde{v}&=v_{i_1}\otimes\cdots\otimes v_{i_a},\\
\widetilde{w}&=v_{(m-f+1)'}\otimes v_{(m-f+2)'}\otimes\cdots\otimes
v_{(m-f+s)'}\otimes v_{m-f+s}\otimes\cdots\otimes v_{m-f+1}.
\end{aligned}
$$Then
$$\bigl(\widetilde{v}\otimes\widetilde{w}\bigr)T_{d^{-1}}
=\zeta^{z}\delta_{d,d_{J}}\widetilde{w}\otimes\widetilde{v}+\sum_{\bu\in
I(2m,2s+a)}A_{\bu}v_{u_1}\otimes\cdots\otimes v_{u_{2s+a}},
$$
for some $z\in\mathbb{Z}$, and $A_{\bu}\neq 0$ only if
\begin{enumerate}
\item[(4)] $\ell_s(u_1,\cdots,u_{2s})<s$; and
\item[(5)] any integer $x$ with $(m-f+1)'<x\leq 2m$ or $m-f+s+1\leq x\leq m$ does not appear in $(u_1,\cdots,u_{2s})$.
\end{enumerate}
\end{lem}

\begin{proof} We write $$
j_1=(1)d,\,\, j_2=(2)d,\,\,\cdots j_{2s}=(2s)d.
$$
Then $1\leq j_1<j_2<\cdots<j_{2s}\leq 2s+a$, and $d=d_{J}$ if and
only if $j_t=a+t$ for each integer $1\leq t\leq 2s$. Note that
$$
(s_{j_1-1}\cdots s_2s_1)(s_{j_2-1}\cdots s_3s_2)\cdots
(s_{j_{2s}-1}\cdots s_{2s+1}s_{2s})
$$
is a reduced expression of $d^{-1}$.\smallskip

If $\ell(d)=1$, i.e., $d=1$, then there is nothing to prove. In general, let
$$ d'^{-1}=(s_{j_2-1}\cdots s_3s_2)(s_{j_3-1}\cdots s_4s_3)\cdots
(s_{j_{2s-1}-1}\cdots s_{2s}s_{2s-1}).
$$
Then $$ d^{-1}=(s_{j_1-1}\cdots s_2s_1)d'^{-1}(s_{j_{2s}-1}\cdots
s_{2s+1}s_{2s}),\,\,\ell(d)=\ell(d')+j_1+j_{2s}-2s-1,
$$
and $d'$ is a distinguished right coset representative of
$$\BS_{\{2,3,\cdots,2s-1\}}\times\BS_{\{2s,\cdots,2s+a-1\}}$$ in
$\BS_{\{2,3,\cdots,2s+a-1\}}$. Note that since $j_1\leq a+1$, any simple tensor involved
in $$ \bigl(\widetilde{v}\otimes\widetilde{w}\bigr)T_{j_1-1}\cdots
T_2T_1$$ is of the form  $$v_{\widehat{i}_1}\otimes
v_{\widehat{i}_2}\otimes\cdots\otimes v_{\widehat{i}_a}\otimes
v_{\widehat{i}_{a+1}}\otimes \widetilde{w}'\otimes v_{(m-f+s)'},
$$
where $$ \widetilde{w}'=v_{(m-f+1)'}\otimes\cdots\otimes
v_{(m-f+s)'}\otimes v_{m-f+s}\otimes\cdots\otimes v_{m-f+1}.
$$
It suffices to consider the following three cases:
\medskip

\noindent {\it Case 1.} $1\leq j_1\leq a$. Then
$\widehat{i}_{a+1}=m-f+s$. Since $\ell_s(i_1,\cdots,i_a)=0$, by the
definition of $\beta'$, it is easy to see that $$
\bwt(i_1,\cdots,i_a)=\bwt(\widehat{i}_1,\cdots,\widehat{i}_a).
$$
In particular, either $1\leq \widehat{i}_1<m-f+1$ or
$m'<\widehat{i}_1\leq (m-f+s+1)'$. We define $$
\widetilde{v}''=v_{\widehat{i}_2}\otimes
v_{\widehat{i}_3}\otimes\cdots\otimes v_{\widehat{i}_a},\,\,\,
\widetilde{w}''=\widetilde{w}.
$$
Then our conclusion follows easily from induction on $a$.
\smallskip

\noindent {\it Case 2.} $j_1=a+1$ and $\widehat{i}_1\neq (m-f+1)'$.
Then we must have $j_t=a+t$ for each integer $1\leq t\leq 2s$. By
the definition of $\beta'$, it is easy to see that
$\widehat{i}_1\neq (m-f+1)'$ implies that either $1\leq
\widehat{i}_1<m-f+1$ or $m'<\widehat{i}_1\leq (m-f+s+1)'$. We define
$$ \widetilde{v}''=v_{\widehat{i}_2}\otimes
v_{\widehat{i}_3}\otimes\cdots\otimes v_{\widehat{i}_a}\otimes
v_{\widehat{i}_{a+1}},\,\,\, \widetilde{w}''=\widetilde{w}'.
$$
By induction on $\ell(d)$, we deduce that $$
\bigl(\widetilde{v}''\otimes\widetilde{w}''\bigr)T_{(d')^{-1}}
=\zeta^{z}\widetilde{w}''\otimes\widetilde{v}''+\sum_{\bu\in
I(2m,2s+a-2)}A'_{\bu}v_{u_1}\otimes\cdots\otimes v_{u_{2s+a-2}},
$$
for some $z\in\mathbb{Z}$, and $A_{\bu}\neq 0$ only if
\begin{enumerate}
\item[(a1)] $\ell_s(u_1,\cdots,u_{2s-2})<s-1$; and
\item[(a2)] any integer $x$ with $1\leq x<m-f+1$ or $m'\leq x\leq (m-f+s)'$ does not appear in $(u_1,\cdots,u_{2s-2})$.
\end{enumerate}
It remains to consider $$\begin{aligned}
&\Bigl(v_{\widehat{i}_1}\otimes
\widetilde{w}''\otimes\widetilde{v}''\otimes v_{m-f+1}\Bigr)T_{a+2s-1}\cdots T_{2s+1}T_{2s},\\
&\Bigl(v_{\widehat{i}_1}\otimes v_{u_1}\otimes\cdots\otimes
v_{u_{2s+a-2}}\otimes v_{m-f+1} \Bigr)T_{a+2s-1}\cdots
T_{2s+1}T_{2s},\end{aligned}
$$
where $u_1,\cdots,u_{2s-2}$ satisfy the conditions (a1), (a2) above.
Note that under the action of $T_{a+2s-1}\cdots T_{2s+1}T_{2s}$, the
first $(2s-1)$ parts do not change, while by (\ref{useful}) the
$2s$ position will be replaced by a vector of the form $v_{p}$ with
$p\leq m-f+1$. Now using the condition (a2), our conclusion follows
immediately.
\smallskip

\noindent {\it Case 3.} $j_1=a+1$ and $\widehat{i}_1=(m-f+1)'$. Then
we also must have $\widehat{i}_{2s+1}=i_{2s}$, $j_t=a+t$ and
$\widehat{i}_t={i}_{t-1}$ for each integer $2\leq t\leq 2s$. In this
case, our conclusion follows from the same argument used in the
proof of Case 2.
\end{proof}

\begin{lem} \label{J1} Let $\bk\in I_f$, $v=v_{\bk}\otimes v_{\bc}\in V^{\otimes n}$. Let
$w\in\mathcal{D}_{\nu_f}$. If $w\neq d_0d_{J_0}$, then
$$T_{w}^{\ast}E_1E_3\cdots E_{2f-1}\in\ann(v).$$
\end{lem}

\begin{proof} Let $w\in\mathcal{D}_{\nu_f}$. We write $w=d_1d_J$,
where $d_1\in\mathcal{D}_f, J\in\mathcal{P}_f$. Then
$$vT_{w}^{\ast}=(v_{\bk}\otimes v_{\bc})T_{d_{J}^{-1}}T_{d_1^{-1}}.
$$

Let $J=(i_1,i_2,\cdots,i_{2f})$. By Lemma \ref{EXP1}, $$
(s_{i_1-1}\cdots s_2s_1)(s_{i_2-1}\cdots s_3s_2)\cdots
(s_{i_{2f-1}-1}\cdots s_{2f}s_{2f-1})(s_{i_{2f}-1}\cdots
s_{2f+1}s_{2f})
$$
is a reduced expression of $d_{J}^{-1}$. Using Lemma \ref{Tech1}, we
get that

\begin{equation} \label{F1}
(v_{\bk}\otimes
v_{\bc}){T}_{d_{J}^{-1}}=\zeta^z\delta_{J,J_0}(v_{\bc}\otimes
v_{\bk})+\sum_{v'}A_{v'}v',
\end{equation}
for some integer $z$ and some $A_{v'}\in K$, where the subscript
$v'$ runs over all the simple $n$-tensor such that its first
$(2f)$-parts have symplectic length less than $f$. It follows that
if $J\neq J_0$, then we are done. Henceforth, we assume that
$J=J_0$, then $d_1\neq d_0$ and
$$
vT_{w}^{\ast}E_1E_3\cdots E_{2f-1}=\zeta^a(v_{\bc}\otimes
v_{\bk})T_{d_1^{-1}}E_1E_3\cdots E_{2f-1},
$$
for some integer $a$.

By (\ref{EXP2}), $$\begin{aligned} \sigma &=(s_{2f-1}s_{2f-2}\cdots
s_3s_2)(s_{2f-1}s_{2f-2}\cdots s_5s_4)\cdots
\\
&\qquad\qquad\qquad
(s_{2f-1}s_{2f-2}s_{2f-3}s_{2f-4})(s_{2f-1}s_{2f-2})
\end{aligned}$$ is a reduced expressed expression of $d_0^{-1}$.
By Lemma \ref{maximal}, for any $d\in\mathcal{D}_f$, $d$ is less or
equal than $d_0$ in the Bruhat order. It follows that there is a
subexpression of $\sigma$ which is equal to a reduced expression of
$d^{-1}$. Combining this with the definitions of the operator
$\beta'$ and the indices $\bk, \bc$, it is easy to see that
\begin{equation}\label{F2}
(v_{\bc}\otimes v_{\bk}){T}_{d^{-1}}=v_{\bc}{T}_{d^{-1}}\otimes
v_{\bk}=\zeta^z\delta_{d,d_0} v_{\bc}d_{0}^{-1}\otimes
v_{\bk}+\sum_{v'}B_{v'}v',
\end{equation}
for some integer $z\in\mathbb{Z}$ and some $B_{v'}\in K$, where the
subscript $v'$ runs over all the simple $n$-tensor
$v_{j_1}\otimes\cdots\otimes v_{j_{n}}$ such that there exists
$1\leq s\leq f$ satisfying $j_{2s-1}\neq (j_{2s})'$. Now using the
fact that $d_1\neq d_0$, it is easy to see that $(v_{\bc}\otimes
v_{\bk}){T}_{d_{1}^{-1}}E_1E_3\cdots E_{2f-1}=0$. Hence,
$vT_{w}^{\ast}E_1E_3\cdots E_{2f-1}=0$, as required.
\end{proof}

We are now ready to prove the key lemma from which our main result
in this section will follow easily.

\begin{lem} \label{keylem2}  Let $S$ be the subset
$$\biggl\{T_{d_1}^{\ast}E_1E_3\cdots E_{2f-1}T_{\sigma}T_{d_2}\biggm|\begin{matrix}
\text{$d_1, d_2\in\mathcal{D}_{\nu_f}$, $d_1\neq d_0d_{J_0}$,}\\
\text{$\sigma\in\BS_{\{2f+1,\cdots,n\}}$}\\
\end{matrix}\biggr\}$$ of the basis (\ref{Enyangbasis2}) of
$\bb_n(-\zeta^{2m+1},\zeta)$, and let $U$ be the subspace spanned by
$S$. Then
$$ B^{(f)}\bigcap\Bigl(\bigcap_{\bk\in I_{f}}\ann(v_{\bk}\otimes v_{\bc})\Bigr)=B^{(f+1)}\oplus U.
$$
\end{lem}

\begin{proof} By definition of $I_f$, $\ell_s(v_{\bk})=0$. Hence
$\ell_s(v_{\bk}\otimes v_{\bc})=f$. It follows that
$B^{(f+1)}\subseteq\ann(v_{\bk}\otimes v_{\bc})$. This, together
with the Lemma \ref{J1}, shows that the right-hand side is contained
in the left-hand side.\smallskip

Now let $x\in B^{(f)}\bigcap\bigl(\bigcap_{\bk\in
I_{f}}\ann(v_{\bk}\otimes v_{\bc})\bigr)$. We want to show that
$x\in B^{(f+1)}\oplus U$. Using the basis (\ref{Enyangbasis2}) of
the BMW algebra $\bb_n(-\zeta^{2m+1},\zeta)$, we may assume that
$$
x=T_{d_0d_{J_0}}^{\ast}E_1E_3\cdots
E_{2f-1}\biggl(\sum_{d\in\mathcal{D}_{\nu_f}}z_d T_d\biggr),$$ where
$\nu=\nu_f=((2^f),(n-2f))$ and the elements $z_d$,
$d\in\mathcal{D}_{\nu_f}$ are taken from the $K$-linear subspace
spanned by $\bigl\{T_w\bigm|w\in\BS_{\{2f+1,\cdots,n\}}\bigr\}$. We
then have to show $x=0$, or equivalently, to show that $z_{d}=0$ for
each $d\in\mathcal{D}_{\nu_f}$.\smallskip

Let $d\in\mathcal{D}_{\nu_f}$. We write
$z_d=\sum_{\sigma\in\BS_{\{2f+1,\cdots,n\}}}B_{\sigma}T_{\sigma}$,
where $B_{\sigma}\in K$ for each $\sigma$.  Suppose that
$v_{\bk}z_{d}=0$ for any $\bk\in I_{f}$. By the definition of
$I_{f}$, it is easy to see that $$\begin{aligned}
0=v_{\bk}z_{d}&=\sum_{\sigma\in\BS_{\{2f+1,\cdots,n\}}}B_{\sigma}v_{\bk}T_{\sigma}\\
&=\sum_{\sigma\in\BS_{\{2f+1,\cdots,n\}}}B_{\sigma}v_{\bk}\widehat{T}_{\sigma}\\
&=v_{\bk}\sum_{\sigma\in\BS_{\{2f+1,\cdots,n\}}}B_{\sigma}\widehat{T}_{\sigma}.
\end{aligned}$$ However, since $m-2f\geq n-2f$, the Hecke algebra
$\HH_{K}(\BS_{\{2f+1,\cdots,n\}})$ acts faithfully on $v_{\bk}$.
This implies $B_{\sigma}=0$ for each
$\sigma\in\BS_{\{2f+1,\cdots,n\}}$. Thus $z_d=0$, as required.
Therefore, to show that $z_d=0$, it suffices to show that
$v_{\bk}z_{d}=0$ for any $\bk\in I_f$. We divide the proof into two
steps: \medskip

\noindent {\it Step 1.} We first prove that $z_{d_0d_{J_0}}=0$,
equivalently, $v_{\bk}z_{d_0d_{J_0}}=0$ for any $\bk\in I_f$. Let
$\bk\in I_f$. $$\begin{aligned} 0&=(v_{\bk}\otimes v_{\bc})x
=\sum_{d\in\mathcal{D}_{\nu_f}}(v_{\bk}\otimes
v_{\bc})T_{d_{J_0}^{-1}}T_{d_0^{-1}}
E_1E_3\cdots E_{2f-1}z_d T_d\\
&=\zeta^z\sum_{d\in\mathcal{D}_{\nu_f}}\Bigl(v_{\bc_0}E_1E_3\cdots
E_{2f-1}\otimes
v_{\bk}\Bigr)z_d T_d\\
&=\zeta^z\sum_{d\in\mathcal{D}_{\nu_f}}\Bigl(v_{\bc_0}E_1E_3\cdots
E_{2f-1}\otimes v_{\bk}z_d\Bigr)T_d,
\end{aligned}
$$
for some integer $z\in\mathbb{Z}$, where the third equality follows
from (\ref{F1}) and (\ref{F2}).\smallskip

By \cite[Lemma 3.8]{DDH}, for each $d\in\mathcal{D}_{\nu}$, we can
write $d=d_1d_{J}$, where $d_1\in\mathcal{D}_f$,
$J\in\mathcal{P}_f$, and $\ell(d)=\ell(d_1)+\ell(d_J)$. Hence
$T_d=T_{d_1}T_{d_J}$. Therefore
$$\begin{aligned}
0&=\sum_{d\in\mathcal{D}_{\nu}}\Bigl(v_{\bc_0}E_1E_3\cdots
E_{2f-1}\otimes
v_{\bk}z_d\Bigr)T_d\\
&=\sum_{J\in\mathcal{P}_f}\sum_{d_1\in\mathcal{D}_{f}}\Bigl(v_{\bc_0}E_1E_3\cdots
E_{2f-1}\otimes v_{\bk}z_{d_1d_J}\Bigr)T_{d_1}{T}_{d_J}\\
&=\sum_{J\in\mathcal{P}_f}\sum_{d_1\in\mathcal{D}_{f}}\Bigl(v_{\bc_0}E_1E_3\cdots
E_{2f-1}T_{d_1}\otimes v_{\bk}z_{d_1d_J}\Bigr){T}_{d_J}.
\end{aligned}
$$

We want to show that $v_{\bk}z_{d_0d_{J_0}}=0$. Note that $$
v_{\bc_0}E_1E_3\cdots E_{2f-1}=\sum_{1\leq i_1,\cdots,i_f\leq 2m}\pm
\zeta^{a_{\bi}}v_{i_1}\otimes v_{i'_1}\otimes v_{i_2}\otimes
v_{i'_2}\otimes\cdots\otimes v_{i_f}\otimes v_{i'_f},
$$
for some $a_{\bi}\in\Z$. Note also that each simple tensor
$v_{\widehat{\bk}}$ involved $v_{\bk}z_{d_0d_{J_0}}$ has the same
$GL_{2m}$-weight as $v_{\bk}$. Let $(i_1,\cdots,i_f)\in I(2m,f)$. We
claim that
\begin{enumerate} \item [(a)] for any $\widehat{\bk},\widetilde{\bk}\in I(2m,n-2f), \widehat{\bj}\in I(2m,2f)$
with $\bwt(\widehat{\bk})=\bwt(\widetilde{\bk})=\bwt(\bk)$ and
$\ell_s(\bj)=f$, the simple $n$-tensor $(v_{\bc_0d_0}\otimes
v_{\widehat{\bk}})d_{J_0}$ is involved in
$$ (v_{\widehat{j}_1}\otimes v_{\widehat{j}_2}\otimes\cdots\otimes v_{\widehat{j}_{2f}}\otimes
v_{\widetilde{\bk}})T_{d_{J}}
$$
if and only if
$(\widehat{j}_1,\widehat{j}_2,\cdots,\widehat{j}_{2f})=\bc_0d_0$,
$J=J_0$ and $\widehat{\bk}=\widetilde{\bk}$;\smallskip

\item [(b)] for any $d_1\in\BS_{2f}$ with $d_1\leq d_0$ (in the Bruhat order), the simple $(2f)$-tensor
$v_{\bc_0d_0}$ is involved in
$$ (v_{i_1}\otimes v_{i'_1}\otimes v_{i_2}\otimes
v_{i'_2}\otimes\cdots\otimes v_{i_f}\otimes v_{i'_f})T_{d_1}
$$
if and only if $(i_1,i_2,\cdots,i_f)=((m-f+1)',\cdots,(m-1)',m')$ and
$d_1=d_0$.
\end{enumerate}
Once these two claims are proved to be true, it is easy to see
that the identity $v_{\bk}z_{d_0d_{J_0}}=0$ follows at once.
Therefore, it remains to prove the claims (a) and (b).\smallskip

Suppose that $(v_{\bc_0d_0}\otimes v_{\widehat{\bk}})d_{J_0}$ is
involved in
$$ (v_{\widehat{j}_1}\otimes v_{\widehat{j}_2}\otimes\cdots\otimes v_{\widehat{j}_{2f}}\otimes
v_{\widetilde{\bk}})T_{d_{J}}.
$$ By definition, $$ (v_{\bc_0d_0}\otimes
v_{\widehat{\bk}})d_{J_0}=v_{\widehat{\bk}}\otimes v_{\bc_0d_0}.
$$
Let $J=(j_1,j_2,\cdots,j_{2f})$. Then $$ (s_{2f}s_{2f+1}\cdots
s_{j_{2f}-1})(s_{2f-1}s_{2f}\cdots s_{j_{2f-1}-1})\cdots
(s_{1}s_{2}\cdots s_{j_1-1})
$$
is a reduced expression of $d_{J}$. Note that $1\leq
j_1<\cdots<j_{2f}\leq 2f$. If $j_{2f}\neq n$, then the rightmost
vector of any simple tensor involved in $(v_{\widehat{j}_1}\otimes
v_{\widehat{j}_2}\otimes\cdots\otimes v_{\widehat{j}_{2f}}\otimes
v_{\widetilde{\bk}})T_{d_{J}}$ must be $v_{\widetilde{b}_{n-2f}}$,
which is impossible (because $\widetilde{b}_{n-2f}\leq m-f$).
Therefore, we deduce that $j_{2f}=n$. Let $\Sigma$ be the set of all
the simple $n$-tensor $v$ which is involved in $$
(v_{\widehat{j}_1}\otimes v_{\widehat{j}_2}\otimes\cdots\otimes
v_{\widehat{j}_{2f}}\otimes v_{\widetilde{\bk}})T_{2f}T_{2f+1}\cdots
T_{n-1}.
$$
Note that $\widetilde{b}_t\leq m-f$ for each $t$.  We claim that
for each integer $t$ with $1\leq t\leq n-2f$, $\widetilde{b}_t\neq
(\widehat{j}_{2f})'$.

In fact, if $1\leq t\leq n-2f$ is the smallest integer such that
$\widetilde{b}_t=(\widehat{j}_{2f})'$, then the $(2f+t)$th position
of any simple tensor involved in $$ (v_{\widehat{j}_1}\otimes
v_{\widehat{j}_2}\otimes\cdots\otimes v_{\widehat{j}_{2f}}\otimes
v_{\widetilde{\bk}})(T_{2f}T_{2f+1}\cdots T_{2f+t-1})
$$ is a
vector $v_{a}$ with either $a>(m-f+1)'$ or $a<(m-f+1)$. It follows
(from the definition of $\beta'$ and (\ref{useful})) that the $n$th position of any
simple tensor involved in
$$
(v_{\widehat{j}_1}\otimes v_{\widehat{j}_2}\otimes\cdots\otimes
v_{\widehat{j}_{2f}}\otimes
v_{\widetilde{\bk}})(T_{2f}T_{2f+1}\cdots T_{n-1})$$ is a vector
$v_{a}$ with either $a\leq m-f$ or $a\geq (m-f)'$. Since the action of
$(T_{2f-1}T_{2f}\cdots T_{j_{2f-1}-1})\cdots (T_{1}T_{2}\cdots
T_{j_1-1})$ on any simple $n$-tensor does not change its rightmost
vector, we deduce that $v_{\widehat{\bk}}\otimes v_{\bc_0d_0}$ can
not be involved in
$$\begin{aligned}&(v_{\widehat{j}_1}\otimes
v_{\widehat{j}_2}\otimes\cdots\otimes v_{\widehat{j}_{2f}}\otimes
v_{\widetilde{\bk}})(T_{2f}T_{2f+1}\cdots
T_{n-1})(T_{2f-1}T_{2f}\cdots T_{j_{2f-1}-1})\\
&\qquad\qquad\qquad\qquad\qquad\qquad\qquad
\cdots(T_{1}T_{2}\cdots T_{j_1-1}),\end{aligned}$$ a contradiction.

Therefore, $\widetilde{b}_t\neq (\widehat{j}_{2f})'$ for any $1\leq
t\leq n-2f$. It follows that $v=v_{\widehat{j}_1}\otimes
v_{\widehat{j}_2}\otimes\cdots\otimes v_{\widehat{j}_{2f-1}}\otimes
v_{\widetilde{\bk}}\otimes v_{\widehat{j}_{2f}}$ is the unique
simple $n$-tensor in $\Sigma$ such that $v_{\widehat{\bk}}\otimes
v_{\bc_0d_0}$ is involved in
$$v(T_{2f-1}T_{2f}\cdots T_{j_{2f-1}-1})\cdots (T_{1}T_{2}\cdots
T_{j_1-1}).$$ In particular, we deduce that
$\widehat{j}_{2f}=(\bc_0d_0)_{2f}=m-f+1$. Now we are in a position to use
induction on $n$. It follows easily that $$\begin{aligned}
&(j_1,j_2,\cdots,j_{2f-1})=(n-2f+1,n-2f+2,\cdots,n-1),\\
&(\widehat{j}_1,\cdots,\widehat{j}_{2f-1})=((\bc_0d_0)_1,\cdots,(\bc_0d_0)_{2f-1}),\quad
\widetilde{\bk}=\widehat{\bk}. \end{aligned}$$ Conversely, by the
definition of $\beta'$, $(v_{\bc_0d_0}\otimes
v_{\widehat{\bk}})T_{d_{J_0}}=v_{\widehat{\bk}}\otimes v_{\bc_0d_0}$. This proves the claim (a).\smallskip

We now turn to the claim (b). By Lemma \ref{obser2} and direct verification, it is easy to see
that the simple $(2f)$-tensor $v_{\bc_0d_0}$ is involved in
$$ (v_{(m-f+1)'}\otimes v_{m-f+1}\otimes v_{(m-f+2)'}\otimes
v_{m-f+2}\otimes\cdots\otimes v_{m'}\otimes
v_{m})T_{d_0}.
$$
This proves one direction of the claim (b). Now suppose that the
simple $(2f)$-tensor $v_{\bc_0d_0}$ is involved in
$$ (v_{i_1}\otimes v_{i'_1}\otimes v_{i_2}\otimes
v_{i'_2}\otimes\cdots\otimes v_{i_f}\otimes v_{i'_f})T_{d_1},
$$
where $(i_1,\cdots,i_f)\in I(2m,f)$, $d_1\in\BS_{2f}$ with $d_1\leq
d_0$ (in the Bruhat order). Then $d_1$ has a reduced expression
which is a subexpression of (\ref{EXP2}). Hence we can write
$d_1=d'_1d''_1$, where $d'_1$ is a subexpression of $$
(s_{2f-2}s_{2f-1})(s_{2f-4}s_{2f-3}s_{2f-2}s_{2f-1})\cdots
(s_{4}s_{5}\cdots s_{2f-2}s_{2f-1}),
$$
$d''_1$ is a subexpression of $s_{2}s_{3}\cdots s_{2f-2}s_{2f-1}$,
such that $\ell(d_1)=\ell(d'_1)+\ell(d''_1)$. Then
$T_{d_1}=T_{d'_1}T_{d''_1}$.

By definition of $d'_1$, any simple tensor involved in
$$ (v_{i_1}\otimes v_{i'_1}\otimes v_{i_2}\otimes
v_{i'_2}\otimes\cdots\otimes v_{i_f}\otimes v_{i'_f})T_{d'_1}$$ is
of the form $$ v_{i_1}\otimes v_{i'_1}\otimes v_{l_1}\otimes
v_{l_2}\otimes\cdots\otimes v_{l_{2f-2}},
$$
where $\bl=(l_1,l_2,\cdots,l_{2f-2})\in I(2m,2f-2)$ with
$\ell_s(\bl)=f-1$.  By assumption, we can choose one such simple
$n$-tensor, say
$$ v^{[1]}:=v_{i_1}\otimes v_{i'_1}\otimes v_{l_1}\otimes
v_{l_2}\otimes\cdots\otimes v_{l_{2f-2}},
$$ such that $v_{\bc_0d_0}$ is involved in
$vT_{d''_1}$. By definition of $d''_1$, it is easy to see that
$i_1=(m-f+1)'$. We claim that
\begin{enumerate}
\item[(b1)] $(l_1,l_2,\cdots,l_{2f-2})=((m-f+2)',(m-f+3)',\cdots,m',m,\cdots,m-f+3,m-f+2)$;
\item[(b2)] $d''_1=s_{2}s_{3}\cdots s_{2f-2}s_{2f-1}$.
\end{enumerate}
If both (b1) and (b2) are true, then the claim (b) follows easily
from induction on $f$. Therefore, it suffices to prove the two
claims (b1) and (b2).\smallskip

Recall that
$\bc_0d_0=((m-f+1)',\cdots,(m-1)',m',m,m-1,\cdots,m-f+1)$. If
$l_1<(m-f+2)'$, then (by (\ref{useful})) the second position of any
simple $(2f)$-tensor involved in $v^{[1]}T_{d''_1}$ is always
occupied by a vector $v_a$ with $a<(m-f+2)'$, which is impossible
(because $v_{\bc_0d_0}$ is involved in $vT_{d''_1}$). Hence $l_1\geq
(m-f+2)'$. By similar reason, we can deduce that $l_1$ can not be
strictly bigger than $(m-f+1)'$. Therefore,
$l_1\in\{(m-f+1)',(m-f+2)')\}$. Assume that $l_1=(m-f+1)'$. Then by
(\ref{useful}) and the fact that $\ell(\bl)=f-1$, it is easy to see
for any simple $(2f)$-tensor $v_{k_1}\otimes\cdots\otimes
v_{k_{2f}}$ involved in $v^{[1]}T_{d''_1}$, we have $k_t\leq m-f+1$
for some $t\geq 3$, a contradiction. This forces $l_1=(m-f+2)'$.
Repeating the same argument, we deduce that for any integer $1\leq
t\leq f-1$, $l_t=(m-f+t+1)'$. Now since $\ell_s(\bl)=f-1$, it
follows that
$\{l_f,l_{f+1},\cdots,l_{2f-2}\}=\{(m-f+1)',(m-f+2)',\cdots,(m-1)'\}$.
In particular, $l_t\neq m'$ for any $t$. Using the same arguments as
before, we easily deduce that for each integer $f\leq t\leq 2f-2$,
$l_t=m-t+f$. This proves (b1). Now (b2) follows immediately from
(b1) and the fact that $v_{\bc_0d_0}$ is involved in $vT_{d''_1}$.
This proves another direction of the claim (b), hence completes the
proof of the claim (b).
\medskip\smallskip

\noindent {\it Step 2.} Let $S'$ be the subset
$$\biggl\{T_{d_0d_{J_0}}^{\ast}E_1E_3\cdots E_{2f-1}T_{\sigma}T_{d_2}\biggm|\begin{matrix}
\text{$d_2\in\mathcal{D}_{\nu_f}$, $d_2\neq d_0d_{J_0}$,}\\
\text{$\sigma\in\BS_{\{2f+1,\cdots,n\}}$}\\
\end{matrix}\biggr\}$$ of the basis (\ref{Enyangbasis2}) of
$\bb_n(-\zeta^{2m+1},\zeta)$, and let $U'$ be the subspace spanned
by $S'$. By the main results we obtained in Step 1, we know that $$
B^{(f)}\bigcap\Bigl(\bigcap_{\bk\in I_{f}}\ann(v_{\bk}\otimes
v_{\bc})\Bigr)\subseteq B^{(f+1)}\oplus U\oplus U'.
$$
We want to prove that $$ B^{(f)}\bigcap\Bigl(\bigcap_{\bk\in
I_{f}}\ann(v_{\bk}\otimes v_{\bc})\Bigr)\subseteq B^{(f+1)}\oplus U.
$$
If this is not the case, then by Lemma \ref{J1}, $$ U'\bigcap
B^{(f)}\bigcap\Bigl(\bigcap_{\bk\in I_{f}}\ann(v_{\bk}\otimes
v_{\bc})\Bigr)\neq 0.
$$
Let $\widetilde{\Sigma}$ be the set of those
$d_2\in\mathcal{D}_{\nu_f}$, such that there exist some $$x'\in
U'\bigcap
B^{(f)}\bigcap\ann_{\bb_n(-\zeta^{2m+1},\zeta)}\bigl(V^{\otimes
n}\bigr),\,\, d_2\in\mathcal{D}_{\nu},\,\,
\sigma_2\in\BS_{\{2f+1,\cdots,n\}}$$ satisfying
$$T_{d_0d_{J_0}}^{\ast}E_1E_3\cdots E_{2f-1}T_{\sigma_2}T_{d_2}$$ is
involved in $x'$. We choose $d_2\in\widetilde{\Sigma}$ such that
$\ell(d_2)$ is as big as possible.

By the definition of $U'$, $d_2\neq d_0d_{J_0}$. It follows from
Lemma \ref{maximal} that we can find an integer $j$ with $1\leq
j\leq n-1$, such that $d_2s_j\in\mathcal{D}_{\nu_f}$ and
$\ell(d_2s_j)=\ell(d_2)+1$. Let $$x'\in U'\bigcap
B^{(f)}\bigcap\ann_{\bb_n(-\zeta^{2m+1},\zeta)}\bigl(V^{\otimes
n}\bigr), d'_2\in\mathcal{D}_{\nu_f},
\sigma_2\in\BS_{\{2f+1,\cdots,n\}}$$ such that
$T_{d_0d_{J_0}}^{\ast}E_1E_3\cdots E_{2f-1}T_{\sigma_2}T_{d_2}$ is
involved in $x'$. We claim that there exist
$\tau\in\BS_{\{2f+1,\cdots,n\}}, d_3\in\mathcal{D}_{\nu_f}$ with
$\ell(d_3)>\ell(d_2)$, such that
$$T_{d_0d_{J_0}}^{\ast}E_1E_3\cdots E_{2f-1}T_{\tau}T_{d_3}$$ is
involved in $x'T_j$. \smallskip

We write $$\begin{aligned} x'&=A_0T_{d_0d_{J_0}}^{\ast}E_1E_3\cdots
E_{2f-1}z_{d_2}T_{d_2}+\\
&\qquad\qquad\sum_{\substack{\sigma\in\BS_{\{2f+1,\cdots,n\}}\\
d_2\neq d'_2\in\mathcal{D}_{\nu_f}\\
\ell(d'_2)\leq\ell(d_2)}}A_{d'_2,\sigma}T_{d_0d_{J_0}}^{\ast}E_1E_3\cdots
E_{2f-1}T_{\sigma}T_{d'_2},\end{aligned}
$$
where \text{$0\neq z_{d_2}\in K$-Span
$\{T_w|w\in\BS_{\{2f+1,\cdots,n\}}\}$}, $T_{\sigma_2}$ is involved
in $z_{d_2}$, $0\neq A_0\in K$, $A_{d'_2,\sigma}\in K$ for each
$d'_2,\sigma$.

Note that $T_{d_0d_{J_0}}^{\ast}E_1E_3\cdots
E_{2f-1}z_{d_2}T_{d_2}T_j=T_{d_0d_{J_0}}^{\ast}E_1E_3\cdots
E_{2f-1}z_{d_2}T_{d_2s_j}$. It remains to analyze how each
$T_{d_0d_{J_0}}^{\ast}E_1E_3\cdots E_{2f-1}T_{\sigma}T_{d'_2}T_j$ is
expressed as a linear combination of the basis elements given in
Corollary \ref{Enyangbasis2}. Our purpose is to show that
$T_{d_0d_{J_0}}^{\ast}E_1E_3\cdots E_{2f-1}T_{\sigma_2}T_{d_2s_j}$
is not involved in each $$T_{d_0d_{J_0}}^{\ast}E_1E_3\cdots
E_{2f-1}T_{\sigma}T_{d'_2}T_j. $$ We divide the discussion into
cases:
\smallskip

\noindent {\it Case 1.} $\ell(d'_2s_j)=\ell(d'_2)+1$. In this case,
$T_{d'_2}T_j=T_{d'_2s_j}$. We write $d'_2s_j=w_4d_4$, where
$w_4\in\BS_{\nu_f}$, and $\ft^{\nu}d_4$ is row standard. Then
$\ell(w_4d_4)=\ell(w_4)+\ell(d_4)$. Furthermore, we can write
$T_{w_4}=x_4T_{w'_4}$, where $$
x_4\in\langle T_1,T_3,\cdots,T_{2f-1}\rangle,\,\,\, w'_4\in\BS_{\nu_f^{(2)}}.$$
We have
$$\begin{aligned} &T_{d_0d_{J_0}}^{\ast}E_1E_3\cdots
E_{2f-1}T_{\sigma}T_{d'_2}T_j=T_{d_0d_{J_0}}^{\ast}E_1E_3\cdots
E_{2f-1}T_{\sigma}T_{w_4}T_{d_4}\\
&\qquad\qquad\equiv\sum_{\widehat{\sigma}\in\BS_{\{2f+1,\cdots,n\}}}A_{\widehat{\sigma}}T_{d_0d_{J_0}}^{\ast}E_1E_3\cdots
E_{2f-1}T_{\widehat{\sigma}}T_{d_4}\!\!\pmod{B^{(f+1)}},
\end{aligned}
$$
for some $A_{\widehat{\sigma}}\in K$, where the second equality follows from the fact that $E_{2i-1}T_{2i-1}=r^{-1}E_{2i-1}$ for each $i$ and \cite[(3.2)]{E}. Now we use \cite[Proposition 3.7]{E} to express
$T_{d_0d_{J_0}}^{\ast}E_1E_3\cdots
E_{2f-1}T_{\widehat{\sigma}}T_{d_4}$ as a linear combination of the
basis elements given in Corollary \ref{Enyangbasis2}. Note that in
the notation of \cite[Proposition 3.7]{E}, it is easy to check (in all the three cases listed in \cite[Proposition 3.7]{E}) that $u=s_{2j}s_{2j+1}s_{2j-1}s_{2j}w$, $\ell(u)\leq\ell(w), \ell(v)=\ell(w)-1, \ell(v')\leq\ell(w)-1$. It follows
that each $T_{d_0d_{J_0}}^{\ast}E_1E_3\cdots
E_{2f-1}T_{\widehat{\sigma}}T_{d_4}$ can be expressed as a linear
combination of the basis elements of the form $$
T_{d_0d_{J_0}}^{\ast}E_1E_3\cdots E_{2f-1}T_{\sigma''}T_{d''_2},
$$
where $\sigma''\in\BS_{\{2f+1,\cdots,n\}}$,
$d''_2\in\mathcal{D}_{\nu_f}$, with either \begin{enumerate}
\item $\ell(d''_2)<\ell(d_4)$; or
\item $\ell(d''_2)=\ell(d_4)$ and $d_4=zd''_2$ for some $z\in\Pi$.
\end{enumerate}
Note that $\ell(d_4)\leq\ell(d'_2)+1$ with equality holds only if
$d'_2s_j=d_4$. As a result, $\ell(d''_2)\leq\ell(d_2)+1$. If
$\ell(d''_2)<\ell(d_2)+1$, then $d''_2\neq d_2s_j$ and we are done.
If $\ell(d''_2)=\ell(d_2)+1$, then $\ell(d'_2)=\ell(d_2)$,
$\ell(d_4)=\ell(d''_2)=\ell(d'_2)+1$ and $d'_2s_j=d_4=zd''_2$ for
some $z\in\Pi$. In this case we claim that $d''_2\neq d_2s_j$. This
is true because otherwise we would deduce that $d'_2=zd_2$, which is
impossible (since $d'_2, d_2$ are different elements in
$\mathcal{D}_{\nu_f}$). This completes the proof in Case 1.
\smallskip

\noindent {\it Case 2.} $\ell(d'_2s_j)=\ell(d'_2)-1$. Then by Lemma \ref{lm52}, $d'_2s_j\in\mathcal{D}_{\nu_f}$. In this case (note that our $T_j$ is $\zeta^{-1}T_j$ in \cite{E}'s notation), by
\cite[Lemma 2.1]{E},
$$
T_{d'_2}T_j=T_{d'_2s_j}+(\zeta-\zeta^{-1})(T_{d'_2}+\zeta^{-2m-1}T_{d'_2s_j}E_j).
$$
Therefore, $$\begin{aligned} &\quad\,
T_{d_0d_{J_0}}^{\ast}E_1E_3\cdots E_{2f-1}T_{\sigma}T_{d'_2}T_j\\
&=T_{d_0d_{J_0}}^{\ast}E_1E_3\cdots
E_{2f-1}T_{\sigma}T_{d'_2s_j}+(\zeta-\zeta^{-1})T_{d_0d_{J_0}}^{\ast}E_1E_3\cdots
E_{2f-1}T_{\sigma}T_{d'_2}\\
&\qquad\qquad
+(\zeta-\zeta^{-1})\zeta^{-2m-1}T_{d_0d_{J_0}}^{\ast}E_1E_3\cdots
E_{2f-1}T_{\sigma}T_{d'_2s_j}E_j.
\end{aligned}
$$
By comparing their length, we see that $d_2s_j\not\in\{d'_2s_j,d'_2\}$.
Hence $$T_{d_0d_{J_0}}^{\ast}E_1E_3\cdots E_{2f-1}T_{\sigma_2}T_{d_2s_j}$$ is not
involved in $$T_{d_0d_{J_0}}^{\ast}E_1E_3\cdots
E_{2f-1}T_{\sigma}T_{d'_2s_j}+(\zeta-\zeta^{-1})T_{d_0d_{J_0}}^{\ast}E_1E_3\cdots
E_{2f-1}T_{\sigma}T_{d'_2}.$$ Note that $d_2s_j,
d_2\in\mathcal{D}_{\nu_f}$ imply that $j,j+1$ are not in the same
row of $\ft^{\nu}d_2s_j$. Combing this with \cite[Lemma 3.4, 3.5, Proposition 3.3, 3.4]{E}, we deduce that
$T_{d_0d_{J_0}}^{\ast}E_1E_3\cdots E_{2f-1}T_{\sigma_2}T_{d_2s_j}$ is not involved in $$
(\zeta-\zeta^{-1})\zeta^{-2m-1}T_{d_0d_{J_0}}^{\ast}E_1E_3\cdots
E_{2f-1}T_{\sigma}T_{d'_2s_j}E_j, $$ as required. This completes the proof in Case 2.
\smallskip

As a consequence, we can deduce that $T_{d_0d_{J_0}}^{\ast}E_1E_3\cdots
E_{2f-1}T_{\sigma_2}T_{d_2s_j}$ is always involved in $x'T_j$. Note
that $x'T_j\in B^{(f)}\bigcap\Bigl(\bigcap_{\bk\in
I_{f}}\ann(v_{\bk}\otimes v_{\bc})\Bigr)$ and
$\ell(d_2s_j)>\ell(d_2)$. We get a contradiction to our choice of
$d_2$. This completes the proof of the lemma.
\end{proof}

\begin{thm} \label{keythm} With the above notations, we have that $$
B^{(f)}\bigcap\ann_{\bb_n(-\zeta^{2m+1},\zeta)}\bigl(V^{\otimes
n}\bigr)\subseteq B^{(f+1)}.
$$
\end{thm}

\begin{proof} Suppose that $$
B^{(f)}\bigcap\ann_{\bb_n(-\zeta^{2m+1},\zeta)}\bigl(V^{\otimes
n}\bigr)\not\subseteq B^{(f+1)}.
$$
By Lemma \ref{keylem2}, we can find an element $x$ in
$B^{(f)}\bigcap\ann_{\bb_n(-\zeta^{2m+1},\zeta)}\bigl(V^{\otimes
n}\bigr)$ of the following form: $$ x=z+y,
$$
where $z\in B^{(f+1)}, 0\neq y\in U$.

We write $$
y=\sum_{d_1,d_2,\sigma}A_{d_1,d_2,\sigma}T_{d_1}^{\ast}E_1E_3\cdots
E_{2f-1}T_{\sigma}T_{d_2},
$$
where the subscripts run over all elements $d_1,d_2,\sigma$ such
that $d_1,d_2\in\mathcal{D}_{\nu}$ and $d_1\neq d_0d_{J_0}$. Let
$\Sigma$ be the set of those $d_1\in\mathcal{D}_{\nu_f}$, such that
there exist some $x\in
B^{(f)}\bigcap\ann_{\bb_n(-\zeta^{2m+1},\zeta)}\bigl(V^{\otimes
n}\bigr), d_2\in\mathcal{D}_{\nu},
\sigma\in\BS_{\{2f+1,\cdots,n\}}$ satisfying
$$\text{$T_{d_1}^{\ast}E_1E_3\cdots E_{2f-1}T_{\sigma}T_{d_2}$ is involved
in $y$.} $$
We choose $d'_1\in\Sigma$ such that $\ell(d'_1)$ is as big
as possible.

By definition of $U$, $d'_1\neq d_0d_{J_0}$. It follows from Lemma
\ref{maximal} that we can find an integer $j$ with $1\leq j\leq
n-1$, such that $d'_1s_j\in\mathcal{D}_{\nu}$ and
$\ell(d'_1s_j)=\ell(d'_1)+1$. Let $x\in
B^{(f)}\bigcap\ann_{\bb_n(-\zeta^{2m+1},\zeta)}\bigl(V^{\otimes
n}\bigr), d'_2\in\mathcal{D}_{\nu},
\sigma'\in\BS_{\{2f+1,\cdots,n\}}$ be such that
$T_{d'_1}^{\ast}E_1E_3\cdots E_{2f-1}T_{\sigma'}T_{d'_2}$ is
involved in $y$. Note that $$T_j^{\ast}z+T_j^{\ast}y=T_j^{\ast}x\in
B^{(f)}\bigcap\ann_{\bb_n(-\zeta^{2m+1},\zeta)}\bigl(V^{\otimes
n}\bigr),\,\, T_j^{\ast}z\in B^{(f+1)}.$$
Our purpose is to show that there exist some $d_3\in\mathcal{D}_{\nu},
\tau\in\BS_{\{2f+1,\cdots,n\}}$, such that
$T_{d'_1s_j}^{\ast}E_1E_3\cdots E_{2f-1}T_{\tau}T_{d_3}$ is involved
in $T_j^{\ast}y$. If this true, then we get a contradiction to our
choice of $d'_1$, and we are done.\smallskip

We write $$\begin{aligned} y&=A_0T_{d'_1}^{\ast}E_1E_3\cdots
E_{2f-1}T_{\sigma'}T_{d'_2}+\\
&\qquad\sum_{\substack{d_1,d_2\in\mathcal{D}_{\nu}, d_1\neq
d'_1,\\
\sigma\in\BS_{\{2f+1,\cdots,n\}}, \\
\ell(d_1)\leq\ell(d'_1)}}A_{d_1,d_2,\sigma}T_{d_1}^{\ast}E_1E_3\cdots
E_{2f-1}T_{\sigma}T_{d_2}, \end{aligned}$$ where $0\neq A_0\in K$,
$A_{d_1,d_2,\sigma}\in K$ for each
$d_1,d_2\in\mathcal{D}_{\nu_f},\sigma\in\BS_{\{2f+1,\cdots,n\}}$.\smallskip

Note that $T_j^{\ast}T_{d'_1}^{\ast}E_1E_3\cdots
E_{2f-1}T_{\sigma'}T_{d'_2}=T_{d'_1s_j}^{\ast}E_1E_3\cdots
E_{2f-1}T_{\sigma'}T_{d'_2}$. Using the same argument as in the
proof of Step 2 in Lemma \ref{keylem2}, we can show that
$$T_{d'_1s_j}^{\ast}E_1E_3\cdots E_{2f-1}T_{\sigma'}T_{d'_2}$$ is not
involved in $T_j^{\ast}T_{d_1}^{\ast}E_1E_3\cdots E_{2f-1}T_{\sigma}T_{d_2}$,
as required. This completes the proof of the theorem.
\end{proof}
\medskip

\noindent {\bf Proof of Theorem \ref{mainthm2} in the case where
$m\geq n$:} It follows easily from Lemma \ref{step1}, Theorem
\ref{keythm} and induction on $f$ that $$
\ann_{\bb_n(-\zeta^{2m+1},\zeta)}\bigl(V^{\otimes n}\bigr)=0.
$$
This proves the injectivity of $\varphi_C$, and hence $\varphi_C$
must map $\bb_n(-\zeta^{2m+1},\zeta)$ isomorphically onto
$\End_{\U_{K}(\mathfrak{sp}_{2m})}\bigl(V_K^{\otimes n}\bigr)$ in
this case.

\bigskip\bigskip
\bigskip\bigskip
\section{Proof of Theorem \ref{mainthm2} in the case where $m<n$}

The purpose of this section is to give a proof of Theorem
\ref{mainthm2} in the case where $m<n$. Our strategy is similar to
but technically more difficult than \cite[Section 4]{DDH}.\smallskip

Let $m_0$ be a natural number with $m_0\geq n$. Let
$\widetilde{V}_{\A}$ be a free $\A$-module of rank $2m_0$.
Assume that $\widetilde{V}_{\A}$ is equipped with a skew bilinear
form $(,)$ as well as an ordered basis
$\bigl\{{v}_1,{v}_2,\cdots,{v}_{2m_0}\bigr\}$ satisfying
$$ ({v}_i, {v}_{j})=\begin{cases} 1, &\text{if $i+j=2m_0+1$ and
$i<j$;}\\
-1, &\text{if $i+j=2m_0+1$ and
$i>j$;}\\
0 &\text{otherwise.}\\
\end{cases}
$$
For any $\A$-algebra $K$, we set
$\widetilde{V}_{K}:=\widetilde{V}_{\A}\otimes_{\A}K$. Let $\zeta$ be
the image of $q$ in $K$. Let ${\iota}$ be the $K$-linear injection
from $V_K$ into $\widetilde{V}_K$ defined by
$$\sum_{i=1}^{2m}k_iv_i\mapsto\sum_{i=1}^{2m} k_i{v}_{i+m_0-m},\quad\forall\,k_1,\cdots,k_{2m}\in
K.
$$
Let ${\pi}$ be the $K$-linear surjection from $\widetilde{V}_K$ into
$V_K$ defined by
$$\sum_{i=1}^{2m_0}k_i{v}_i\mapsto\sum_{i=1}^{2m}k_{i+m_0-m}{v}_{i}, \quad\forall\,k_1,\cdots,k_{2m_0}\in
K.$$ We set $\widetilde{\iota}:=\zeta^m\iota$,
$\widetilde{\pi}:=\zeta^{m_0}\pi$. We regard $\mathbb{C}$ as an
$\A$-algebra by specializing $q$ to $1$. As before, we identify
$\mathfrak{sp}(V_{\mathbb{C}})$ with
$\mathfrak{sp}_{2m}(\mathbb{C})$ and
$\mathfrak{sp}(\widetilde{V}_{\mathbb{C}})$ with
$\mathfrak{sp}_{2m_0}(\mathbb{C})$. Then, ${\iota}$ induces an
identification of $\mathfrak{sp}_{2m}(\mathbb{C})$ with the Lie
subalgebra of $\mathfrak{sp}_{2m_0}(\mathbb{C})$ consisting of the
following block diagonal matrices:
$$ \Bigl\{\diag(\underbrace{0,\cdots,0,}_{\text{$(m_0-m)$ copies}}A,\underbrace{0,\cdots,0}_{\text{$(m_0-m)$ copies}})
\Bigm|A\in\mathfrak{sp}_{2m}(\mathbb{C})\Bigr\}.
$$
Henceforth let $K$ be a field which is an $\A$-algebra, we set
$$
\widetilde{\mathfrak{g}}:=\mathfrak{sp}_{2m_0}(\mathbb{C}),\,\,
{\mathfrak{g}}:=\mathfrak{sp}_{2m}(\mathbb{C}),\,\,V=V_{K},\,\,\widetilde{V}=\widetilde{V}_{K}.
$$ The inclusion
$\mathfrak{g}\subset\widetilde{\mathfrak{g}}$ naturally induces an
injection
$$\begin{aligned}
&\U_{\Q(q)}(\mathfrak{g})\hookrightarrow\U_{\Q(q)}(\widetilde{\mathfrak{g}})\\
&e_i\mapsto e_{i+m_0-m},\,\,e_i\mapsto e_{i+m_0-m},\,\,k_i\mapsto
k_{i+m_0-m},\,\,i=1,2,\cdots,m.
\end{aligned}
$$
By restriction, we get an injection $
\U_{\A}(\mathfrak{g})\hookrightarrow\U_{\A}(\widetilde{\mathfrak{g}})$.
By base change, we get a natural map
$\U_{K}(\mathfrak{g})\rightarrow\U_{K}(\widetilde{\mathfrak{g}})$.
It is easy to see that $$ \widetilde{\pi}^{\otimes
{2n}}\Bigl(\bigl(\widetilde{V}^{\otimes
2n}\bigr)^{\U_{K}(\widetilde{\mathfrak{g}})}\Bigr)\subseteq
\bigl(V^{\otimes 2n}\bigr)^{\U_{K}(\mathfrak{g})}.
$$

For each integer $i$ with $1\leq i\leq 2m$, we define $$
w_{i}=\begin{cases} v_{i}, &\text{if $1\leq i\leq m$;}\\
(-1)^{i-m-1}v_{i}, &\text{if $m+1\leq i\leq 2m$.}
\end{cases}
$$
Then the natural representation of $\U_{\Q(q)}(\mathfrak{g})$ on
$V_{\Q(q)}$ is given by $$\begin{aligned}
&e_{i}{w}_{j}=\begin{cases} {w}_i, &\text{if $j=i+1$,}\\
{w}_{2m+1-(i+1)}, &\text{if $j=2m+1-i$,}\\
0, &\text{otherwise;}\end{cases}\,
e_{m}{w}_{j}=\begin{cases} {w}_{m}, &\text{if $j=m+1$,}\\
0, &\text{otherwise,}\end{cases} \\
&f_{i}{w}_{j}=\begin{cases} {w}_{i+1}, &\text{if $j=i$,}\\
{w}_{2m+1-i}, &\text{if $j=2m+1-(i+1)$,}\\
0, &\text{otherwise;}\end{cases}\,
f_{m}{w}_{j}=\begin{cases} {w}_{m+1}, &\text{if $j=m$,}\\
0, &\text{otherwise,}\end{cases}\\
&k_{i}{w}_{j}=\begin{cases} q {w}_j, &\text{if $j=i$ or $j=2m+1-(i+1)$,}\\
q^{-1}{w}_{j}, &\text{if $j=i+1$ or $j=2m+1-i$,}\\
{w}_{j}, &\text{otherwise,}\end{cases}\\
&k_{m}{w}_{j}=\begin{cases} q {w}_j, &\text{if $j=m$,}\\
 q^{-1}{w}_{j}, &\text{if $j=m+1$,}\\
{w}_{j}, &\text{otherwise,}\end{cases}\end{aligned}
$$
where $1\leq i\leq m-1$, $1\leq j\leq 2m$. By definition (cf. \cite[(8.18)]{HK}),
$\bigl\{{w}_{i}\bigr\}_{1\leq i\leq 2m}$ is a canonical basis of the
$\U_{\Q(q)}(\mathfrak{g})$-module $V_{\Q(q)}$ in the sense of
\cite{Lu3}. Similarly, the natural
$\U_{\Q(q)}(\widetilde{\mathfrak{g}})$-module
$\widetilde{V}_{\Q(q)}$ has a canonical basis
$\bigl\{{\widetilde{w}}_{i}\bigr\}_{1\leq i\leq 2m_0}$ such that
$$\begin{aligned}
e_{i}{\widetilde{w}}_{j}&=\begin{cases} {\widetilde{w}}_i, &\text{if $j=i+1$,}\\
{\widetilde{w}}_{2m_0+1-(i+1)}, &\text{if $j=2m_0+1-i$,}\\
0, &\text{otherwise;}\end{cases}\\
f_{i}{\widetilde{w}}_{j}&=\begin{cases} {\widetilde{w}}_{i+1}, &\text{if $j=i$,}\\
{\widetilde{w}}_{2m_0+1-i}, &\text{if $j=2m_0+1-(i+1)$,}\\
0, &\text{otherwise;}\end{cases}\\
e_{m_0}{\widetilde{w}}_{j}&=\begin{cases} {\widetilde{w}}_{m_0}, &\text{if $j=m_0+1$,}\\
0, &\text{otherwise,}\end{cases}\\
f_{m_0}{\widetilde{w}}_{j}&=\begin{cases} {\widetilde{w}}_{m_0+1},
&\text{if $j=m_0$,}\\
0, &\text{otherwise,}\end{cases}\\
k_{i}{\widetilde{w}}_{j}&=\begin{cases} q {\widetilde{w}}_j, &\text{if $j=i$ or $j=2m_0+1-(i+1)$,}\\
q^{-1}{\widetilde{w}}_{j}, &\text{if $j=i+1$ or $j=2m_0+1-i$,}\\
{\widetilde{w}}_{j}, &\text{otherwise,}\end{cases}\\
k_{m_0}{\widetilde{w}}_{j}&=\begin{cases} q {\widetilde{w}}_j, &\text{if $j=m_0$,}\\
 q^{-1}{\widetilde{w}}_{j}, &\text{if $j=m_0+1$,}\\
{\widetilde{w}}_{j}, &\text{otherwise,}\end{cases}\end{aligned}
$$
where $1\leq i\leq m_0-1$, $1\leq j\leq 2m_0$. Note that the
subspace $\widehat{V}_{\Q(q)}$ spanned by
$\bigl\{\widetilde{w}_{i}\bigr\}_{m_0-m+1\leq i\leq m_0+m}$ is
stable under the action of the subalgebra
$\U_{\Q(q)}(\mathfrak{g})$, and it is canonically isomorphic to
$\U_{\Q(q)}(\mathfrak{g})$-module $V_{\Q(q)}$.

For each integer $i$ with $1\leq i\leq 2m$, we define
$w_{i}^{\ast}\in V_{\A}^{\ast}:=\Hom_{\A}(V_{\A},\A)$ by
$w_i^{\ast}(v)=(w_i,v)$, $\forall\,v\in V_{\A}$. Then
$\{w_i^{\ast}\}_{i=1}^{2m}$ is an $\A$-basis of $V_{\A}^{\ast}$, and
$w_1^{\ast}$ is a highest weight vector of weight $\varepsilon_1$.
Furthermore, the map $w_1^{\ast}\mapsto w_1$ can be naturally
extended to an $\U_{\A}(\mathfrak{g})$-module isomorphism $\tau:
V_{\A}^{\ast}\cong V_{\A}$ such that $$
\tau(w_i^{\ast})=\begin{cases} q^{1-i}w_i, &\text{if $1\leq
i\leq m$;}\\
q^{-i}w_i, &\text{if $m+1\leq i\leq 2m$,}
\end{cases}
$$
where the $\U_{\A}(\mathfrak{g})$-structure on $V_{\A}^{\ast}$ is
defined via the antipode $S$. In a similar way, we can define an
$\A$-basis $\{\widetilde{w}_i^{\ast}\}_{i=1}^{2m_0}$ of
$\widetilde{V}_{\A}^{\ast}$, and an
$\U_{\A}(\widetilde{\mathfrak{g}})$-module isomorphism
$\widetilde{\tau}: \widetilde{V}_{\A}^{\ast}\cong
\widetilde{V}_{\A}$. Let $\widehat{V}_{\A}$ be the free
$\A$-submodule generated by
$$\bigl\{\widetilde{w}_{m_0-m+1},\widetilde{w}_{m_0-m+2},\cdots,\widetilde{w}_{m_0+m}\bigr\}.$$
Set $\widehat{V}:=\widehat{V}_{\A}\otimes_{\A} K$. Note that the
algebra $\U_{K}(\mathfrak{g})$ acts on $\widehat{V}$ via the natural
map
$\U_{K}(\mathfrak{g})\rightarrow\U_{K}(\widetilde{\mathfrak{g}})$.
The resulting $\U_{K}(\mathfrak{g})$-module $\widehat{V}$ is
naturally isomorphic to the natural $\U_{K}(\mathfrak{g})$-module
$V$ via the correspondence $$ \widetilde{w}_{i}\mapsto w_{i-m_0+m},
\quad \text{for $i=m_0-m+1,m_0-m+2,\cdots,m_0+m$.}
$$

\smallskip

Recall our definitions of $\widetilde{\iota}, \widetilde{\pi}$ at
the beginning of this section. We define a linear map $\Theta_{0}$
as follows: $$\begin{aligned} \Theta_0:&\,\,
\End\bigl(\widetilde{V}^{\otimes
n}\bigr)\rightarrow\End\bigl({V}^{\otimes n}\bigr),\\
&f\mapsto \bigl(\widetilde{\pi}^{\otimes n}\bigr)\circ
f\circ\bigl(\widetilde{\iota}^{\otimes n}\bigr).
\end{aligned}
$$

One can verify directly that $$
\Theta_0\Bigl(\End_{\U_{K}(\widetilde{\mathfrak{g}})}\bigl(\widetilde{V}^{\otimes
n}\bigr)\Bigr)\subseteq\End_{\U_{K}(\mathfrak{g})}\bigl(\widehat{V}^{\otimes
n}\bigr)\cong\End_{\U_{K}(\mathfrak{g})}\bigl(V^{\otimes n}\bigr),
$$
where the last isomorphism comes from the natural
$\U_{K}(\mathfrak{g})$-module isomorphism $\widehat{V}\cong
V$.\smallskip

By Corollary \ref{Enyangbasis2}, the BMW algebra
$\bb_{n}(-q^{2m+1},q)$ has a basis
$$ \biggl\{T_{d_1}^{\ast}E_1E_3\cdots
E_{2f-1}T_{\sigma}T_{d_2}\biggm|\begin{matrix}
\text{$0\leq f\leq [n/2]$, $\lam\vdash n-2f$, $\sigma\in\BS_{\{2f+1,\cdots,n\}}$,}\\
\text{$d_1, d_2\in\mathcal{D}_{\nu}$, where $\nu:=((2^f), (n-2f))$}
\end{matrix}\biggr\}.$$
The same is true for the BMW algebra $\bb_n(-q^{2m_0+1},q)$. To
distinguish its basis elements with those of $\bb_{n}(-q^{2m+1},q)$,
we denote them by
$$\widetilde{T}_{d_1}^{\ast}\widetilde{E}_1\widetilde{E}_3\cdots
\widetilde{E}_{2f-1}\widetilde{T}_{\sigma}\widetilde{T}_{d_2},$$
where $\widetilde{T}_i, \widetilde{E}_i$ are standard generators of
$\bb_n(-q^{2m_0+1},q)$. We define an $\A$-linear isomorphism
$\Theta_1$ from the BMW algebra $\bb_n(-q^{2m_0+1},q)_{\A}$ to the
BMW algebra $\bb_n(-q^{2m+1},q)_{\A}$ as follows: $$
\Theta_1\Bigl(\widetilde{T}_{d_1}^{\ast}\widetilde{E}_1\widetilde{E}_3\cdots
\widetilde{E}_{2f-1}\widetilde{T}_{\sigma}\widetilde{T}_{d_2}\Bigr)=q^{(m_0+m)n}T_{d_1}^{\ast}E_1E_3\cdots
E_{2f-1}T_{\sigma}T_{d_2},
$$
for each $0\leq f\leq [n/2]$, $\lam\vdash n-2f$, $\fs,
\ft\in\Std(\lam)$ and $d_1, d_2\in\mathcal{D}_{\nu_f}$. By base
change, we get a $K$-linear isomorphism
$\bb_n(-\zeta^{2m_0+1},\zeta)\cong\bb_n(-\zeta^{2m+1},\zeta)$, which
will be still denoted by $\Theta_1$.
\smallskip

By the main result in last section, we know that the natural
homomorphism $\varphi_C$ from $\bb_n(-\zeta^{2m_0+1},\zeta)$ to
$\End_{\U_{K}(\mathfrak{sp}_{2m_0})}\bigl((\widetilde{V})^{\otimes
n}\bigr)$ is always an isomorphism. Therefore, in order to prove
Theorem \ref{mainthm2} (in the case where $m<n$), it suffices to prove the following lemma.

\begin{lem} \label{lm61} With the notations as above, \begin{enumerate}
\item the following diagram of maps $$\begin{CD}
\bb_n(-\zeta^{2m_0+1},\zeta)@>{{\varphi}_C}>>\End_{\U_{K}(\widetilde{\mathfrak{g}})}\bigl(\widetilde{V}^{\otimes
n}\bigr)\\
@V{\Theta_1}VV @V{\Theta_0}VV  \\
\bb_n(-\zeta^{2m+1},\zeta)@>{\varphi_C}>>\End_{\U_{K}(\mathfrak{g})}\bigl(V^{\otimes
n}\bigr)\\
\end{CD}
$$
is commutative;
\item the map $$
\Theta_0:
\End_{\U_{K}(\widetilde{\mathfrak{g}})}\bigl(\widetilde{V}^{\otimes
n}\bigr)\rightarrow\End_{\U_{K}(\mathfrak{g})}\bigl(V^{\otimes
n}\bigr)
$$
is surjective.
\end{enumerate}
\end{lem}

The remaining part of this section is devoted to the proof of Lemma
\ref{lm61}. The proof of Lemma \ref{lm61} (2) is almost the same as
\cite[Section 4]{DDH}, which we just sketch here. First, we note
that the following diagram of maps $$\begin{CD}
\End_{\U_{K}(\widetilde{\mathfrak{g}})}\bigl(\widetilde{V}^{\otimes
n}\bigr)@>{\sim}>>\bigl(\widetilde{V}^{\otimes n}\otimes
\bigl(\widetilde{V}^{\ast}\bigr)^{\otimes
n}\bigr)^{\U_{K}(\widetilde{\mathfrak{g}})}@>{\id^{\otimes
n}\otimes\widetilde{\tau}^{\otimes n}}
>>\Bigl(\widetilde{V}^{\otimes 2n}\Bigr)^{\U_{K}(\widetilde{\mathfrak{g}})}\\
@V{\Theta_0}VV @. @V{\widetilde{\pi}^{\otimes {2n}}}VV  \\
\End_{\U_{K}(\mathfrak{g})}\bigl(V^{\otimes
n}\bigr)@>{\sim}>>\bigl({V}^{\otimes n}\otimes
\bigl({V}^{\ast}\bigr)^{\otimes
n}\bigr)^{\U_{K}({\mathfrak{g}})}@>{\id^{\otimes
n}\otimes{\tau}^{\otimes n}}>>
\Bigl({V}^{\otimes 2n}\Bigr)^{\U_{K}({\mathfrak{g}})}\\
\end{CD}
$$
is commutative\footnote{This is the point where we have to use the
isomorphisms $\widetilde{\pi}, \widetilde{\iota}$ instead of $\pi,
\iota$.}. Note that (by the theory of tilting modules)
$$\begin{aligned} \Bigl({V}^{\otimes
2n}\Bigr)^{\U_{K}({\mathfrak{g}})}&\cong\bigl({V}^{\otimes n}\otimes
\bigl({V}^{\ast}\bigr)^{\otimes
n}\bigr)^{\U_{K}({\mathfrak{g}})}\cong\End_{\U_{K}({\mathfrak{g}})}\Bigl({V}^{\otimes
n}\Bigr)\\
&\cong\End_{\U_{\A}({\mathfrak{g}})}\Bigl({V}_{\A}^{\otimes
n}\Bigr)\otimes_{\A} K\cong \Bigl({V}_{\A}^{\otimes
2n}\Bigr)^{\U_{\A}({\mathfrak{g}})}\otimes_{\A}K.\end{aligned}
$$
Therefore, to prove Lemma \ref{lm61}(2), it suffices to show that $$
\widetilde{\pi}^{\otimes {2n}}\Bigl(\bigl(\widetilde{V}^{\otimes
2n}\bigr)^{\U_{K}(\widetilde{\mathfrak{g}})}\Bigr)=\bigl(V^{\otimes
2n}\bigr)^{\U_{K}(\mathfrak{g})},
$$
equivalently, to show that

\begin{equation}
\label{lm611}\widetilde{\pi}^{\otimes
{2n}}\Bigl(\bigl(\widetilde{V}_{\A}^{\otimes
2n}\bigr)^{\U_{\A}(\widetilde{\mathfrak{g}})}\Bigr)=
\bigl(V_{\A}^{\otimes 2n}\bigr)^{\U_{\A}(\mathfrak{g})}.
\end{equation}

Let ${M}:=({V}_{\Q(q)})^{\otimes {2n}}.$ By \cite[(27.3)]{Lu3}, the
$\U_{\Q(q)}({\mathfrak{g}})$-module ${M}$ is a based module. There
is a canonical basis ${B}$ of ${M}$, in Lusztig's notation (\cite[
(27.3.2)]{Lu3}), where each element in ${B}$ is of the form
${w}_{i_1}{\diamond}{w}_{i_2}{\diamond}\cdots{\diamond}
{w}_{i_{2n}}$, and ${w}_{i_1}{\diamond}\cdots{\diamond}{w}_{i_{2n}}$
is equal to ${w}_{i_1}\otimes\cdots\otimes {w}_{i_{2n}}$ plus a
linear combination of elements ${w}_{j_1}\otimes\cdots\otimes
{w}_{j_{2n}}$ with
$({w}_{j_1},\cdots,{w}_{j_{2n}})<({w}_{i_1},\cdots,{w}_{i_{2n}})$
and with coefficients in $v^{-1}\Z[v^{-1}]$, where $"<"$ is a
partial order defined in \cite[(27.3.1)]{Lu3}. In particular, $B$ is
an $\A$-basis of $V_{\A}^{\otimes n}$. By \cite[(27.2.1)]{Lu3},
there is a partition
$$ B=\bigsqcup_{\lambda\in X^{+}} B[\lam]. $$
Let $$ B[\neq 0]:=\bigsqcup_{0\neq\lambda\in X^{+}}
B[\lam],\,\,\,\,\, M[\neq 0]_{\A}:=\sum_{b\in B[\neq 0]}\A b.
$$
By \cite[(27.1),(27.2)]{Lu3} and the discussion in \cite[Section
4]{DDH}, we know that the isomorphism
$\bigl(\tau^{-1}\bigr)^{\otimes 2n}: V_{\A}^{\otimes 2n}\rightarrow
\bigl(V_{\A}^{\ast}\bigr)^{\otimes 2n}\cong \bigl(V_{\A}^{\otimes
2n}\bigr)^{\ast}$ induces an isomorphism
$$ \bigl(V_{\A}^{\otimes 2n}\bigr)^{\U_{\A}(\mathfrak{g})}\cong
\biggl(V_{\A}^{\otimes 2n}/M[\neq 0]_{\A}\biggr)^{\ast}.
$$
All the above have a counterpart with respect to $\widetilde{V}$,
which we will just put the symbol ``$\sim$". Therefore, we have the
notations $\widetilde{M}:=(\widetilde{V}_{\Q(q)})^{\otimes {2n}}$,
$\widetilde{B}$,
$\tilde{w}_{i_1}{\tilde\diamond}\tilde{w}_{i_2}{\tilde\diamond}\cdots{\tilde\diamond}\tilde{w}_{i_{2n}}$,
$\widetilde{M}[\neq 0]_{\A}$, and we also have an isomorphism
$$ \bigl(\widetilde{V}_{\A}^{\otimes 2n}\bigr)^{\U_{\A}(\widetilde{\mathfrak{g}})}\cong
\biggl(\widetilde{V}_{\A}^{\otimes 2n}/\widetilde{M}[\neq
0]_{\A}\biggr)^{\ast}.
$$

\begin{lem} With the notations as above, the following diagram of maps $$\begin{CD}
\bigl(\widetilde{V}_{\A}^{\otimes
2n}\bigr)^{\U_{\A}(\widetilde{\mathfrak{g}})}@>{\sim}>>
\biggl(\widetilde{V}_{\A}^{\otimes 2n}/\widetilde{M}[\neq
0]_{\A}\biggr)^{\ast}@>{}>> \biggl(\widetilde{V}_{\A}^{\otimes
2n}\biggr)^{\ast}\\
@V{\widetilde{\pi}^{\otimes 2n}}VV @. @V{\bigl(\widetilde{\iota}^{\otimes 2n}\bigr)^{\ast}}VV  \\
\bigl(V_{\A}^{\otimes
2n}\bigr)^{\U_{\A}(\mathfrak{g})}@>{\sim}>>\biggl(V_{\A}^{\otimes 2n}/M[\neq 0]_{\A}\biggr)^{\ast}
@>{}>>\biggl(V_{\A}^{\otimes 2n}\biggr)^{\ast}\\
\end{CD}
$$
is commutative. In particular, we have $$ \widetilde{\iota}^{\otimes
2n}\bigl(M[\neq 0]_{\A}\bigr)\subseteq \widetilde{M}[\neq 0]_{\A}.
$$
\end{lem}

\begin{proof} This follows from direct verification.
\end{proof}

Therefore, to prove (\ref{lm611}), it suffices to show that

\begin{equation} \label{lm6111}
\bigl(\widetilde{\iota}^{\otimes
2n}\bigr)^{\ast}\biggl(\bigl(\widetilde{V}_{\A}^{\otimes
2n}/\widetilde{M}[\neq
0]_{\A}\bigr)^{\ast}\biggr)=\bigl(V_{\A}^{\otimes 2n}/M[\neq
0]_{\A}\bigr)^{\ast}.
\end{equation}

Let $$\begin{aligned} J_0&:=\Bigl\{(i_1,\cdots,i_{2n})\in
I(2m,2n)\Bigm|w_{i_1}\diamond\cdots\diamond w_{i_{2n}}\in
B[0]\Bigr\},\\
\widetilde{J}_0&:=\Bigl\{(i_1,\cdots,i_{2n})\in
I(2m_0,2n)\Bigm|\tilde{w}_{i_1}\tilde{\diamond}
\cdots\tilde{\diamond}
\tilde{w}_{i_{2n}}\in\widetilde{B}[0]\Bigr\}.\end{aligned}
$$

\begin{lem} \label{lm65} {\rm (\cite[Corollary 4.5]{DDH})} With the above notation, the set $$
\Bigl\{w_{i_1}\otimes\cdots\otimes w_{i_{2n}}+M[{\neq
0}]_{\A}\Bigm|(i_1,\cdots,i_{2n})\in J_0\Bigr\}
$$
forms an $\A$-basis of $V_{\A}^{\otimes {2n}}/M[{\neq 0}]_{\A}$, and
the set $$ \Bigl\{\tilde{w}_{i_1}\otimes\cdots\otimes
\tilde{w}_{i_{2n}}+\widetilde{M}[{\neq
0}]_{\A}\Bigm|(i_1,\cdots,i_{2n})\in \widetilde{J}_0\Bigr\}
$$
forms an $\A$-basis of $\widetilde{V}_{\A}^{\otimes
{2n}}/\widetilde{M}[{\neq 0}]_{\A}$.
\end{lem}
We set $$
J_0[m_0-m]:=\Bigl\{(m_0-m+i_1,\cdots,m_0-m+i_{2n})\Bigm|(i_1,\cdots,i_{2n})\in
J_0\Bigr\}.
$$

\begin{thm} \label{thm66} With the above notation,
$J_0[m_0-m]\subseteq\widetilde{J}_0$.
\end{thm}

\begin{proof} This is proved by using the same argument as in the
proof of \cite[Theorem 4.7]{DDH}).
\end{proof}

Now Lemma \ref{lm65} and Theorem \ref{thm66} imply that
$\widetilde{\iota}^{\otimes 2n}$ maps $V_{\A}^{\otimes 2n}/M[\neq
0]_{\A}$ onto an $\A$-direct summand of $\widetilde{V}_{\A}^{\otimes
2n}/\widetilde{M}[\neq 0]_{\A}$. It follows that $$
\bigl(\widetilde{\iota}^{\otimes
2n}\bigr)^{\ast}\biggl(\bigl(\widetilde{V}_{\A}^{\otimes
2n}/\widetilde{M}[\neq
0]_{\A}\bigr)^{\ast}\biggr)=\bigl(V_{\A}^{\otimes 2n}/M[\neq
0]_{\A}\bigr)^{\ast}, $$ which proves (\ref{lm6111}). This completes
the proof of Lemma \ref{lm61} (2).\medskip

It remains to prove Lemma \ref{lm61} (1). {\it From now on and until
the end of this paper we shall set, for any integer $i$ with $1\leq
i\leq 2m_0$}, $$i':=2m_0+1-i. $$

Note that both $\Theta_0$ and $\Theta_1$ are in general not algebra
maps. Let $f$ be an integer with $0\leq f\leq [n/2]$. By definition,
$$ \Theta_1\Bigl(\widetilde{E}_1\widetilde{E}_3\cdots
\widetilde{E}_{2f-1}\Bigr)=\zeta^{(m_0+m)n}E_1E_3\cdots E_{2f-1}.
$$
Therefore,
$$\begin{aligned} \varphi_{C}\Theta_1\Bigl(\widetilde{E}_1\widetilde{E}_3\cdots
\widetilde{E}_{2f-1}\Bigr)
&=\zeta^{(m_0+m)n}\varphi_{C}\Bigl({E}_1{E}_3\cdots {E}_{2f-1}\Bigr)\\
&=\zeta^{(m_0+m)n}\varphi_{C}(E_1)\varphi_{C}(E_3)\cdots\varphi_{C}(E_{2f-1})\\
&=\Theta_0\Bigl(\varphi_{C}(\widetilde{E}_1)
\varphi_{C}(\widetilde{E}_3)\cdots
\varphi_{C}(\widetilde{E}_{2f-1})\Bigr)\\
&=\Theta_0\varphi_{C}\Bigl(\widetilde{E}_1\widetilde{E}_3\cdots
\widetilde{E}_{2f-1}\Bigr),
\end{aligned}
$$
where the third equality follows from the fact that different
$\varphi_{C}(\widetilde{E}_{2i-1})$ acts on pairwise non-intersected
positions.

Let $\sigma\in\BS_{\{2f+1,\cdots,n\}}$, $\nu:=((2^f), (n-2f))$,
$d_1,d_2\in\mathcal{D}_{\nu}$. Our purpose is to show that $$
\varphi_{C}\Theta_1\Bigl(\widetilde{T}_{d_1}^{\ast}\widetilde{E}_1\widetilde{E}_3\cdots
\widetilde{E}_{2f-1}\widetilde{T}_{\sigma}\widetilde{T}_{d_2}\Bigr)=\Theta_0\varphi_{C}\Bigl(
\widetilde{T}_{d_1}^{\ast}\widetilde{E}_1\widetilde{E}_3\cdots
\widetilde{E}_{2f-1}\widetilde{T}_{\sigma}\widetilde{T}_{d_2}\Bigr).
$$
Equivalently,
\begin{equation}\label{last}
\varphi_{C}\bigl({T}_{d_1}^{\ast}{E}_1{E}_3\cdots
{E}_{2f-1}{T}_{\sigma}{T}_{d_2}\bigr)=
\zeta^{-(m_0+m)n}\Theta_0\varphi_{C}\bigl(
\widetilde{T}_{d_1}^{\ast}\widetilde{E}_1\widetilde{E}_3\cdots
\widetilde{E}_{2f-1}\widetilde{T}_{\sigma}\widetilde{T}_{d_2}\bigr).
\end{equation}
We divide the proof into four steps:\medskip

\noindent {\it Step 1.} We want to prove that
$\varphi_{C}\Theta_1\Bigl(\widetilde{T}_w\Bigr)=\Theta_0\varphi_{C}\Bigl(\widetilde{T}_w\Bigr)$
for any $w\in\BS_{n}$.

We set $$ \widehat{I}(2m,n)=\bigl\{\bi\in I(2m,n)\bigm|\text{$m_0-m+1\leq
i_t\leq m_0+m$ for each $t$}\bigr\}.$$ Note that, by definition,
$\Theta_1\Bigl(\widetilde{T}_w\Bigr)=\zeta^{(m_0+m)n}T_{w}$ for each
$w\in\BS_n$. Hence it suffices to show that for any $w\in\BS_{n}$,
\begin{equation}\label{equa68} \varphi_{C}\Bigl({T}_w\Bigr)=\zeta^{-(m_0+m)n}
\Theta_0\varphi_{C}\Bigl(\widetilde{T}_w\Bigr).\end{equation}

Let $w\in\BS_n$, $1\leq k\leq n-1$, such that
$\ell(ws_k)=\ell(w)+1$. Let $\bi\in\widehat{I}(2m,n)$, $\bj\in
I(2m_0,n)$, such that $v_{\bj}$ is involved in
$v_{\bi}\widetilde{T}_{ws_k}$. We claim that \begin{enumerate}
\item[(A)] for any $\bl\in I(2m_0,n)$ such that $v_{\bl}$ is involved in
$v_{\bi}\widetilde{T}_{w}$ and $v_{\bj}$ is involved in
$v_{\bl}\widetilde{T}_{k}$, if $l_k=(l_{k+1})'$, and $m_0-m+1\leq
l_b\leq m_0+m$ whenever $b\neq k,k+1$, then $l_k\geq m_0-m+1$.
\end{enumerate}

If $\ell(w)=0$, there is nothing to prove. Assume $\ell(w)=a\geq 1$.
We fix a reduced expression of $w$ as follows: $$
w=s_{q_1}s_{q_2}\cdots s_{q_a}.
$$
We set $q_{a+1}=k$. By assumption, $v_{\bl}$ is involved in
$v_{\bi}\widetilde{T}_{q_1}\cdots\widetilde{T}_{q_{a}}$. It follows
that there exist
$$\bj^{[0]},\bj^{[1]},\cdots,\bj^{[a+1]}\in I(2m_0,n)$$ such that
\begin{enumerate}
\item $\bj^{[0]}=\bi$, $\bj^{[a]}=\bl$, $\bj^{[a+1]}=\bj$;
\item for each integer $1\leq t\leq a+1$,
$v_{\bj^{[t]}}$ is involved in
$v_{\bj^{[t-1]}}\widetilde{T}_{q_{t}}$.
\end{enumerate}
For each integer $2\leq t\leq a+1$, we set
$$ w_t:=s_{q_1}s_{q_2}\cdots s_{q_{t-1}}.
$$

For any integer $t$ with $2\leq t\leq a+1$, we define
$$\begin{aligned}
&A_t:=\sum_{\substack{1\leq b\leq n\\ 1\leq j^{[t-1]}_b<m_0-m+1
}}(b)w_t^{-1},\,\,\,
B_t:=\sum_{\substack{1\leq b\leq n\\ m_0+m<j^{[t-1]}_b\leq 2m_0}}(b)w_t^{-1},\\
&C_t:=A_t-B_t.
\end{aligned}$$ We claim that \begin{equation}\label{KeyStep} 0\leq C_2\leq
C_3\leq\cdots\leq C_{a+1}.
\end{equation}

Recall by definition that $\widetilde{T}_{s_{q_t}}$ acts only on the
$(q_t,q_t+1)$ positions of the simple tensor $v_{\bj^{[t-1]}}$ in
the same way as the operator $\beta'$ acts on
$v_{j^{[t-1]}_{q_t}}\otimes v_{j^{[t-1]}_{q_t+1}}$. To compare $C_t$
with $C_{t+1}$, it suffices to check the simple tensors involved in
$$
\Bigl(v_{j^{[t-1]}_{q_t}}\otimes v_{j^{[t-1]}_{q_t+1}}\Bigr)\beta'.
$$
By the explicit definition of the operator $\beta'$, we need only
consider the following six possibilities:
\medskip

\noindent {\it Case 1.}
$\bigl\{j^{[t-1]}_{q_t},j^{[t-1]}_{q_t+1},j^{[t]}_{q_t},j^{[t]}_{q_t+1}\bigr\}\subseteq\bigl\{m_0-m+1,
m_0-m+2,\cdots,m_0+m\bigr\}$. In this case, it is clear that
$C_t=C_{t+1}$.\medskip

\noindent {\it Case 2.} $j^{[t-1]}_{q_t}\neq
\bigl(j^{[t-1]}_{q_t+1}\bigr)'$. Then we have either
$\bigl(j^{[t-1]}_{q_t},j^{[t-1]}_{q_t+1}\bigr)=\bigl(j^{[t]}_{q_t+1},
j^{[t]}_{q_t}\bigr)$ or
$\bigl(j^{[t-1]}_{q_t},j^{[t-1]}_{q_t+1}\bigr)=\bigl(j^{[t]}_{q_t},
j^{[t]}_{q_t+1}\bigr)$. In both cases, we still get that
$C_t=C_{t+1}$.\medskip

\noindent {\it Case 3.}
$j^{[t-1]}_{q_t}=\bigl(j^{[t-1]}_{q_t+1}\bigr)'>m_0+m$. Then we must
have
$$j^{[t]}_{q_t}=\bigl(j^{[t]}_{q_t+1}\bigr)'\leq
j^{[t-1]}_{q_t+1}<m_0-m+1.
$$
In this case, since $$\begin{aligned}
&\bigl({q_t}\bigr)s_{q_t}^{-1}={q_t+1},\,\,\,
\bigl({q_t+1}\bigr)s_{q_t}^{-1}={q_t},\\
&j^{[t-1]}_{q_t+1}, \, j^{[t]}_{q_t}<m_0-m+1,\,\,\, j^{[t-1]}_{q_t},
j^{[t]}_{q_t+1}>m_0+m.
\end{aligned}
$$ and
$\bigl({b}\bigr)s_{q_t}^{-1}={b}$ for any $b\notin\{q_t,q_t+1\}$, it
follows easily that $C_t=C_{t+1}$.
\medskip

\noindent {\it Case 4.} $m_0-m+1\leq
j^{[t-1]}_{q_t}=\bigl(j^{[t-1]}_{q_t+1}\bigr)'\leq m_0+m$, and
$$j^{[t]}_{q_t}=\bigl(j^{[t]}_{q_t+1}\bigr)'<m_0-m+1.
$$
In this case, since $\ell(w_ts_{q_t})=\ell(w_t)+1$, it follows that
$$
\bigl({q_t}\bigr)w_{t+1}^{-1}=\bigl({q_t+1}\bigr)w_{t}^{-1}>\bigl({q_t}\bigr)w_{t}^{-1}=
\bigl({q_t+1}\bigr)w_{t+1}^{-1}.
$$
and $\bigl({b}\bigr)s_{q_t}^{-1}={b}$ for any
$b\notin\{q_t,q_t+1\}$, it follows that $$
C_{t+1}=C_t+\Bigl(\bigl({q_t}\bigr)w_{t+1}^{-1}-\bigl({q_t+1}\bigr)w_{t+1}^{-1}\Bigr)>C_t.
$$

\noindent {\it Case 5.}
$j^{[t-1]}_{q_t}=\bigl(j^{[t-1]}_{q_t+1}\bigr)'<m_0-m+1$, and
$\bigl(j^{[t-1]}_{q_t},
j^{[t-1]}_{q_t+1}\bigr)\neq\bigl(j^{[t]}_{q_t},
j^{[t]}_{q_t+1}\bigr).$ In this case, we must have
$j^{[t]}_{q_t}=\bigl(j^{[t]}_{q_t+1}\bigr)'$. Since
$\ell(w_ts_{q_t})=\ell(w_t)+1$, it follows that
$$
\bigl({q_t}\bigr)w_{t+1}^{-1}=\bigl({q_t+1}\bigr)w_{t}^{-1}>\bigl({q_t}\bigr)w_{t}^{-1}=
\bigl({q_t+1}\bigr)w_{t+1}^{-1}.
$$
and $\bigl({b}\bigr)s_{q_t}^{-1}={b}$ for any
$b\notin\{q_t,q_t+1\}$. If $j^{[t]}_{q_t}>m_0+m$, then it is clear
that $C_t=C_{t+1}$; if $j^{[t]}_{q_t}<m_0-m+1$, then $$
C_{t+1}=C_t-\Bigl(\bigl({q_t}\bigr)w_{t}^{-1}-\bigl({q_t+1}\bigr)w_{t}^{-1}\Bigr)+
\Bigl(\bigl({q_t}\bigr)w_{t+1}^{-1}-\bigl({q_t+1}\bigr)w_{t+1}^{-1}\Bigr)>C_t;
$$
if $m_0-m+1\leq j^{[t]}_{q_t}\leq m_0+m$, then $$
C_{t+1}=C_t-\Bigl(\bigl({q_t}\bigr)w_{t}^{-1}-\bigl({q_t+1}\bigr)w_{t}^{-1}\Bigr)>C_t.
$$

\noindent {\it Case 6.}
$j^{[t-1]}_{q_t}=\bigl(j^{[t-1]}_{q_t+1}\bigr)'<m_0-m+1$, and
$\bigl(j^{[t-1]}_{q_t}, j^{[t-1]}_{q_t+1}\bigr)=\bigl(j^{[t]}_{q_t},
j^{[t]}_{q_t+1}\bigr).$ Since $\ell(w_ts_{q_t})=\ell(w_t)+1$, it
follows that
$$
\bigl({q_t}\bigr)w_{t+1}^{-1}=\bigl({q_t+1}\bigr)w_{t}^{-1}>\bigl({q_t}\bigr)w_{t}^{-1}=
\bigl({q_t+1}\bigr)w_{t+1}^{-1}.
$$
and $\bigl({b}\bigr)s_{q_t}^{-1}={b}$ for any
$b\notin\{q_t,q_t+1\}$. Therefore $$
C_{t+1}=C_t-\Bigl(\bigl({q_t}\bigr)w_{t}^{-1}-\bigl({q_t+1}\bigr)w_{t}^{-1}\Bigr)+
\Bigl(\bigl({q_t}\bigr)w_{t+1}^{-1}-\bigl({q_t+1}\bigr)w_{t+1}^{-1}\Bigr)>C_t,
$$
as required. This proves our claim (\ref{KeyStep}).\smallskip

Since $\ell(w_{a+1}s_{k})=\ell(w_{a+1})+1$, it follows that
$\bigl({k}\bigr)w_{a+1}^{-1}<\bigl({k+1}\bigr)w_{a+1}^{-1}$. Now
suppose that $l_k<m_0-m+1$. Then by our assumption on $\bl$, it is
easy to see that
$$ C_{a+1}=(k)w_{a+1}^{-1}-(k+1)w_{a+1}^{-1}<0,
$$
which contradicts to (\ref{KeyStep}). It follows that $l_k\geq
m_0-m+1$. This completes the proof of our claim (A).

Now we use induction on $\ell(w)$ and the results (A) to prove our
claim (\ref{equa68}). If $\ell(w)=0$, there is nothing to prove. Let
$w=us_k$ with $\ell(w)=\ell(u)+1$. Suppose
$\varphi_{C}\Bigl({T}_u\Bigr)=\zeta^{-(m_0+m)n}\Theta_0\varphi_{C}\Bigl(\widetilde{T}_u\Bigr)$.
Then for each $\bi\in\widehat{I}(2m,n)$, we can write
$$
v_{\bi}\widetilde{T}_u={\iota}^{\otimes
n}\bigl(v_{\bi-m_0+m}T_{u})+\sum_{\bl\in
I(2m_0,n)\setminus\widehat{I}(2m,n)}A_{\bi,\bl}v_{\bl},
$$
where $A_{\bi,\bl}\in K$ for each $\bl$, and
$$\bi-m_0+m:=(i_1-m_0+m,\cdots,i_n-m_0+m).$$ Therefore,
$$ v_{\bi}\widetilde{T}_w=\biggl({\iota}^{\otimes
n}\bigl(v_{\bi-m_0+m}T_{u})\biggr)\widetilde{T}_k+\sum_{\bl\in
I(2m_0,n)\setminus\widehat{I}(2m,n)}A_{\bi,\bl}v_{\bl}\widetilde{T}_k.
$$
Note that $A_{\bi,\bl}\neq 0$ implies that $v_{\bl}$ is involved in
$v_{\bi}\widetilde{T}_u$.

We claim that $\pi^{\otimes n}\bigl(v_{\bl}\widetilde{T}_k\bigr)=0$
whenever $A_{\bi,\bl}\neq 0$. In fact, by the definition of $\beta'$
and the fact that $\bl\in I(2m_0,n)\setminus\widehat{I}(2m,n)$, it
is easy to see that $\pi^{\otimes
n}\bigl(v_{\bl}\widetilde{T}_k\bigr)\neq 0$ only if
$l_k=(l_{k+1})'<m_0-m+1$ and $m_0-m+1\leq l_b\leq m_0+m$ whenever
$b\neq k,k+1$. Applying our result (A), we know that this is
impossible. This proves our claim.

Note also that
$$ \biggl({\iota}^{\otimes
n}\bigl(v_{\bi-m_0+m}T_{u})\biggr)\widetilde{T}_k={\iota}^{\otimes
n}\bigl(v_{\bi-m_0+m}T_{u}T_k)+\sum_{\bj\in
I(2m_0,n)\setminus\widehat{I}(2m,n)}A'_{\bi,\bj}v_{\bj},
$$
where $A'_{\bi,\bj}\in K$ for each $\bj$. As a consequence, we get
that
$$ {\pi}^{\otimes
n}\Bigl(v_{\bi}\widetilde{T}_w\Bigr)=v_{\bi-m_0+m}T_{w}.
$$
Equivalently,
$\varphi_{C}\Bigl({T}_w\Bigr)=\zeta^{-(m_0+m)n}\Theta_0\varphi_{C}\Bigl(\widetilde{T}_w\Bigr)$,
as required.  This completes the proof of our claim (\ref{equa68}).
As a direct consequence, it is easy to see that for any
$\sigma\in\BS_{\{2f+1,\cdots,n\}}$, \begin{equation}\label{KeyStep2}
\varphi_{C}\Theta_1\Bigl(\widetilde{E}_1\widetilde{E}_3\cdots
\widetilde{E}_{2f-1}\widetilde{T}_{\sigma}\Bigr)=\Theta_0\varphi_{C}\Bigl(
{E}_1{E}_3\cdots {E}_{2f-1}T_{\sigma}\Bigr).
\end{equation}
\smallskip

\noindent {\it Step 2.} We claim that for any
$d_1\in\mathcal{D}_{\nu_f}$,
$$
\varphi_{C}\Theta_1\Bigl(\widetilde{T}_{d_1}^{\ast}\widetilde{E}_1\widetilde{E}_3\cdots
\widetilde{E}_{2f-1}\widetilde{T}_{\sigma}\Bigr)=\Theta_0\varphi_{C}\Bigl(
\widetilde{T}_{d_1}^{\ast}\widetilde{E}_1\widetilde{E}_3\cdots
\widetilde{E}_{2f-1}\widetilde{T}_{\sigma}\Bigr).
$$
By definition,
$\Theta_1\Bigl(\widetilde{T}_{d_1}^{\ast}\widetilde{E}_1\widetilde{E}_3\cdots
\widetilde{E}_{2f-1}\widetilde{T}_{\sigma}\Bigr)=\zeta^{(m_0+m)n}{T}_{d_1}^{\ast}{E}_1{E}_3\cdots
{E}_{2f-1}{T}_{\sigma}$. Therefore, our claim is equivalent to
\begin{equation}\label{KeyStep3}
\varphi_{C}\Bigl({T}_{d_1}^{\ast}{E}_1{E}_3\cdots
{E}_{2f-1}{T}_{\sigma}\Bigr)=\zeta^{-(m_0+m)n}\Theta_0\varphi_{C}\Bigl(
\widetilde{T}_{d_1}^{\ast}\widetilde{E}_1\widetilde{E}_3\cdots
\widetilde{E}_{2f-1}\widetilde{T}_{\sigma}\Bigr).
\end{equation}
By (\ref{KeyStep2}), for any $\bi\in\widehat{I}(2m,n)$,
$$
v_{\bi}\widetilde{T}_{d_1}^{\ast}={\iota}^{\otimes
n}\bigl(v_{\bi-m_0+m}T_{d_1}^{\ast})+\sum_{\bl\in
I(2m_0,n)\setminus\widehat{I}(2m,n)}\widehat{A}_{\bi,\bl}v_{\bl},
$$
where $\widehat{A}_{\bi,\bl}\in K$ for each $\bl$.

To prove (\ref{KeyStep3}), it suffices to show that
\begin{equation}\label{KeyStep33}
\text{$\ell_s(l_1,\cdots,l_{2f})<f$ whenever
$\widehat{A}_{\bi,\bl}\neq 0$.}
\end{equation}

It remains to prove (\ref{KeyStep33}). By \cite[Lemma 3.8]{DDH}, we
can write $d_1=d_{11}d_{J}$, where $d_{11}\in\mathcal{D}_f,
J\in\mathcal{P}_f$. Then, $$
\widetilde{T}_{d_1}^{\ast}=\widetilde{T}_{d_{J}}^{\ast}\widetilde{T}_{d_{11}}^{\ast}=
\widetilde{T}_{d_{J}^{-1}}\widetilde{T}_{d_{11}^{-1}}.
$$
Since $d_{11}^{-1}\in\BS_{2f}$, the action of
$\widetilde{T}_{d_{11}^{-1}}$ does not change the symplectic length
of the first $(2f)$ parts of any simple $n$-tensor. Therefore, using
(\ref{equa68}), we can assume without loss of generality that
$d_{11}=1$, and hence
$d_1=d_{J}\in\widetilde{\mathcal{D}}_{(2f,n-2f)}$. With this
assumption, we claim that
\begin{equation}\label{KeyStep333}
v_{\bi}\widetilde{T}_{d_J}^{\ast}={\iota}^{\otimes
n}\bigl(v_{\bi-m_0+m}T_{d_J}^{\ast})+\sum_{\bl\in
I(2m_0,n)\setminus\widehat{I}(2m,n)}\widehat{A}_{\bi,\bl}v_{\bl},
\end{equation}
where $\widehat{A}_{\bi,\bl}\in K$ for each $\bl$, and
$\widehat{A}_{\bi,\bl}\neq 0$ only if $$
\bigl\{l_1,\cdots,l_{2f}\bigr\}\bigcap\bigl\{m_0+m+1,m_0+m+2,\cdots,2m_0\bigr\}=\emptyset,$$
and either $l_1$ or $l_2$ belongs to
$\bigl\{1,2,\cdots,m_0-m\bigr\}.$ If this is true, then it is clear
that $\widehat{A}_{\bi,\bl}\neq 0$ only if $\ell_s(\bl)<f$ and hence
(\ref{KeyStep33}) follows.

Let $J=(j_1,j_2,\cdots,j_{2f})$. Then $1\leq
j_1<j_2<\cdots<j_{2f}\leq n$. By Lemma \ref{EXP1},
$$ (s_{j_1-1}\cdots s_2s_1)(s_{j_2-1}\cdots s_3s_2)\cdots
(s_{j_{2f-1}-1}\cdots s_{2f}s_{2f-1})(s_{j_{2f}-1}\cdots
s_{2f+1}s_{2f})
$$
is a reduced expression of $d_{J}^{-1}$. If $f=1$, then $$
\widetilde{T}_{d_{J}}^{\ast}=(\widetilde{T}_{j_1-1}\cdots
\widetilde{T}_2\widetilde{T}_1)(\widetilde{T}_{j_2-1}\cdots
\widetilde{T}_3\widetilde{T}_2).
$$
In this case, suppose that $v_{\bl}$ is involved in
$v_{\bi}\widetilde{T}_{d_{J}}^{\ast}$, where
$\bi\in\widehat{I}(2m,n)$. Then, there must exist
$$\bl^{[0]},\bl^{[1]},\cdots,\bl^{[j_1+j_2-3]}\in I(2m_0,n)$$ such that
\begin{enumerate}
\item $\bl^{[0]}=\bi$, $\bl^{[j_1+j_2-3]}=\bl$;
\item for each $1\leq t\leq j_1-1$,
$v_{\bl^{[t]}}$ is involved in $v_{\bi}(\widetilde{T}_{j_1-1}\cdots
\widetilde{T}_{j_1-t+1}\widetilde{T}_{j_1-t})$;
\item for each $j_1\leq t\leq j_1+j_2-3$,
$v_{\bl^{[t]}}$ is involved in $$v_{\bi}(\widetilde{T}_{j_1-1}\cdots
\widetilde{T}_2\widetilde{T}_1)(\widetilde{T}_{j_2-1}\cdots
\widetilde{T}_{j_2-t+j_1}\widetilde{T}_{j_2-t+j_1-1});$$
\item for each $1\leq t\leq j_1-1$, $v_{\bl^{[t]}}$ is involved in $v_{\bl^{[t-1]}}
\widetilde{T}_{j_1-t}$;
\item for each $j_1\leq t\leq j_1+j_2-3$, $v_{\bl^{[t]}}$ is involved in $v_{\bl^{[t-1]}}
\widetilde{T}_{j_1+j_2-t-1}$.
\end{enumerate}
Now suppose $\bl\in I(2m_0,n)\setminus\widehat{I}(2m,n)$. If there
exists an integer $1\leq b\leq j_1$ such that $$ m_0-m+1\leq
l_{j_1-b}^{[b-1]}=\bigl(l_{j_1-b+1}^{[b-1]}\bigr)'\leq
m_0+m,\,\,\,l_{j_1-b}^{[b]}=\bigl(l_{j_1-b+1}^{[b]}\bigr)'<m_0-m+1,
$$
then we choose such a $b$ which is maximal. By (\ref{useful}), we
have $l_1=l_1^{[j_1-1]}<m_0-m+1$ and $l_2\leq l_{j_2}^{[j_1-1]}\leq
m_0+m$. If there does not exist such a $b$, then there must exist an
integer $j_1\leq c\leq j_1+j_2-3$ such that
$$\begin{aligned} m_0-m+1\leq
l_{j_1+j_2-c-1}^{[c-1]}&=\bigl(l_{j_1+j_2-c}^{[c-1]}\bigr)'\leq
m_0+m,\\
l_{j_1+j_2-c-1}^{[c]}&=\bigl(l_{j_1+j_2-c}^{[c]}\bigr)'<m_0-m+1.
\end{aligned}$$ We choose such a $c$ which is
maximal. By (\ref{useful}), $l_2<m_0-m+1$. The non-existence of $b$
implies that $m_0-m+1\leq l_1\leq m_0+m$. This proves our claim in
the case $f=1$.\smallskip

Now we assume that $f\geq 2$. We use induction on $n-2f$. If
$n-2f=0$, then $d_J=1$, there is nothing to prove. We set $$
\widehat{d}_J=(s_{2f-1}s_{2f}\cdots s_{j_{2f-1}-1})
(s_{2f-2}s_{2f-1}\cdots s_{j_{2f-2}-1})\cdots (s_{1}s_{2}\cdots
s_{j_{1}-1}).
$$
Then $d_J^{-1}=\widehat{d}_J^{-1}(s_{j_{2f}-1}\cdots
s_{2f+1}s_{2f})$ and
$$\ell(d_J^{-1})=\ell(\widehat{d}_{J}^{-1})+j_{2f}-2f.$$ If $j_{2f}\leq n-1$, then
$d_J\in\mathcal{D}_{(2f,n-1-2f)}$, and we are done by induction
hypothesis. If $j_{2f}=n$, then by induction hypothesis, we have $$
v_{\bi}\widetilde{T}_{\widehat{d}_J}^{\ast}={\iota}^{\otimes
n}\bigl(v_{\bi-m_0+m}T_{\widehat{d}_J}^{\ast})+\sum_{\bl\in
I(2m_0,n)\setminus\widehat{I}(2m,n)}\widehat{B}_{\bi,\bl}v_{\bl},
$$
where $\widehat{B}_{\bi,\bl}\in K$ for each $\bl$, and
$\widehat{B}_{\bi,\bl}\neq 0$ only if \begin{enumerate}
\item $m_0-m+1\leq l_n\leq m_0+m$; and
\item
$\bigl\{l_1,\cdots,l_{2f}\bigr\}\bigcap\bigl\{m_0+m+1,m_0+m+2,\cdots,2m_0\bigr\}=\emptyset$;
and
\item either $l_1$ or $l_2$ belongs to
$\bigl\{1,2,\cdots,m_0-m\bigr\}$.\end{enumerate} It remains to check
the simple tensors involved in
$v_{\bl}\widetilde{T}_{n-1}\widetilde{T}_{n-2}\cdots\widetilde{T}_{2f}$
as well as in ${\iota}^{\otimes
n}\bigl(v_{\bi-m_0+m}T_{\widehat{d}_J}^{\ast})\widetilde{T}_{n-1}\widetilde{T}_{n-2}\cdots\widetilde{T}_{2f}$.
\smallskip

Since the action of
$\widetilde{T}_{n-1}\widetilde{T}_{n-2}\cdots\widetilde{T}_{2f}$
does not change the first $(2f-1)$ positions, it follows from
(\ref{useful}) and the fact $m_0-m+1\leq l_n\leq m_0+m$ that $$
{\iota}^{\otimes
n}\bigl(v_{\bi-m_0+m}T_{\widehat{d}_J}^{\ast})\widetilde{T}_{n-1}\widetilde{T}_{n-2}\cdots\widetilde{T}_{2f}
={\iota}^{\otimes
n}\bigl(v_{\bi-m_0+m}T_{{d}_J}^{\ast})+\sum_{\bu\in
I(2m_0,n)}\widehat{B}'_{\bi,\bu}v_{\bu},
$$
where $\widehat{B}'_{\bi,\bu}\in K$ for each $\bu$, and
$\widehat{B}_{\bi,\bu}\neq 0$ only if $$
\bigl\{u_1,\cdots,u_{2f}\bigr\}\bigcap\bigl\{m_0+m+1,m_0+m+2,\cdots,2m_0\bigr\}=\emptyset,$$
and either $u_1$ or $u_2$ belongs to
$\bigl\{1,2,\cdots,m_0-m\bigr\}.$ By the same reason, we can deduce
that
$$v_{\bl}\widetilde{T}_{n-1}\widetilde{T}_{n-2}\cdots\widetilde{T}_{2f-1}
=\sum_{\bu\in I(2m_0,n)}\widehat{B}''_{\bi,\bu}v_{\bu},
$$
where $\widehat{B}''_{\bi,\bu}\in K$ for each $\bu$, and
$\widehat{B}'_{\bi,\bu}\neq 0$ only if $$
\bigl\{u_1,\cdots,u_{2f}\bigr\}\bigcap\bigl\{m_0+m+1,m_0+m+2,\cdots,2m_0\bigr\}=\emptyset,$$
and either $u_1$ or $u_2$ belongs to
$\bigl\{1,2,\cdots,m_0-m\bigr\}.$ This completes the proof of
(\ref{KeyStep333}), and hence the proof of (\ref{KeyStep3}).
\medskip\smallskip

\noindent {\it Step 3.} We want to show that for any
$d_2\in\mathcal{D}_{\nu_f}$, $$
\varphi_{C}\Theta_1\Bigl(\widetilde{E}_1\widetilde{E}_3\cdots
\widetilde{E}_{2f-1}\widetilde{T}_{\sigma}\widetilde{T}_{d_2}\Bigr)=\Theta_0\varphi_{C}\Bigl(
\widetilde{E}_1\widetilde{E}_3\cdots
\widetilde{E}_{2f-1}\widetilde{T}_{\sigma}\widetilde{T}_{d_2}\Bigr).
$$
As before, it is equivalent to show that
\begin{equation}\label{Keystep4}
\varphi_{C}\Bigl({E}_1{E}_3\cdots
{E}_{2f-1}{T}_{\sigma}{T}_{d_2}\Bigr)=\zeta^{-(m_0+m)n}\Theta_0\varphi_{C}\Bigl(
\widetilde{E}_1\widetilde{E}_3\cdots
\widetilde{E}_{2f-1}\widetilde{T}_{\sigma}\widetilde{T}_{d_2}\Bigr).
\end{equation}

Recall that $\widetilde{V}^{\otimes n}$ has a basis consists of all
the simple $n$-tensors $v_{\bi}$, where $\bi\in I(2m_0,n)$. We
ordered the elements of this basis as $X_1, X_2, \cdots,
X_{(2m_0)^n}$ such that the subset $\bigl\{X_1, X_2, \cdots,
X_{(2m)^n}\bigr\}$ is a basis of $\iota^{\otimes n}\bigl(V^{\otimes
n}\bigr)$. With this ordered basis $\{X_1,\cdots,X_{(2m_0)^n}\}$ in
mind, we identify $\End_K\bigl(\widetilde{V}^{\otimes n}\bigr)$ with
full matrix algebra ${\rm M}_{(2m_0)^n\times (2m_0)^n}(K)$, and
identify the homomorphism $$\varphi_{C}:\,\,
(\bb_n(-\zeta^{2m_0+1},\zeta))^{\rm
op}\rightarrow\End_K\bigl(\widetilde{V}^{\otimes n}\bigr)$$ with  a
homomorphism $$\varphi_{C}:\,\, (\bb_n(-\zeta^{2m_0+1},\zeta))^{\rm
op}\rightarrow{\rm M}_{(2m_0)^n\times (2m_0)^n}(K).$$ We claim that
for any $x\in\bb_n(-\zeta^{2m_0+1},\zeta))$,
\begin{equation}\label{Keystep4a}\varphi_{C}\bigl(x^{\ast}\bigr)=\bigl(\varphi_{C}(x)\bigr)^t,
\end{equation}
where $\bigl(\varphi_{C}(x)\bigr)^t$ means the transpose of the
matrix $\varphi_{C}(x)$, and $``\ast"$ denotes the algebra
anti-automorphism of $\bb_n(-\zeta^{2m_0+1},\zeta))$ defined in
Section 5.\smallskip

In fact, by direct verification, it is easy to see that for each
$1\leq i\leq n-1$, both $\varphi_{C}(T_i)$ and $\varphi_{C}(E_i)$
are symmetric matrices (with respect to the above ordered basis).
Since $\bb_n(-\zeta^{2m_0+1},\zeta)$ is generated by $T_i, E_i,
i=1,2,\cdots,n-1$ and $\varphi_{C}$ is an algebra homomorphism, the
claim (\ref{Keystep4a}) follows immediately.\smallskip

The above argument applies equally well to $V^{\otimes n}$. For each
integer $i$ with $1\leq i\leq (2m)^n$, let $Y_i:=\pi^{\otimes
n}(X_{i})$. Then, $\{Y_1,\cdots,Y_{(2m)^n}\}$ is a basis of
$V^{\otimes n}$. With the ordered basis $\{Y_1,\cdots,Y_{(2m)^n}\}$
in mind, we identify $\End_K\bigl({V}^{\otimes n}\bigr)$ with full
matrix algebra ${\rm M}_{(2m)^n\times (2m)^n}(K)$, and identify the
homomorphism
$$\varphi_{C}:\,\, (\bb_n(-\zeta^{2m+1},\zeta))^{\rm
op}\rightarrow\End_K\bigl({V}^{\otimes n}\bigr)$$ with  a
homomorphism $$\varphi_{C}:\,\, (\bb_n(-\zeta^{2m+1},\zeta))^{\rm
op}\rightarrow{\rm M}_{(2m)^n\times (2m)^n}(K).$$ As before, for any
$x\in\bb_n(-\zeta^{2m_0+1},\zeta))$, we have
$\varphi_{C}\bigl(x^{\ast}\bigr)=\bigl(\varphi_{C}(x)\bigr)^t$.\smallskip

We define a linear map $\Theta'_0$ from ${\rm M}_{(2m_0)^n\times
(2m_0)^n}(K)$ to ${\rm M}_{(2m)^n\times (2m)^n}(K)$ as follows: for
any $M\in{\rm M}_{(2m_0)^n\times (2m_0)^n}(K)$, $\Theta'_0(M)$ is
the submatrix of $M$ in the upper-left corner, obtained from $M$ by
deleting the last $(2m_0)^n-(2m)^n$ rows and the last
$(2m_0)^n-(2m)^n$ columns. Then, it is clear that
\begin{equation}\label{Keystep4b}\Theta'_0\bigl(M^{t}\bigr)=\bigl(\Theta'_0(M)\bigr)^t.
\end{equation}
With the ordered bases $\{X_1, X_2, \cdots, X_{(2m_0)^n}\}$ and
$\{Y_1,\cdots,Y_{(2m)^n}\}$ in mind, it is easy to check that we can
identify the linear map $\Theta_0$ as the restriction of
$\zeta^{(m_0+m)n}\Theta'_0$.

Applying (\ref{Keystep4a}), (\ref{Keystep4b}) and (\ref{KeyStep3}),
we get that
$$\begin{aligned} &\quad\, \zeta^{-(m_0+m)n}\Theta_0\varphi_{C}\Bigl(
\widetilde{E}_1\widetilde{E}_3\cdots
\widetilde{E}_{2f-1}\widetilde{T}_{\sigma}\widetilde{T}_{d_2}\Bigr)\\
&=\Theta'_0 \varphi_{C}\Bigl( \widetilde{E}_1\widetilde{E}_3\cdots
\widetilde{E}_{2f-1}\widetilde{T}_{\sigma}\widetilde{T}_{d_2}\Bigr)\\
&=\Theta'_0 \varphi_{C}\biggl(\Bigl(\widetilde{T}_{d_2}^{\ast}
\widetilde{E}_1\widetilde{E}_3\cdots
\widetilde{E}_{2f-1}\widetilde{T}_{\sigma^{-1}}\Bigr)^{\ast}\biggr)\\
&=\Theta'_0 \biggl(\Bigl(\varphi_{C}\bigl(\widetilde{T}_{d_2}^{\ast}
\widetilde{E}_1\widetilde{E}_3\cdots
\widetilde{E}_{2f-1}\widetilde{T}_{\sigma^{-1}}\bigr)\Bigr)^{t}\biggr)\\
&=\biggl(\Theta'_0\varphi_{C}\Bigl(\widetilde{T}_{d_2}^{\ast}
\widetilde{E}_1\widetilde{E}_3\cdots
\widetilde{E}_{2f-1}\widetilde{T}_{\sigma^{-1}}\Bigr)\biggr)^{t}\\
&=\biggl(\varphi_{C}\Bigl({T}_{d_2}^{\ast}{E}_1{E}_3\cdots
{E}_{2f-1}{T}_{\sigma^{-1}}\Bigr)\biggr)^t\\
&=\varphi_{C}\Bigl(\bigl({T}_{d_2}^{\ast}{E}_1{E}_3\cdots
{E}_{2f-1}{T}_{\sigma^{-1}}\bigr)^{\ast}\Bigr)\\
&=\varphi_{C}\Bigl({E}_1{E}_3\cdots
{E}_{2f-1}{T}_{\sigma}{T}_{d_2}\Bigr),
\end{aligned}
$$
as required. This completes the proof of (\ref{Keystep4}).
\medskip\smallskip

\noindent {\it Step 4.} We are now in a position to prove
(\ref{last}). Let $\bi\in\widehat{I}(2m,n)$,
$\sigma\in\BS_{\{2f+1,\cdots,n\}}$, $\nu:=((2^f), (n-2f))$,
$d_1,d_2\in\mathcal{D}_{\nu}$. It suffices to show that
$$ \begin{aligned}
v_{\bi}\widetilde{T}_{d_1}^{\ast}\widetilde{E}_1\widetilde{E}_3\cdots
\widetilde{E}_{2f-1}\widetilde{T}_{\sigma}\widetilde{T}_{d_2}&=\iota^{\otimes
n}\bigl(v_{\bi-m_0+m}T_{d_1}^{\ast}{E}_1{E}_3\cdots
{E}_{2f-1}{T}_{\sigma}{T}_{d_2}\bigr)\\
&\qquad +\sum_{\bj\in
I(2m_0,n)\setminus\widehat{I}(2m,n)}\widehat{A}'_{\bi,\bj}v_{\bj},
\end{aligned}$$ where $\widehat{A}'_{\bi,\bj}\in K$ for each $\bj$.

By (\ref{KeyStep33}), $$
v_{\bi}\widetilde{T}_{d_1}^{\ast}={\iota}^{\otimes
n}\bigl(v_{\bi-m_0+m}T_{d_1}^{\ast})+\sum_{\substack{\bl\in
I(2m_0,n)\setminus\widehat{I}(2m,n)\\
\ell_s(l_1,\cdots,l_{2f})<f}}\widehat{A}_{\bi,\bl}v_{\bl},
$$
where $\widehat{A}_{\bi,\bl}\in K$ for each $\bl$. Therefore,
$$
v_{\bi}\widetilde{T}_{d_1}^{\ast}\widetilde{E}_1\widetilde{E}_3\cdots
\widetilde{E}_{2f-1}\widetilde{T}_{\sigma}\widetilde{T}_{d_2}=\iota^{\otimes
n}\bigl(v_{\bi-m_0+m}T_{d_1}^{\ast}\bigr)\widetilde{E}_1\widetilde{E}_3\cdots
\widetilde{E}_{2f-1}\widetilde{T}_{\sigma}\widetilde{T}_{d_2}.
$$
Now we use (\ref{Keystep4}), it follows that $$\begin{aligned}
&\quad\,\iota^{\otimes
n}\bigl(v_{\bi-m_0+m}T_{d_1}^{\ast}\bigr)\widetilde{E}_1\widetilde{E}_3\cdots
\widetilde{E}_{2f-1}\widetilde{T}_{\sigma}\widetilde{T}_{d_2}\\
&=\iota^{\otimes n}\bigl(v_{\bi-m_0+m}T_{d_1}^{\ast}{E}_1{E}_3\cdots
{E}_{2f-1}{T}_{\sigma}{T}_{d_2}\bigr)+\sum_{\bj\in
I(2m_0,n)\setminus\widehat{I}(2m,n)}\widehat{A}'_{\bi,\bj}v_{\bj},
\end{aligned}$$ where $\widehat{A}'_{\bi,\bj}\in K$ for each $\bj$, as required.
This completes the proof of (\ref{last}). Hence we complete the
proof of Lemma \ref{lm61} (1).

\bigskip
\bigskip\bigskip\bigskip

\centerline{Acknowledgement}
\bigskip\bigskip

\thanks{The Research was supported by National Natural Science Foundation
of China (Project 10771014), the Program NCET and partly by an
Australian Research Council discovery grant. The author also
acknowledges the support from the Chern Institute of Mathematics
during his visit in the March of 2007.}

\bibliographystyle{amsplain}

\bigskip\bigskip\bigskip\bigskip

\end{document}